\documentclass[envcountresetsect,envcountsame,envcountsect]{svmult}

\usepackage{mathptmx}      
\usepackage{helvet}            
\usepackage{courier}           
\usepackage{graphicx}              
\usepackage{mathrsfs}%
\usepackage{amsmath}%
\usepackage{amsfonts}%
\usepackage{amssymb}%
\smartqed

 { 

 \makeatletter  
 \def\BlackBoxes{\global\overfullrule5\p@}
 \makeatother
 \BlackBoxes 

\predisplaypenalty=10000 

\def\cprime{$'$}
 \newcommand*{\HAbeq}{$$\refstepcounter{equation}}
 \newcommand*{\HAeeq}{\eqno(\theequation)$$}
 \newcommand*{\HAbal}{\begin{aligned}}
 \newcommand*{\HAeal}{\end{aligned}}
 \newcommand*{\HAqa}{\;,\qquad}
 \newcommand*{\HAqb}{\;,\quad}
 \newcommand*{\HAmf}[1]{\mathbf{#1}} 
 \newcommand*{\HAci}{\mathaccent"7017 }  
 \newcommand*{\HAhb}[1]{\hbox{$#1$}} 
 \newcommand*{\HAsdot}{\!\cdot\!}
 \newcommand*{\HAsco}{\colon}
 \newcommand*{\HAsn}{\kern1pt|\kern1pt}
 \newcommand*{\HAbsn}{\kern1pt\big|\kern1pt}
 \newcommand*{\HAssm}{\!\setminus\!}
 \newcommand*{\HAbssm}{\,\big\backslash\,}
 \newcommand*{\HAeea}[1]{\hbox{$[\![#1]\!]$}}
 \newcommand*{\HAvsdot}{\hbox{$|\HAsdot|$}}
 \newcommand*{\HAVsdot}{\hbox{$\|\HAsdot\|$}}
 \newcommand*{\HAhh}[1]{{\textbf{#1}}}
 \newcommand*{\HAnpb}{\postdisplaypenalty=10000}
 \newcommand*{\HAol}{\overline}
 \newsavebox{\HAPrel}
 \sbox{\HAPrel}{\begin{picture}(2,6)(0,0)\put(1,2.5){\circle*{2}}\end{picture}}
 \newcommand*{\HAbtdot}{\mathrel{\usebox{\HAPrel}}}
 \newsavebox{\HAPtop}
 \sbox{\HAPtop}{\begin{picture}(2,2)(-1,-1)\put(0,0){\circle*{2}}\end{picture}}
 \newcommand*{\HAthdot}[3]{\mbox{\HAslap{\kern#3ex\raisebox{#2ex}{\usebox{\HAPtop}}}{$#1$}}}
 \newcommand*{\HAhthdot}{{^{\kern-.1ex\scriptscriptstyle\HAbt}}}
 \newcommand*{\HAthBB}{\rlap{\HAthdot{\HABB}{1.7}{.4}}{\HABB}}
 \newcommand*{\HAthgK}{\HAthdot{\HAgK}{1.7}{.7}}
 \newcommand*{\HAthR}{\rlap{\HAthdot{R}{1.7}{0.7}}{\HAph{R}}}   
 \newcommand*{\HAthsR}{\HAthdot{\HAsR}{1.7}{.4}}
 \newcommand*{\HAthg}{\HAthdot{g}{1.4}{.4}}
 \newcommand*{\HAthia}{\rlap{\HAthdot{\HAia}{1.4}{.3}}\HAph{\HAia}}
 \newcommand*{\HAthka}{\rlap{\HAthdot{\HAka}{1.4}{.6}}{\HAph{\HAka}}}   
 \newcommand*{\HAithka}{{\HAthdot{\scriptstyle\HAka}{1.0}{.4}}}
 \newcommand*{\HAthtau}{\rlap{\HAthdot{\tau}{1.4}{.6}}{\HAph{\tau}}}   
 \newcommand*{\HAthpsi}{\rlap{\HAthdot{\psi}{1.4}{.7}}\HAph{\psi}} 
 \newcommand*{\HAthrho}{\rlap{\HAthdot{\rho}{1.4}{.5}}\HAph{\rho}}
 \newcommand*{\HAph}{\phantom}
 \newcommand*{\HAhr}{\hookrightarrow}
 \newcommand*{\HAlllllora}{\relbar\joinrel\relbar\joinrel\relbar\joinrel\relbar\joinrel\longrightarrow}
 \newcommand*{\HAllllora}{\relbar\joinrel\relbar\joinrel\relbar\joinrel\longrightarrow}
 \newcommand*{\HAlora}{\longrightarrow}
 \newcommand*{\HAra}{\rightarrow}
 \newcommand*{\HAsmtB}{\smash[t]{\vphantom{\Big|}}}
 \newcommand*{\HAapproxS}%
 {\underset{\mbox{\raisebox{1ex}{$\scriptstyle S$}}}\approx}
 \newcommand*{\HAsimS}{\underset{\mbox{\raisebox{1ex}{$\scriptstyle S$}}}\sim}
 \newcommand*{\HAslap}[1]{\hspace{0pt}\hbox to 0pt{#1\hss}}
 \newcommand*{\HAcard}{{\rm card}}
 \newcommand*{\HAdiam}{\mathop{\rm diam}\nolimits}
 \newcommand*{\HAdist}{\mathop{\rm dist}\nolimits}
 \newcommand*{\HAtdiv}{\mathop{\rm div}\nolimits}
 \newcommand*{\HAdom}{{\rm dom}}
 \newcommand*{\HAgrad}{\mathop{\rm grad}\nolimits}
 \newcommand*{\HAid}{{\rm id}}
 \newcommand*{\HAim}{\mathop{\rm im}\nolimits}
 \newcommand*{\HAmult}{\rm mult}
 \newcommand*{\HAsign}{\mathop{\rm sign}\nolimits}
 \newcommand*{\HAloc}{{\rm loc}}
 \newcommand*{\HAis}{\subset}
 \newcommand*{\HAjs}{\supset}
 \newcommand*{\HAbt}{\bullet}
 \newcommand*{\HAes}{\emptyset}
 \newcommand*{\HAiy}{\infty}
 \newcommand*{\HAmt}{\mapsto}
 \newcommand*{\HApl}{\partial}
 \newcommand*{\HApa}{\partial^\alpha}
 \newcommand*{\HAcona}{\kern-1pt}
 \newcommand*{\HAcoU}{\kern-1pt} 
 \newcommand*{\HAcoW}{\kern-1pt}
 \newcommand*{\HAal}{\ualpha}
 \newcommand*{\HAba}{\ubeta}
 \newcommand*{\HADa}{\Delta}
 \newcommand*{\HAda}{\udelta}
 \newcommand*{\HAGa}{\Gamma}
 \newcommand*{\HAga}{\ugamma}
 \newcommand*{\HAia}{\iota}
 \newcommand*{\HAka}{\kappa}
 \newcommand*{\HALda}{\Lambda}
 \newcommand*{\HAlda}{\lambda}
 \newcommand*{\HAna}{\nabla}
 \newcommand*{\HAom}{\omega}
 \newcommand*{\HASa}{\Sigma}
 \newcommand*{\HAsa}{\sigma}
 \newcommand*{\HAta}{\theta}
 \newcommand*{\HAve}{\varepsilon}
 \newcommand*{\HAvp}{\varphi}
 \newcommand*{\HAmfal}{\pmb{\HAal}} 
 \newcommand*{\HAmfba}{\pmb{\HAba}} 
 \newcommand*{\HAmfGa}{\pmb{\HAGa}} 
 \newcommand*{\HAimfGa}{\pmb{\scriptsize{\HAGa}}}   
 \newcommand*{\HAJhJ}{(J,\hat J\,)}
 \newcommand*{\HAMg}{(M,g)}
 \newcommand*{\HAMgsg}{(M,\HAgsg)}
 \newcommand*{\HAcMg}{(\HAcM,g)}
 \newcommand*{\HAhM}{(\hat M)} 
 \newcommand*{\HAhMhg}{(\hat M,\hat g)}
 \newcommand*{\HAtMtg}{(\tilde{M},\tilde{g})}
 \newcommand*{\HAMr}{(M;\rho)}
 \newcommand*{\HAMtr}{(M;\tilde\rho)}
 \newcommand*{\HANg}{(N,g)}
 \newcommand*{\HANh}{(N,h)}
 \newcommand*{\HARmgm}{(\HABR^m,g_m)}
 \newcommand*{\HARngn}{(\HABR^n,g_n)}
 \newcommand*{\HARm}{(\HABR^m)}
 \newcommand*{\HARp}{(\HABR^+)}
 \newcommand*{\HArgK}{(\rho,\HAgK)}
 \newcommand*{\HAtrtK}{(\tilde{\rho},\tilde{\HAgK})}
 \newcommand*{\HAVE}{(V,E)}
 \newcommand*{\HAVg}{(V,g)}
 \newcommand*{\HAVR}{(V,\HABR)}
 \newcommand*{\HAXY}{(X,Y)}
 \newcommand*{\HABB}{{\mathbb B}}
 \newcommand*{\HABH}{{\mathbb H}}
 \newcommand*{\HABN}{{\mathbb N}}
 \newcommand*{\HABR}{{\mathbb R}}
 \newcommand*{\HABS}{{\mathbb S}}
 \newcommand*{\HABX}{{\mathbb X}}
 \newcommand*{\HABZ}{{\mathbb Z}}
 \newcommand*{\HAcA}{{\mathcal A}}
 \newcommand*{\HAcB}{{\mathcal B}}
 \newcommand*{\HAcC}{{\mathcal C}}
 \newcommand*{\HAcD}{{\mathcal D}}
 \newcommand*{\HAcF}{{\mathcal F}}
 \newcommand*{\HAcM}{{\mathcal M}}
 \newcommand*{\HAcR}{{\mathcal R}}
 \newcommand*{\HAcS}{{\mathcal S}}
 \newcommand*{\HAcU}{{\mathcal U}}
 \newcommand*{\HAcV}{{\mathcal V}}
 \newcommand*{\HAcW}{{\mathcal W}}
 \newcommand*{\HAgB}{{\mathfrak B}}
 \newcommand*{\HAgsg}{{\mathfrak g}}
 \newcommand*{\HAgK}{{\mathfrak K}}
 \newcommand*{\HAgL}{{\mathfrak L}}
 \newcommand*{\HAgM}{{\mathfrak M}}
 \newcommand*{\HAgN}{{\mathfrak N}}
 \newcommand*{\HAgP}{{\mathfrak P}}
 \newcommand*{\HAgS}{{\mathfrak S}}
 \newcommand*{\HAsA}{{\mathsf A}}
 \newcommand*{\HAsR}{{\mathsf R}}

 \makeatletter
 \newif\ifinany@
 \newcount\column@
 \def\column@plus{%
    \global\advance\column@\@ne
 }
 \newcount\maxfields@
 \def\add@amps#1{%
    \begingroup
        \count@#1
        \DN@{}%
        \loop
            \ifnum\count@>\column@
                \edef\next@{&\next@}%
                \advance\count@\m@ne
        \repeat
    \@xp\endgroup
    \next@
 }
 \def\Let@{\let\\\math@cr}
 \def\restore@math@cr{\def\math@cr@@@{\cr}}
 \restore@math@cr
 \def\default@tag{\let\tag\dft@tag}
 \default@tag

 \newbox\strutbox@
 \def\strut@{\copy\strutbox@}
 \addto@hook\every@math@size{%
  \global\setbox\strutbox@\hbox{\lower.5\normallineskiplimit
         \vbox{\kern-\normallineskiplimit\copy\strutbox}}}

 \renewcommand{\start@aligned}[2]{%
    \RIfM@\else
        \nonmatherr@{\begin{\@currenvir}}%
    \fi
    \null\,%
    \if #1t\vtop \else \if#1b \vbox \else \vcenter \fi \fi \bgroup
        \maxfields@#2\relax
        \ifnum\maxfields@>\m@ne
            \multiply\maxfields@\tw@
            \let\math@cr@@@\math@cr@@@alignedat
        \else
            \restore@math@cr
        \fi
        \Let@
        \default@tag
        \ifinany@\else\openup\jot\fi
        \column@\z@
        \ialign\bgroup
           &\column@plus
            \hfil
            \strut@
            $\m@th\displaystyle{##}$%
           &\column@plus
            $\m@th\displaystyle{{}##}$%
            \hfil
            \crcr
 }
 \renewenvironment{aligned}[1][c]{%
    \start@aligned{#1}\m@ne
 }{%
    \crcr\egroup\egroup
 }
 \makeatother

 \makeatletter
 \spn@wtheorem{HAexamples}{Examples}{\bfseries}{\rmfamily}
 \spn@wtheorem{HAremark}{Remark}{\bfseries}{\rmfamily}
 \spnewtheorem*{HAproofs}{Proofs}{\itshape}{\rmfamily}
 \spnewtheorem*{HAproofof1.2}{Proof of Theorem~1.2}{\bfseries}{\rmfamily} 
 \spnewtheorem*{HAproofof1.6,1.8}{Proof of Theorem~1.6 and 
 Proposition 1.8}{\bfseries}{\rmfamily} 
 \spnewtheorem*{HAproofof1.9}{Proof of Theorem~1.9}{\bfseries}{\rmfamily} 
 \@addtoreset{equation}{section} 
 \def\ds@numart{\@numarttrue                
  \@takefromreset{figure}{chapter}%
  \@takefromreset{table}{chapter}%
  \@takefromreset{equation}{chapter}%
  \def\thesection{\@arabic\c@section}%
  \def\thefigure{\@arabic\c@figure}%
  \def\thetable{\@arabic\c@table}%
  \def\theequation{\thesection.\arabic{equation}} 
  \def\thesubequation{\arabic{equation}\alph{subequation}}}
 \makeatother

\begin{document}
\title*{Uniformly Regular and Singular Riemannian Manifolds}
\author{Herbert Amann}
\institute{Math.\ Institut Universit\"at Z\"urich, 
Winterthurerstr.~190, CH-8057 Z\"urich\\ 
\email{herbert.amann@math.uzh.ch}}
\maketitle 

\abstract{A detailed study of uniformly regular Riemannian manifolds and 
manifolds with singular ends is carried out in this paper. Such classes of 
manifolds are of fundamental importance for a Sobolev space solution theory 
for parabolic evolution equations on non-compact Riemannian manifolds with 
and without boundary. Besides pointing out this connection in some detail 
we present large families of uniformly regular and singular manifolds 
which are admissible for this analysis.} 
\section{Introduction}\label{HAsec-I}%
The principal object of our concern is an in-depth study of evolution 
equations on non-compact Riemannian manifolds. We are particularly 
interested in establishing an optimal local existence theory for 
quasilinear parabolic initial boundary value problems in a Sobolev space 
framework. For this we need, in the first instance, a~good understanding of 
fractional order \hbox{$L_p$-Sobolev} spaces, including sharp embedding 
and trace theorems, etc. Although fractional order Sobolev spaces can be 
invariantly defined on any non-compact Riemannian manifold, it is 
not possible to establish embedding and trace theorems in this 
generality. For these to hold one has to impose restrictions on the 
underlying manifold near infinity. 

\par 
In our paper~\cite{Ama12b} we have introduced the class of uniformly 
regular Riemannian manifolds and shown, in particular, that fractional 
order Sobolev spaces on such manifolds possess all the properties alluded to 
above. (Also see~\cite{Ama12c} for complements and extensions to 
anisotropic settings.) This class encompasses the 
well-studied case of complete Riemannian manifolds without boundary and 
bounded geometry. Of course, in the study of boundary value problems 
manifolds with boundary are indispensable. In our previous papers 
\cite{Ama12c}, \cite{Ama12b}, and~\cite{Ama14a} we have 
presented examples of manifolds with boundary which are uniformly 
regular. Yet proofs have not been included. The reason being that it needs 
quite a~bit of argumentation to establish these claims. It is the purpose 
of this paper to close this gap and carry out a detailed study of uniformly 
regular Riemannian manifolds. Some of the main results and their 
ramifications are explained in the following. 

\par 
Let $\HAMg$ be a smooth \hbox{$m$-dimensional} Riemannian manifold 
with boundary (which may be empty). Unless explicitly stated otherwise, 
\HAhb{m\in\HABN^\times:=\HABN\HAssm\{0\}}. 
An atlas~$\HAgK$ for~$M$ is said to be uniformly regular if it 
consists of normalized charts\footnote{Precise definitions of all concepts 
used in this introduction without further explanation are found in the 
main body of this paper---in Section~\ref{HAsec-N}, in particular.}, has 
finite multiplicity, all coordinate changes have uniformly bounded 
derivatives of all orders, and if it is shrinkable. By the latter we mean 
that there is a uniform shrinking of all chart domains such that the result 
is an atlas as well. The normalization of the local charts also means that 
they are well adapted to the boundary in a natural precise sense. The 
shrinkability assumption is the most restrictive one. For example, 
the open unit ball in~$\HABR^m$, endowed with the Euclidean metric~%
\HAhb{|\D x|^2}, does not possess a uniformly regular atlas. 

\par 
Let $M$ be equipped with a uniformly regular atlas~$\HAgK$. The metric~$g$ 
is called uniformly regular if its local representation~$\HAka_*g$ is 
equivalent to the Euclidean metric of~$\HABR^m$ and has bounded 
derivatives of all orders, uniformly with respect to 
\HAhb{\HAka\in\HAgK}. Then $\HAMg$ is said to be \emph{uniformly regular} 
if it 
possesses a uniformly regular atlas~$\HAgK$ and $g$~is uniformly regular. 
Loosely speaking, this means that $M$~has an atlas whose coordinate patches 
are all `of approximately the same size'.  The concept of uniform regularity 
is independent of the particular choice of the atlas~$\HAgK$ in a natural 
sense (made precise in (\ref{HAN.Keq})). Fortunately, in practice a~specific 
atlas is rarely needed. It suffices to know that there exists one. 

\par 
We denote by~$c$ constants 
\HAhb{\geq1} whose actual value may vary from occurrence to occurrence; but 
$c$~is always independent of the free variables in a given formula, unless 
a dependence is explicitly indicated. 

\par 
On the set of all nonnegative functions, defined on some nonempty 
set~$S$, whose specific realization will be clear in any given situation, we 
introduce an equivalence relation 
\HAhb{{}\sim{}} by writing 
\HAhb{f\sim g} iff there exists~$c$ such that 
\HAhb{f/c\leq g\leq cf}. Here inequalities between symmetric 
bilinear forms are understood as inequalities between the corresponding 
polar forms. By~$\HAmf{1}$, more precisely~$\HAmf{1}_S$, we denote 
the constant function 
\HAhb{S\HAra\HABR}, 
\ \HAhb{s\HAmt1}. 

\par 
Now we present some examples to illustrate the extent of the concept 
of uniform regularity. 
\begin{HAexamples}\label{HAexa-I.ex} 
\ (a) 
\HAhb{\HABR^m=(\HABR^m,|\D x|^2)} and closed half-spaces thereof are 
uniformly regular. 

\par 
(b) 
Every compact Riemannian manifold is uniformly regular. 

\par 
(c) 
If $\HAMg$ is a Riemannian submanifold with compact boundary of a uniformly 
regular Riemannian manifold, then it is uniformly regular also. 

\par 
(d) 
Complete Riemannian manifolds without boundary and bounded geometry are 
uniformly regular. 

\par 
(e) 
Products of uniformly regular manifolds are uniformly regular. 

\par 
(f) If $(M_1,g_1)$ and $(M_2,g_2)$ are isometric, then $(M_1,g_1)$ is 
uniformly regular iff $(M_2,g_2)$ is so. 
\end{HAexamples}
\begin{HAproofs} 
For~(a) see (\ref{HAP.RH}). Statements \hbox{(b)--(d)} are proved in 
Section~\ref{HAsec-U}. Assertion~(e) is a particular instance of 
Theorem~\ref{HAthm-P.MM}. Claim~(f) follows from Lemma~\ref{HAlem-P.f}. 
\qed  
\end{HAproofs} 
There are also Riemannian manifolds with singular ends which are uniformly 
regular. To explain this in more detail we need some preparation. We fix 
\HAhb{d\geq m-1} and suppose that $(B,g_B)$ is an 
\HAhb{(m-1)}-dimensional Riemannian submanifold of~$\HABR^d$. Then, given 
\HAhb{\HAal\geq0}, 
\HAbeq \label{HAI.F} 
\bigl\{\,(t,t^\HAal y)\ ;\ t>1,\ y\in B\,\bigr\} 
\HAis\HABR\times\HABR^d=\HABR^{1+d} 
\HAeeq  
is the \emph{infinite model \hbox{$\HAal$-funnel} over}~$B$. We denote 
it by~$F_\HAal(B)$. It is an infinite cylinder if 
\HAhb{\HAal=0}, and an infinite (blunt) cone if 
\HAhb{\HAal=1}. Note that $F_\HAal(B)$ is an \hbox{$m$-dimensional} 
submanifold of~$\HABR^{1+d}$. If 
\HAhb{\HApl B\neq\HAes}, then 
\HAhb{\HApl F_\HAal(B)=F_\HAal(\HApl B)}. 

\par 
\vskip-.8\baselineskip
\noindent 
\begin{minipage}[c]{135pt} 
\vspace*{1\baselineskip} 
\hspace{4mm} 
In Fig.~1 there is depicted part of a (rotated) three-dimensional model 
funnel $F_{1/2}(B)$ with a compact base~$B$ having two connected components 
and three boundary components.  
\end{minipage}
\hfil 
\begin{minipage}[c]{105pt} 
\[
\begin{picture}(101,75)(0,0)
\put(0,-2){\makebox(0,0)[bl]%
{\includegraphics[height=75pt]{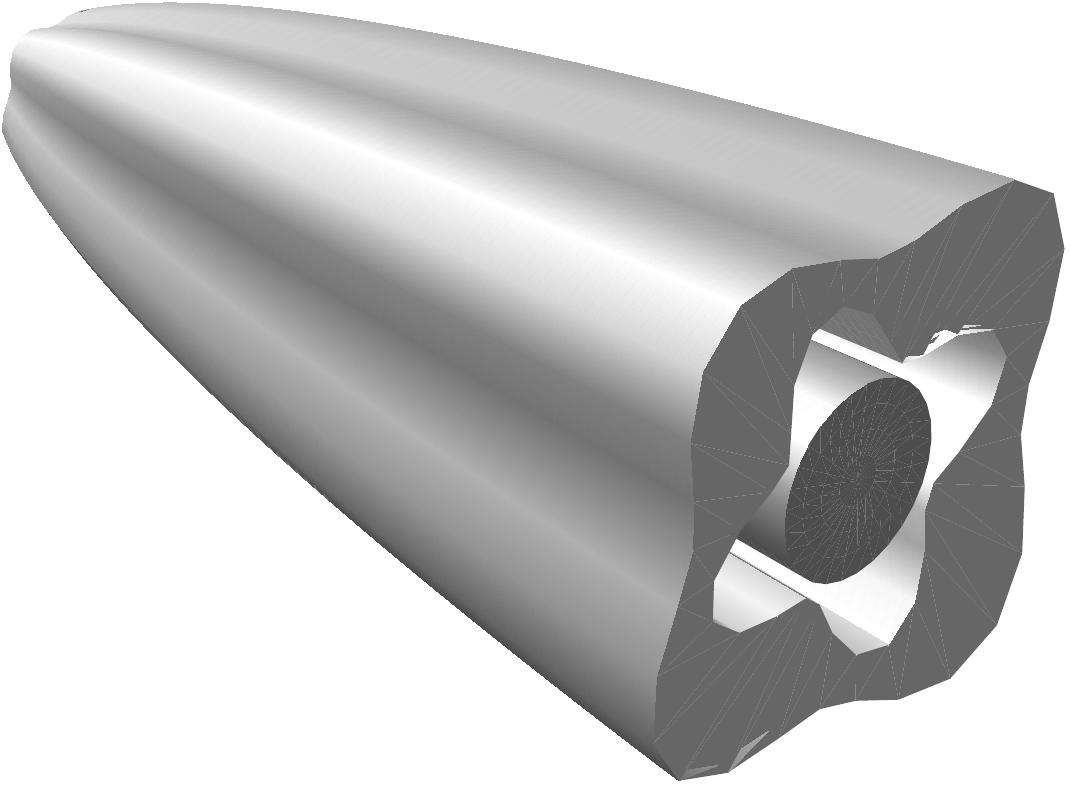}}} 
\put(0,0){\makebox(0,0)[bl]{\HAhh{\small Fig.~1}}}
\end{picture}
\] 
\end{minipage}
\hfil
\vskip.3\baselineskip

\par 
For 
\HAhb{0\leq\HAal\leq1} we endow~$F_\HAal(B)$ with the metric~$g_{F_\HAal(B)}$ 
induced by the embedding 
\HAhb{F_\HAal(B)\HAhr\HABR^{1+d}} so that 
\HAhb{\bigl(F_\HAal(B),g_{F_\HAal(B)}\bigr)} is a Riemannian submanifold of 
\HAhb{(\HABR^{1+d},|\D x|^2)}. An open subset~$V$ of~$M$ is an 
\emph{infinite \hbox{$\HAal$-funnel} over}~$B$ if $\HAVg$ is isometric to 
$\bigl(F_\HAal(B),g_{F_{\HAal(B)}}\bigr)$. It is~a \emph{tame end 
of}~$\HAMg$ if it is an infinite \hbox{$\HAal$-funnel} with 
$\HAal$~belonging to~$[0,1]$ and $B$~being compact. 

\par  
Suppose 
\HAhb{\{V_0,V_1,\ldots,V_k\}} is a finite open covering of~$M$ with 
\HAhb{V_i\cap V_j=\HAes} for 
\HAhb{1\leq i<j\leq k} such that $V_i$~is a tame end for 
\HAhb{1\leq i\leq k} and 
\HAhb{V_0,\ V_0\cap V_1,\ldots,\ V_0\cap V_k} are relatively compact 
in~$M$. Then $\HAMg$ is said to have (finitely many) \emph{tame ends}. 
\begin{theorem}\label{HAthm-I.T} 
If $\HAMg$ has tame ends, then it is uniformly regular. 
\end{theorem}
\begin{proof} 
Section~\ref{HAsec-M}. 
\qed 
\end{proof} 
As mentioned above, our motivation for the study of uniformly regular 
Riemannian manifolds stems from the theory of parabolic equations. To 
explain their role in the present environment we consider a simple model 
problem. We set 
\HAbeq \label{HAI.A.d} 
\HAcA u:=-\HAtdiv(a\HAbtdot\HAgrad u)\;, 
\HAeeq 
with $a$~being a symmetric positive definite 
$(1,1)$-tensor field on~$\HAMg$ which is bounded and has 
bounded and continuous first order (covariant) derivatives. This is 
expressed by saying that $\HAcA$~is a regular uniformly strongly elliptic 
differential operator. This low regularity assumption for~$a$ is of 
basic importance for treating quasilinear problems in which 
$a$~depends on~$u$. 
 
\par 
We assume that $\HApl_0M$~is open and closed in~$\HApl M$ and 
\HAhb{\HApl_1M:=\HApl M\HAssm\HApl_0M}. Then we put 
$$ 
\HAcB_0u:=u\text{ on }\HApl_0M 
\HAqb \HAcB_1u:=(\unu\HAsn a\HAbtdot\HAgrad u)\text{ on }\HApl_1M, 
$$ 
where these operators are understood in the sense of traces and $\unu$~is 
the inward pointing unit normal vector field on~$\HApl_1M$. Thus 
\HAhb{\HAcB:=(\HAcB_0,\HAcB_1)} is the Dirichlet boundary operator 
on~$\HApl_0M$ and the Neumann operator on~$\HApl_1M$. 
 
\par 
Suppose 
\HAhb{0<T<\HAiy}. We write 
\HAhb{M_T:=M\times[0,T]} for the space time cylinder. Moreover, 
\HAhb{\HApl=\HApl_t} is the `time derivative', 
\HAhb{\HASa_T:=\HApl M\times[0,T]} the lateral boundary, and 
\HAhb{M_0=M\times\{0\}} the `initial surface' of~$M_T$. Then we consider the 
problem  
\HAbeq \label{HAI.A.P} 
\HApl u+\HAcA u=f\text{ on }M_T 
\HAqb \HAcB u=0\text{ on }\HASa_T
\HAqb u=u_0\text{ on }M_0. 
\HAnpb 
\HAeeq 
The last equation is to be understood as 
\HAhb{\HAga_0u=u_0} with the `initial trace' operator~$\HAga_0$\;. 

\par 
Of course, $\HApl_0M$ or $\HApl_1M$ or both may be empty. In such a situation 
obvious interpretations and modifications are to be applied. 
 
\par 
We are interested in an optimal \hbox{$L_p$-theory} for (\ref{HAI.A.P}). To 
describe it we have to introduce Sobolev-Slobodeckii spaces. We 
always assume that 
\HAhb{1<p<\HAiy}. The Sobolev space~$W_{\HAcoW p}^k(M)$ is then defined for 
\HAhb{k\in\HABN} to be the 
completion of~$\HAcD(M)$, the space of smooth functions with compact support, 
in $L_{1,\HAloc}(M)$ with respect to the norm 
\HAbeq \label{HAI.A.N}  
u\HAmt 
\Bigl(\sum_{j=0}^k 
\big\|\,|\HAna^ju|_{g_0^j}\,\big\|_{L_p(M)}^p\Bigr)^{1/p}\;. 
\HAeeq 
Here 
\HAhb{\HAna=\HAna_{\HAcona g}} is the Levi-Civita covariant derivative and 
\HAhb{\HAvsdot_{g_0^j}}~the 
$(0,j)$-tensor norm naturally induced by~$g$. Thus 
\HAhb{W_{\HAcoW p}^0(M)=L_p(M)}. If 
\HAhb{s\in\HABR^+\HAssm\HABN}, then the Slobodeckii space~$W_{\HAcoW p}^s(M)$ 
is defined by real interpolation: 
$$ 
W_{\HAcoW p}^s(M) 
:=\bigl(W_{\HAcoW p}^k(M),W_{\HAcoW p}^{k+1}(M)\bigr)_{s-k,p} 
\HAqa k<s<k+1 
\HAqb k\in\HABN\;. 
$$ 

\par 
Although these definitions are meaningful on any Riemannian manifold, they 
are not too useful in such a general setting since they may lack basic 
Sobolev type embedding properties, for example. The situation is different 
if we restrict ourselves to uniformly regular Riemannian manifolds. The 
following theorem is a consequence of the results of our paper~\cite{Ama12b} 
to which we direct the reader for details, proofs, and many more facts.  
\begin{theorem}\label{HAthm-I.S} 
Let $\HAMg$ be uniformly regular. Then the Sobolev-Slobodeckii 
spaces $W_{\HAcoW p}^s(M)$, 
\ \HAhb{s\geq0}, possess the same embedding, interpolation, and trace 
properties as in the classical Euclidean case. They can be characterized 
by means of local coordinates. 
\end{theorem} 
We denote by~$W_{\HAcoW p,\HAcB}^s(M)$ the closed linear subspace of all 
\HAhb{u\in W_{\HAcoW p}^s(M)} satisfying 
\HAhb{\HAcB u=0} whenever $s$~is such that this condition is well-defined 
(cf. \cite[(1.4)]{Ama14a}). 
 
\par 
We set 
$$ 
A:=\HAcA\HAsn W_{\HAcoW p,\HAcB}^2(M)\;, 
$$ 
considered as an unbounded linear operator in~$L_p(M)$ with 
domain~$W_{\HAcoW p,\HAcB}^2(M)$. Then (\ref{HAI.A.P}) can be expressed 
as an initial value problem for the evolution equation 
$$ 
\dot u+Au=f\text{ on }[0,T] 
\HAqb u(0)=u_0 
\HAnpb 
$$ 
in~$L_p(M)$. 
 
\par 
Now we are ready to formulate the basic well-posedness result in the present 
model setting. 
\begin{theorem}\label{HAthm-I.A.P} 
Let $\HAMg$ be uniformly regular and 
\HAhb{p\notin\{3/2,\ 3\}}. Suppose that $\HAcA$~is regularly uniformly 
strongly elliptic. Then (\ref{HAI.A.P}) has for each 
$$ 
(f,u_0)\in L_p\bigl([0,T],L_p(M)\bigr)\times W_{\HAcoW p,\HAcB}^{2-2/p}(M) 
$$ 
a unique solution 
$$ 
u\in L_p\bigl([0,T],W_{\HAcoW p,\HAcB}^2(M)\bigr) 
\cap W_{\HAcoW p}^1\bigl([0,T],L_p(M)\bigr).  
\HAnpb 
$$ 
The map 
\HAhb{(f,u_0)\HAmt u} is linear and continuous.\pagebreak[2]\\
Equivalently: 
$-A$~generates an analytic semigroup on $L_p(M)$ and 
has the property of maximal regularity. 
\end{theorem} 
For this theorem we refer to~\cite{Ama14a} where 
non-homogeneous boundary conditions and lower order terms 
are treated as well and further references are given. 
Analogous theorems apply to higher order problems and parabolic equations 
operating on sections of uniformly regular vector bundles. 

\par 
On the surface, Theorem~\ref{HAthm-I.A.P} looks exactly the same as the very 
classical existence and uniqueness theorem for second order parabolic 
equations on open subsets of~$\HABR^m$ with smooth compact boundary 
(e.g., O.A. Ladyzhenskaya, V.A. Solonnikov, and 
N.N. Ural'ceva~\cite[Chapter~IV]{LSU68a} and 
R.~Denk, M.~Hieber, and J.~Pr{\"u}ss~\cite{DHP03a}). However, it is, in fact, 
a rather deep-rooted vast generalization thereof since it applies to any 
uniformly regular Riemannian manifold. 

\par 
In this connection we have to mention the work of 
G.~Grubb~\cite{Gru95b} who established a general 
\hbox{$L_p$~theory} for parabolic pseudo-differential boundary value 
problems (also see Section~IV.4.1 
in~\cite{Gru96a}). It applies to a class of noncompact manifolds, called 
`admissible', introduced in G.~Grubb and 
N.J.~Kokholm~\cite{GrK93a}. It is a subclass of the above 
manifolds with tame ends, namely a family of manifolds with 
conical ends. Earlier investigations of pseudo-differential operators on 
manifolds with conical ends are due to E.~Schrohe~\cite{Schro87a} who employs 
weighted Sobolev spaces. 

\par  
Recently, a~maximal regularity theory for parabolic differential 
equations on Riemannian manifolds without boundary and cylindrical ends
has been presented by Th.~Krainer~\cite{Krai09a}. This author uses a 
compactification technique to `reduce' the problem to a compact 
Riemannian manifold $\HAtMtg$, 
where $\tilde g$~is the cusp metric 
\HAhb{\D t^2/t^4+g_Y} in a collar neighborhood of the boundary~$Y$ 
of~$\tilde M$. Then the theory of cusp pseudo-differential operators 
is applied in conjunction with the general \hbox{$\HAcR$-boundedness} theory 
of maximal regularity for parabolic evolution equations. The final result 
is then formulated in the Sobolev space setting for $\HAtMtg$ which involves 
rather complicated weighted norms. In contrast, our result 
Theorem~\ref{HAthm-I.A.P} uses the Sobolev space setting of~$\HAMg$ only. 
Due to Theorem~\ref{HAthm-I.T}, it applies to manifolds with cylindrical 
ends, in particular. 

\par 
There is a tremendous amount of literature on heat equations on complete 
Riemannian manifolds without boundary and bounded geometry. Most of it 
concerns heat kernel estimates and spectral theory (see, for example, 
E.B.~Davies~\cite{Dav89a} or A.~Grigor'yan~\cite{Gri09a} and the references 
therein). By imposing further structural conditions, as the assumption of 
non-negative Ricci curvature, for instance, heat kernel estimates lead to 
maximal regularity results for the Laplace-Beltrami operator (e.g.\ 
M.~Hieber and J.~Pr{\"u}ss~\cite{HP97a}, 
A.L.~Mazzucato and V.~Nistor~\cite{MaN06a}. Also see 
A.~Grigor'yan and L.~Saloff-Coste~\cite{GriSal09a} and 
L.~Saloff-Coste~\cite{Sal10a}). Due to Example~\ref{HAexa-I.ex}(d) our 
Theorem~\ref{HAthm-I.A.P} applies in this setting without any additional 
restriction on the geometry of $\HAMg$. 

\par 
Let now $\HAMg$ be a Riemannian manifold which is not uniformly regular. 
It is said to be \emph{singular of type}~$\rho$ if 
\HAhb{\rho\in C^\HAiy\bigl(M,(0,\HAiy)\bigr)} and 
\HAhb{(M,g/\rho^2)} is uniformly regular. Any such~$\rho$ is~a 
\emph{singularity function} for~$\HAMg$. We assume that $\rho$~is bounded 
from above. Then 
\HAhb{\inf\rho=0} and $\HAMg$~is said to be \emph{singular near} 
\HAhb{\rho=0}. In order for 
\HAhb{\rho\in C^\HAiy\bigl(M,(0,\HAiy)\bigr)} to qualify as a singularity 
function it has to satisfy structural conditions naturally associated 
with~$\HAMg$ (see~(\ref{HAN.sd})). Below we describe a large class of 
singularity functions which are closely related to the geometric structure 
near the `singular ends' of~$\HAMg$, that is, the behavior of~$\HAMg$ 
`near infinity'. 

\par 
Suppose $\HAMg$~is singular of type~$\rho$. We set 
\HAhb{\hat g:=g/\rho^2} and 
\HAhb{\HAhMhg:=(M,g/\rho^2)}. Then we can apply the preceding results to the 
uniformly regular Riemannian manifold~$\HAhMhg$. Since $\hat g$~is 
conformally equivalent to~$g$ (and $\rho$~satisfies appropriate structural 
conditions) we can express the Sobolev-Slobodeckii 
spaces~$W_{\HAcoW p}^s\HAhM$, 
which are constructed by means of~$\HAna_{\HAcona\HAcona\hat g}$, in terms of 
weighted Sobolev-Slobodeckii spaces on~$M$. More precisely, we define 
$W_{\HAcoW p}^{k,\HAlda}\HAMr$ for 
\HAhb{k\in\HABN} and 
\HAhb{\HAlda\in\HABR} by replacing (\ref{HAI.A.N}) in the definition 
of~$W_{\HAcoW p}^k(M)$ by 
$$ 
u\HAmt\Bigl(\sum_{j=0}^k 
\big\|\rho^{\HAlda+j}\,|\HAna^ju|_{g_0^j}\big\|_{L_p(M)}^p\Bigr)^{1/p}\;. 
$$ 
Furthermore, 
$$ 
W_{\HAcoW p}^{s,\HAlda}\HAMr 
:=\bigl(W_{\HAcoW p}^{k,\HAlda}\HAMr,W_{\HAcoW p}^{k+1,\HAlda}\HAMr 
\bigr)_{s-k,p} 
\HAqa k<s<k+1 
\HAqb k\in\HABN\;, 
$$ 
and 
\HAhb{L_p^\HAlda\HAMr:=W_{\HAcoW p}^{0,\HAlda}\HAMr}. Then, 
see~\cite{Amaxx1}, 
$$ 
W_{\HAcoW p}^s\HAhM\doteq W_{\HAcoW p}^{s,-m/p}\HAMr 
\HAqa s\geq0\;, 
$$ 
where 
\HAhb{{}\doteq{}}~means: equal except for equivalent norms. In~\cite{Amaxx1} 
it is also shown that 
$$ 
W_{\HAcoW p}^s\HAhM\HAra W_{\HAcoW p}^{s,\HAlda}\HAMr 
\HAqb u\HAmt\rho^{-\HAlda+m/p}u 
$$ 
is an isomorphism. With its help we can transfer all properties enjoyed by 
the Sobolev-Slobodeckii spaces~$W_{\HAcoW p}^s\HAhM$ to the weighted 
spaces~$W_{\HAcoW p}^{s,\HAlda}\HAMr$ (direct proofs, not 
using this isomorphism, are given in~\cite{Ama12b}). 

\par 
There are also simple relations between the differential operators 
$\HAtdiv$  and~$\HAgrad$ on~$\HAMg$ 
and
$\HAtdiv_{\hat g}$  and~$\HAgrad_{\hat g}$ on~$\HAhMhg$, respectively. 
In fact, setting 
\HAhb{\hat a:=\rho^{-2}a} we find (cf.\ \cite[(5.19)]{Ama14a}) 
$$ 
\HAtdiv(a\HAgrad u) 
=\HAtdiv_{\hat g}(\hat a\cdot\HAgrad_{\hat g}u) 
+(u\hat a\cdot\rho^{-1}\HAgrad_{\hat g} 
\rho\HAsn\HAgrad_{\hat g}u)_{\hat g}\;. 
$$ 
Note that Theorem~\ref{HAthm-I.A.P} applies to the operator 
$$ 
\hat\HAcA u:=-\HAtdiv_{\hat g}(\hat a\cdot\HAgrad_{\hat g}u) 
$$ 
provided it is regularly uniformly strongly elliptic on~$\HAhMhg$. This is 
equivalent to the assumption that (\ref{HAI.A.d}) be \emph{regularly 
uniformly strongly \hbox{$\rho$-elliptic}}. By this we mean that the 
following conditions are satisfied: 
$$ 
\HAbal 
{}
\rm{(i)} \qquad &\bigl(a(q)\cdot X\HAbsn X\bigr)_{g(q)} 
                 \sim\rho^2(q)\,|X|_{g(q)}^2, 
                 \ \ X\in T_qM,
                 \ \ q\in M\;.\\ 
\noalign{\vskip-1\jot} 
\rm{(ii)}\qquad &|\HAna a|_{g_1^2}\leq c\rho\;.
\HAeal 
$$ 
An elaboration of these facts leads to the following optimal 
well-posedness result for degenerate parabolic equations on singular 
manifolds. It is a special case of Theorem~5.2 of~\cite{Ama14a}. 
\begin{theorem}\label{HAthm-I.D} 
Let $\HAMg$ be singular of type~$\rho$ and 
\HAhb{p\notin\{3/2,\ 3\}}. Suppose $\HAcA$~is regularly uniformly strongly 
\hbox{$\rho$-elliptic} and 
\HAhb{\HAlda\in\HABR}. Then problem~(\ref{HAI.A.P}) has for each 
$$ 
(f,u_0)\in 
L_p\bigl([0,T],L_p^\HAlda\HAMr\bigr) 
\times W_{\HAcoW p,\HAcB}^{2-2/p,\HAlda}\HAMr 
$$ 
a unique solution 
$$ 
u\in 
L_p\bigl([0,T],W_{\HAcoW p,\HAcB}^{2,\HAlda}\HAMr\bigr) 
\cap W_{\HAcoW p}^1\bigl([0,T],L_p^\HAlda\HAMr\bigr)\;. 
\HAnpb 
$$ 
The map 
\HAhb{(f,u_0)\HAmt u} is linear and continuous. 

\par 
Equivalently: let 
$$ 
A^\HAlda:=\HAcA\HAsn W_{\HAcoW p,\HAcB}^{2,\HAlda}\HAMr\;. 
$$ 
Then $-A^\HAlda$~generates a strongly continuous analytic semigroup 
on~$L_p^\HAlda\HAMr$ and has the property of maximal regularity. 
\end{theorem} 
In order to render this theorem useful we have to provide sufficiently large 
and interesting classes of singular manifolds. This is the aim of the 
following considerations. 

\par 
Let $(B,g_B)$ be as in definition~(\ref{HAI.F}). If we choose there 
\HAhb{\HAal<0}, then we call the resulting Riemannian submanifold 
of~$\HABR^{1+d}$ \emph{infinite model \hbox{$\HAal$-cusp} over}~$B$ and 
denote it by~$C_{\HAiy,\HAal}(B)$ and its metric 
by~$g_{C_{\HAiy,\HAal}(B)}$. Similarly as for funnels, an open 
subset~$V$ of~$M$ is an \emph{infinite \hbox{$\HAal$-cusp} over}~$B$ 
of~$\HAMg$ if $\HAVg$ is isometric to an infinite model 
\hbox{$\HAal$-cusp} 
\HAhb{\bigl(C_{\HAiy,\HAal}(B),g_{C_{\HAiy,\HAal}(B)}\bigr)}. 
It is \emph{smooth} if $B$~is compact. 

\par 
We consider the following conditions: 
\HAbeq \label{HAI.MM} 
 \HAbal 
 \\ 
 \noalign{\vskip-7\jot} 
 \rm{(i)} \qquad     &\HAcMg
                      \text{ is an $m$-dimensional Riemannian manifold}\;.\\
 \noalign{\vskip-1\jot} 
 \rm{(ii)} \qquad    &\HAmf{V}\text{ is a finite set of pairwise disjoint 
                     infinite smooth $\HAal$-cusps }V\\ 
 \noalign{\vskip-1\jot} 
                     &\text{of }\HAcMg\;.\\                      
 \noalign{\vskip-1\jot} 
 \rm{(iii)} \qquad   &V_0\text{ is an open subset of $\HAcM$ such that 
                      $\{V_0\}\cup\HAmf{V}$ is a covering of }\HAcM\\ 
 \noalign{\vskip-1\jot} 
                     &\text{and $\HAcMg$ is uniformly regular 
                      on\footnotemark\ }V_0\;.\\ 
 \noalign{\vskip-1\jot} 
 \rm{(iv)} \qquad    &\HAmfGa\text{ is a finite set of pairwise 
                      nonintersecting compact connected}\\ 
 \noalign{\vskip-1\jot} 
                     &\text{Riemannian submanifolds $\HAGa$ of $\HAcM$ 
                      without boundary and}\\ 
 \noalign{\vskip-1\jot} 
                     &\text{codimension at least $1$ such that 
                      $\HAGa\HAis\HApl\HAcM$ if }
                      \HAGa\cap\HApl\HAcM\neq\HAes\\ 
 \noalign{\vskip-1\jot} 
                     &\text{and $\HAGa\cap V=\HAes$ for $\HAGa\in\HAmfGa$ 
                      and }V\in\HAmf{V}\;.\\
 \noalign{\vskip-1\jot} 
 \rm{(v)} \qquad     &\HAba_\HAGa\geq1\text{ for }\HAGa\in\HAmfGa. 
 \HAeal
\HAeeq\footnotetext{Cf.\ the localized definitions in 
                      Section~\ref{HAsec-N}.}%
We set 
$$ 
\HAmfal:=\{\,\HAal_V\ ;\ V\in\HAmf{V}\,\} 
\HAqb \HAmfba:=\{\HAmf\HAba_\HAGa\ ;\ \HAGa\in\HAmfGa\,\} 
\HAqb \HAcS:={\textstyle\bigcup_{\HAGa\in\HAimfGa}}\HAGa\;, 
$$ 
and 
$$ 
\HAMg:=\bigl(\HAcM\HAssm\HAcS,\ g\HAsn(\HAcM\HAssm\HAcS)\bigr)\;. 
$$ 
Then $\HAmfal$, resp.~$\HAmfba$, is the \emph{cuspidal weight} (vector) 
for~$\HAmf{V}$, resp.~$\HAmfGa$, and $\HAcS$ the (\emph{compact}) 
\emph{singularity set} of~$M$. Furthermore, $\HAMg$ is said to be a 
\emph{Riemannian manifold with smooth cuspidal singularities of type} 
$[\HAmf{V},\HAmfal,\HAmfGa,\HAmfba]$. 

\par 
If 
\HAhb{\HAmf{V}=\HAes} and 
\HAhb{\HAmfGa=\HAes}, then 
\HAhb{\HAcM=V_0} and 
\HAhb{\HAcMg=\HAMg} is uniformly regular. Thus we assume henceforth that 
\HAhb{\HAmf{V}\cup\HAmfGa\neq\HAes}. If 
\HAhb{\HAmf{V}=\HAes}, then 
\HAhb{V_0=\HAcM} and $\HAcMg$ is uniformly regular. By the preceding results 
this is the case, in particular, if $\HAcM$~is compact or $\HAcMg$~has 
tame ends. However, $\HAMg$~is not uniformly regular. 

\par 
By its definition, $\HAMg$~is a Riemannian submanifold of the ambient 
manifold~$\HAcMg$. In turn, the latter is obtained from~$\HAMg$ by setting 
\HAhb{\HAcM:=M\cup\HAcS} and defining~$g_\HAcM$ by smooth extension of~$g$. 
The crucial point of this procedure is that $(\HAcM,g_\HAcM)$ is a 
Riemannian manifold as well. To avoid technical subtleties we prefer to take 
$\HAcMg$ as initial object. Due to the intimate connection between 
$\HAcMg$ and~$\HAMg$ there is often no need to mention $\HAcMg$ 
explicitly. 

\par 
For 
\HAhb{V\in\HAmf{V}} we fix 
\HAhb{q=q_V\in\bar V\HAssm V} and set 
\HAbeq \label{HAI.dV} 
\HAda_V=\HAda_{V,q}:=1+\HAdist_\HAcM(\cdot,q)\HAsco V\HAra[1,\HAiy) 
\HAeeq 
where $\HAdist_\HAcM$~is the Riemannian distance in~$\HAcMg$. Note that 
\HAhb{\sup\HAda_V=\HAiy} and 
\HAhb{\HAdist_\HAcM(\cdot,q)=\HAdist_M(\cdot,q)} on~$V$. 

\par 
For 
\HAhb{\HAGa\in\HAmfGa} there exists an open neighborhood~$\HAcU_\HAGa$ 
of~$\HAGa$ in~$\HAcM$ with 
\HAhb{\HAcU_\HAGa\cap\tilde\HAGa=\HAes} for 
\HAhb{\tilde\HAGa\in\HAmfGa} satisfying 
\HAhb{\tilde\HAGa\neq\HAGa}, and such that 
$\HAdist_\HAcM(\cdot,\HAGa)$ is a well-defined smooth function. Then 
\HAhb{U_\HAGa:=\HAcU_\HAGa\HAssm\HAGa} is open in~$M$ and the 
restriction~$\HAda_\HAGa$ of $\HAdist_\HAcM(\cdot,\HAGa)$ to~$U_\HAGa$ 
is smooth and everywhere positive. 

\par 
The following theorem is the main result of this paper as far as singular 
manifolds are concerned. Its proof is given in 
Section~\ref{HAsec-M}. Here and in similar situations obvious 
interpretations have to be used if either $\HAmf{V}$ or~$\HAmfGa$ 
is empty. 
\begin{theorem}\label{HAthm-I.C} 
Let $\HAMg$ be a Riemannian manifold with smooth cuspidal singularities 
of type $[\HAmf{V},\HAmfal,\HAmfGa,\HAmfba]$. Fix 
\HAhb{\rho\in C^\HAiy\bigl(M,(0,1]\bigr)} such that 
\HAhb{\rho\sim\HAmf{1}} on~$V_0$, 
\ \HAhb{\rho\sim\HAda_V^{\HAal_V}} on 
\HAhb{V\in\HAmf{V}}, and 
\HAhb{\rho\sim\HAda_\HAGa^{\HAba_\HAGa}} near 
\HAhb{\HAGa\in\HAmfGa}. Then $\HAMg$ is singular of type~$\rho$. 
\end{theorem} 
\begin{corollary}\label{HAcor-I.C} 
Theorem~\ref{HAthm-I.D} applies whenever $\HAMg$ is a Riemannian manifold 
with cuspidal singularities. 
\end{corollary} 
It follows from the considerations in the main body of this paper that the 
special choice of~$\rho$ is of no importance. In fact: if $\rho$~is 
replaced by~$\tilde\rho$ with 
\HAhb{\tilde\rho\sim\rho}, then 
\HAhb{(M,\ g/\rho^2)} and 
\HAhb{(M,\ g/\tilde\rho^2)} are equivalent in the sense defined in 
Section~\ref{HAsec-N}. In particular, the Sobolev-Slobodeckii 
spaces $W_{\HAcoW p}^s\HAMr$ and $W_{\HAcoW p}^s\HAMtr$ differ only 
by equivalent norms. Thus merely the behavior of~$\rho$ `near 
infinity along~$V$', for 
\HAhb{V\in\HAmf{V}}, and near 
\HAhb{\HAGa\in\HAmfGa}, where $\rho$~approaches zero, does matter. 
Notably, this shows that the choice of 
\HAhb{q_V\in\bar V\HAssm V}, as well as the special form 
of~$\rho$ on compact subsets of~$M$, is irrelevant. 

\par 
It remains to explain why the naming `manifold with smooth cuspidal 
singularities' has been chosen. This is clear if 
\HAhb{\HAmfGa=\HAes}, but needs elucidation otherwise. The following 
considerations contribute to it. But first we introduce some notation. 

\par 
For 
\HAhb{d\in\HABN^\times} we denote by~$\HABB^d$ the open unit ball 
in~$\HABR^d$, by~$\HABS^{d-1}$ its boundary, the unit sphere, and by 
\HAhb{\HABH^d:=\HABR^+\times\HABR^{d-1}} the closed right half-space, where 
\HAhb{\HABR^0:=\{0\}}. Then 
\HAhb{\HABB_+^d:=\HABB^d\cap\HABH^d} and 
\HAhb{\HABS_+^{d-1}:=\HABS^{d-1}\cap\HABH^d} are the right half-ball 
and half-sphere, respectively. Note that 
\HAhb{\HApl\HABB_+^d=\{0\}\times\HABB^{d-1}\cong\HABB^{d-1}} and 
\HAhb{\HApl\HABS_+^{d-1}=\{0\}\times\HABS^{d-2}} if 
\HAhb{d\geq2} and 
\HAhb{\HApl\HABS^0=\HAes}. Lastly, 
\HAhb{\HAthBB:=\HABB\HAssm\{0\}} for 
\HAhb{\HABB\in\{\HABB^d,\HABB_+^d\}}. 

\par 
Suppose 
\HAhb{1\leq\ell\leq m} and 
\HAhb{\HABS\in\{\HABS^{\ell-1},\HABS_+^{\ell-1}\}}. Given 
\HAhb{\HAal\geq1}, 
\HAbeq \label{HAI.Ca} 
C_\HAal(\HABS)=C_{\HAal,\ell}(\HABS):=\bigl\{\,(t,t^\HAal y) 
\ ;\ 0<t<1,\ y\in\HABS\,\bigr\}\HAis\HABR^{1+\ell} 
\HAeeq  
is an 
\hbox{$\ell$-dimensional} submanifold of~$\HABR^{1+\ell}$ and 
$$ 
\HAvp_\HAal\HAsco C_\HAal(\HABS)\HAra(0,1)\times\HABS 
\HAqb (t,t^\HAal y)\HAmt(t,y) 
$$ 
is the `canonical stretching diffeomorphism'. 
Observe that 
\HAhb{\HApl C_\HAal(\HABS)=\HAes} if 
\HAhb{\HABS=\HABS^{\ell-1}} or 
\HAhb{\ell=1}, and 
\HAhb{\HApl C_\HAal(\HABS_+^{\ell-1})=C_{\HAal,\ell-1}(\HABS^{\ell-2})} 
otherwise. 

\par 
\noindent 
\begin{minipage}[c]{230pt} 
\hspace{4mm} 
$C_\HAal(\HABS)$~is~a (blunt) \emph{model} 
\hbox{$\HAal$\emph{-cusp}}, 
respectively \emph{cone} if 
\HAhb{\HAal=1}, which is \emph{spherical} if 
\HAhb{\HABS=\HABS^{\ell-1}} and \emph{semi-spherical} otherwise. 
In Fig.~2 there is depicted~a (rotated) semi-circular model \hbox{$2$-cusp} 
in~$\HABR^3$. Its boundary consists of two disjoint one-dimensional 
generators. 

\par 
\hspace{4mm} 
We endow 
\HAhb{C_\HAal=C_\HAal(\HABS)} with the Riemannian metric~$g_{C_\HAal}$ 
induced by the natural embedding 
\HAhb{C_\HAal\HAhr\HABR^{1+\ell}}. Then $g_{C_\HAal}$~is equivalent to the 
pull-back by~$\HAvp_\HAal$ of the metric 
\HAhb{\D t^2+t^{2\HAal}g_\HABS} of 
\HAhb{(0,1)\times\HABS}, where $g_\HABS$~is the standard 
metric induced by 
\HAhb{\HABS\HAhr\HABR^\ell}. 
\end{minipage}
\hfil 
\begin{minipage}[c]{80pt} 
\vskip-1\baselineskip 
\[
\begin{picture}(78,105)(0,0)
\put(0,0){\makebox(0,0)[bl]%
{\includegraphics[height=105pt]{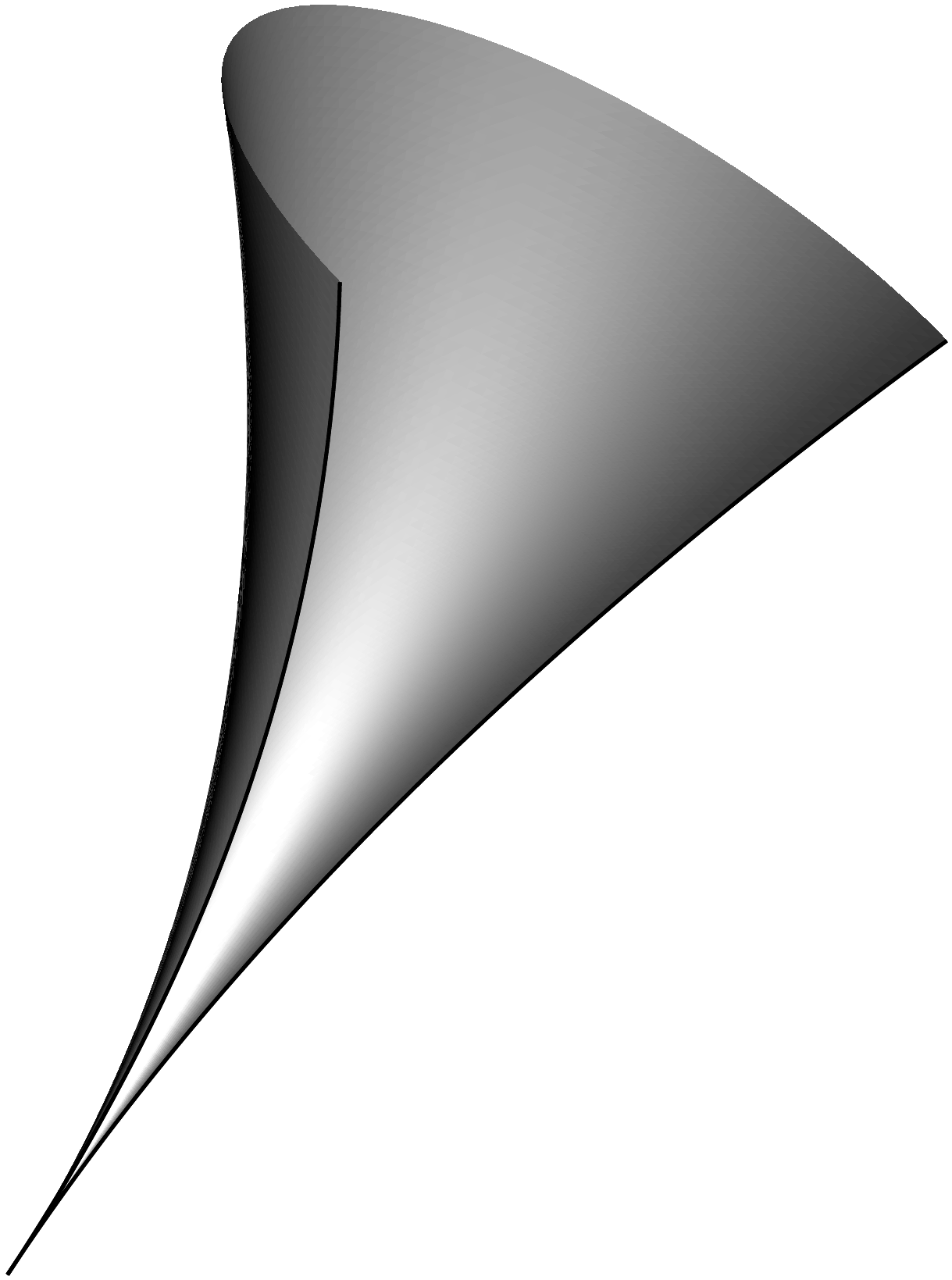}}} 
\put(50,0){\makebox(0,0)[b]{\HAhh{\small Fig.~2}}}
\end{picture}
\]
\end{minipage}
\hfil

\par 
Assume $(\HAGa,g_\HAGa)$ is an 
\HAhb{(m-\ell)}-dimensional compact connected Riemannian manifold 
without boundary. Then 
\HAhb{W_\HAal:=C_\HAal\times\HAGa}, whose metric is 
\HAhb{g_{W_\HAal}:=g_{C_\HAal}+g_\HAGa}, is a \emph{model 
\HAhb{(\HAal,\HAGa)}-wedge} which is also called \emph{spherical} if 
$C_\HAal$~is so, and \emph{semi-spherical} otherwise. If 
\HAhb{m=\ell}, then $\HAGa$~is a one-point space, $W_\HAal$~is 
naturally identified with~$C_\HAal$, and all references to and occurrences 
of~$\HAGa$ are to be disregarded. Thus every cusp is a wedge also. 

\par 
Let $U$ be open in~$M$. Then $(U,g)$, more loosely:~$U$, is~a 
\emph{spherical}, resp.\ \emph{semi-spherical}, \emph{cuspidal end of type 
\HAhb{(\HAal,\HAGa)} of}~$\HAMg$ if there exists an isometry~$\Phi_\HAal$ 
from $(U,g)$ onto a spherical, resp.\ semi-spherical, model 
\HAhb{(\HAal,\HAGa)}-wedge $(W_\HAal,g_{W_\HAal})$. In this case 
$U$~is \emph{represented by} 
\HAhb{[\Phi_\HAal,W_\HAal,g_{W_\HAal}]} or, simply, by~$\Phi_\HAal$. 

\par 
\noindent 
\begin{minipage}[c]{165pt}  
\hspace{4mm} 
Now we return to the setting of Theorem~\ref{HAthm-I.C} and consider 
a particular simple constellation. Namely, we assume that $M$~is obtained 
from a three-dimensional ellipsoid~$\HAcM$ in~$\HABR^3$ by removing an 
equator~$\HAGa$.  Its metric~$g$ is induced by the natural embedding 
\HAhb{\HAcM\HAhr\HABR^3}. In this case 
\HAhb{\HAmf{V}=\HAes} and 
\HAhb{\HAmfGa=\{\HAGa\}}. 
\end{minipage} 
\hfil 
\begin{minipage}[c]{92pt} 
\vspace*{-1\baselineskip}
\[
\begin{picture}(90,92)(0,-17)
\put(0,-1){\makebox(0,0)[bl]%
{\includegraphics[height=75pt]{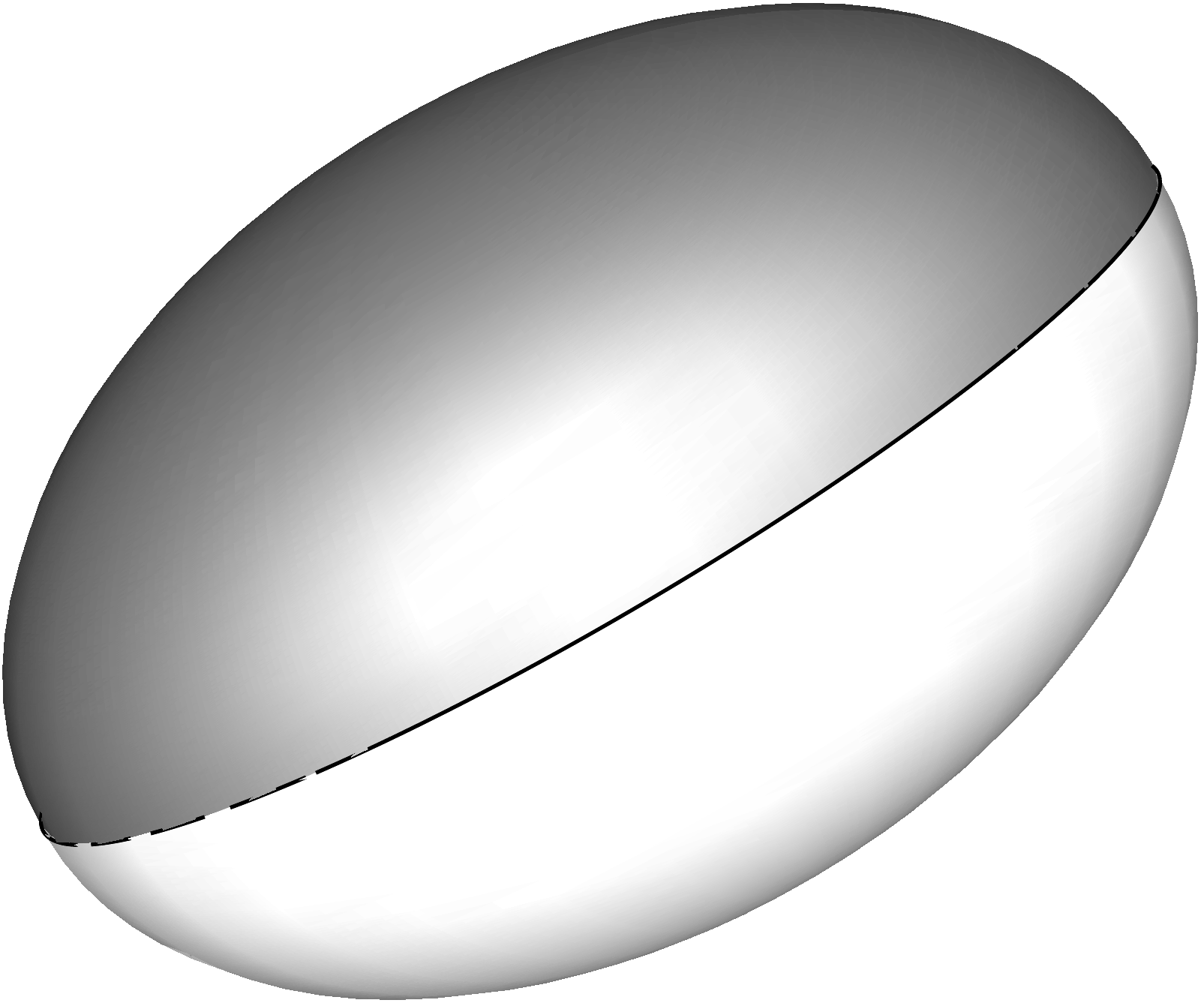}}}%
\put(45,-10){\makebox(0,0)[t]{\HAhh{\small Fig.~3}}}
\end{picture}
\] 
\end{minipage}
\hfil

\par 
On one component,~$\HApl_0M$, of the boundary of~$M$ we put 
Dirichlet conditions 
(e.g.\ on the dark side of Fig.~3) and Neumann conditions on the other 
one,~$\HApl_1M$. Note that $\HApl_0M$ and~$\HApl_1M$ meet in~$\HAcM$ 
along~$\HAGa$, but `do not see each other' in~$M$. In other words, 
$\HApl_0M$ and~$\HApl_1M$ are both open and closed in~$\HApl M$. 

\par 
\noindent 
\begin{minipage}[c]{165pt} 
\hspace{4mm} 
We consider a tubular neighborhood $U$ of~$\HAGa$ in~$M$ and represent it as 
\HAhb{\HAthBB_+^2\times\HAGa} by means of the tubular diffeomorphism 
\HAhb{\utau\HAsco U\HAra\HAthBB_+^2\times\HAGa} 
(see Section~\ref{HAsec-M} for details). A~part of it is depicted in 
Fig.~4 in which the curve along the flat side represents~$\HAGa$ 
(\HAhb{=\{0\}\times\HAGa}), which does not belong to~$\utau(U)$, 
however. 
\end{minipage} 
\hfil 
\begin{minipage}[c]{98pt} 
\vspace*{-1\baselineskip}
\[
\begin{picture}(98,92)(0,-17)
\put(0,-1){\makebox(0,0)[bl]%
{\includegraphics[height=75pt]{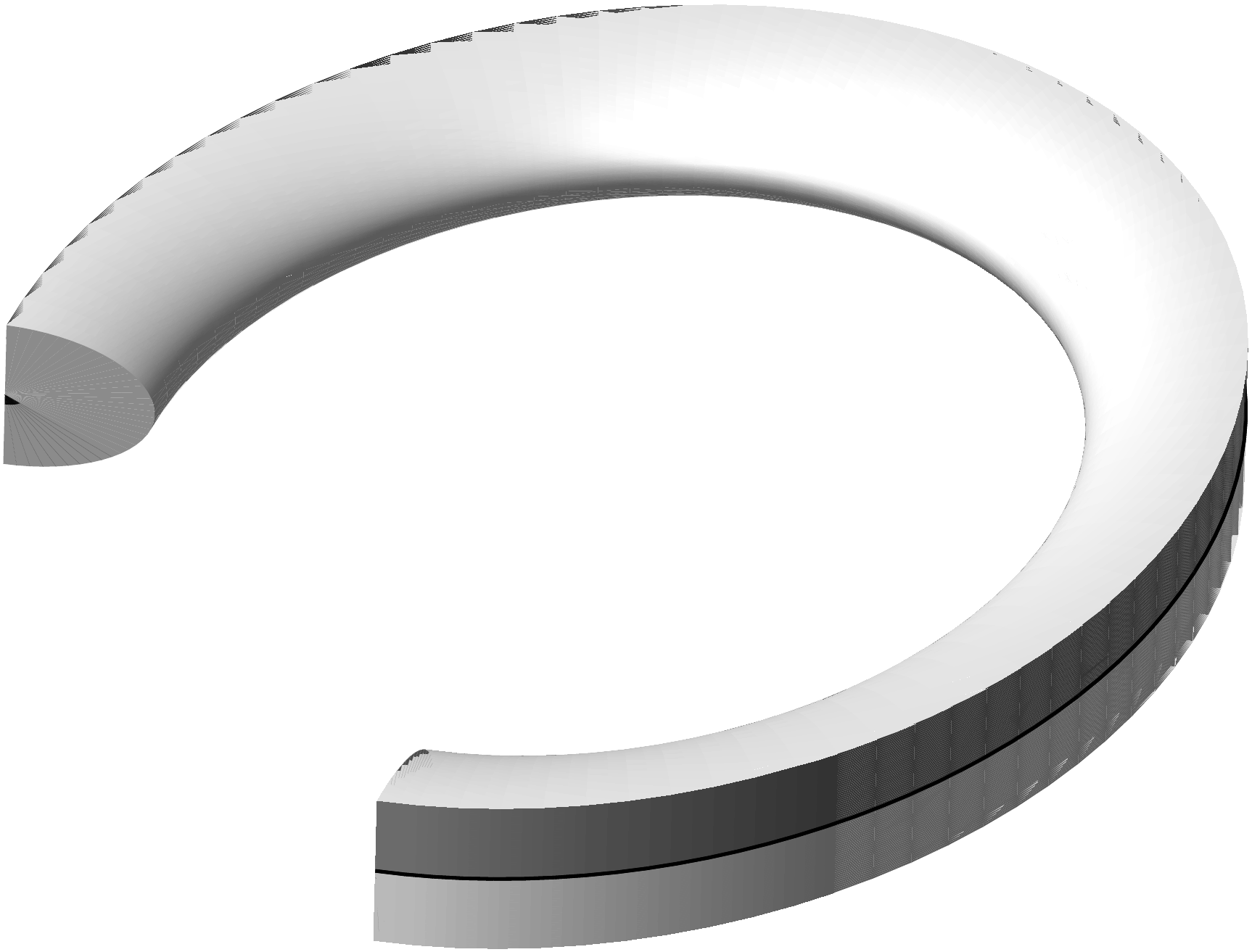}}}%
\put(49,-10){\makebox(0,0)[t]{\HAhh{\small Fig.~4}}}
\end{picture}
\] 
\end{minipage}
\vspace*{.2\baselineskip}
\hfil

\par 
Let 
$$ 
\upi\HAsco\HAthBB_+^2\HAra(0,1)\times\HABS_+^1 
\HAqb x\HAmt(|x|,\ x/|x|) 
$$ 
be the polar coordinate diffeomorphism. Then, given 
\HAhb{\HAal\geq1}, the composition 
\HAbeq \label{HAI.Ph} 
U\stackrel\utau\HAlora 
\HAthBB_+^2\times\HAGa 
\stackrel{\upi\times\HAid_\HAGa}\HAllllora 
(0,1)\times\HABS_+^1\times\HAGa 
\stackrel{\HAvp_\HAal^{-1}\times\HAid_\HAGa}\HAlllllora 
C_\HAal(\HABS_+^1)\times\HAGa 
\HAeeq 
defines a diffeomorphism~$\Phi_\HAal$ from~$U$ onto the semi-circular model 
$(\HAal,\HAGa)$-wedge 
\HAhb{W_\HAal=C_\HAal(\HABS_+^1)\times\HAGa}. We equip 
$C_\HAal(\HABS_+^1)$ with the equivalent metric 
\HAhb{\HAvp_\HAal^*(\D t^2+t^{2\HAal}g_{\HABS_+^1})} and give~$U$ the 
pull-back metric $\Phi_\HAal^*g_{W_\HAal}$. 

\par 
Let $\HAgsg$ be a Riemannian metric for~$M$ such that 
\HAhb{\HAgsg=\Phi_\HAal^*g_{W_\HAal}} on~$U$. Then $U$~is a 
semi-circular $(\HAal,\HAGa)$-end of~$\HAMgsg$. 
In Section~\ref{HAsec-M} it is shown that 
\HAhb{\Phi_\HAal^*g_{W_\HAal}\sim g/\HAda_\HAGa^{2\HAal}} on~$U$. 
Thus, if we fix any 
\HAhb{\rho\in C^\HAiy\bigl(M,(0,1]\bigr)} with 
\HAhb{\rho\sim\HAda_\HAGa^{2\HAal}} on~$U$ and 
\HAhb{\rho\sim\HAmf{1}} on 
\HAhb{M\HAssm U}, it follows from Theorem~\ref{HAthm-I.C} that 
\HAhb{(M,\ g/\rho^2)} is uniformly regular. 

\par 
\noindent 
\begin{minipage}[c]{170pt} 
\hspace{4mm}
These considerations and Corollary~\ref{HAcor-I.C} show that the 
Zaremba problem on~$\HAcM$ for (\ref{HAI.A.P}), in which Dirichlet boundary 
conditions are assigned on one half of the boundary of the 
ellipsoid~$\HAcM$ and Neumann conditions on the other half, is well-posed 
provided $\HAcA$~is regularly uniformly strongly 
\hbox{$\rho$-elliptic} where 
\HAhb{\rho\sim\HAda_\HAGa^{2\HAal}} near~$\HAGa$\linebreak  
\vspace*{-.8\baselineskip}
\end{minipage} 
\hfil 
\begin{minipage}[c]{116pt} 
\[ 
\begin{picture}(116,80)(0,-12)
\put(0,-1){\makebox(0,0)[bl]%
{\includegraphics[height=75pt]{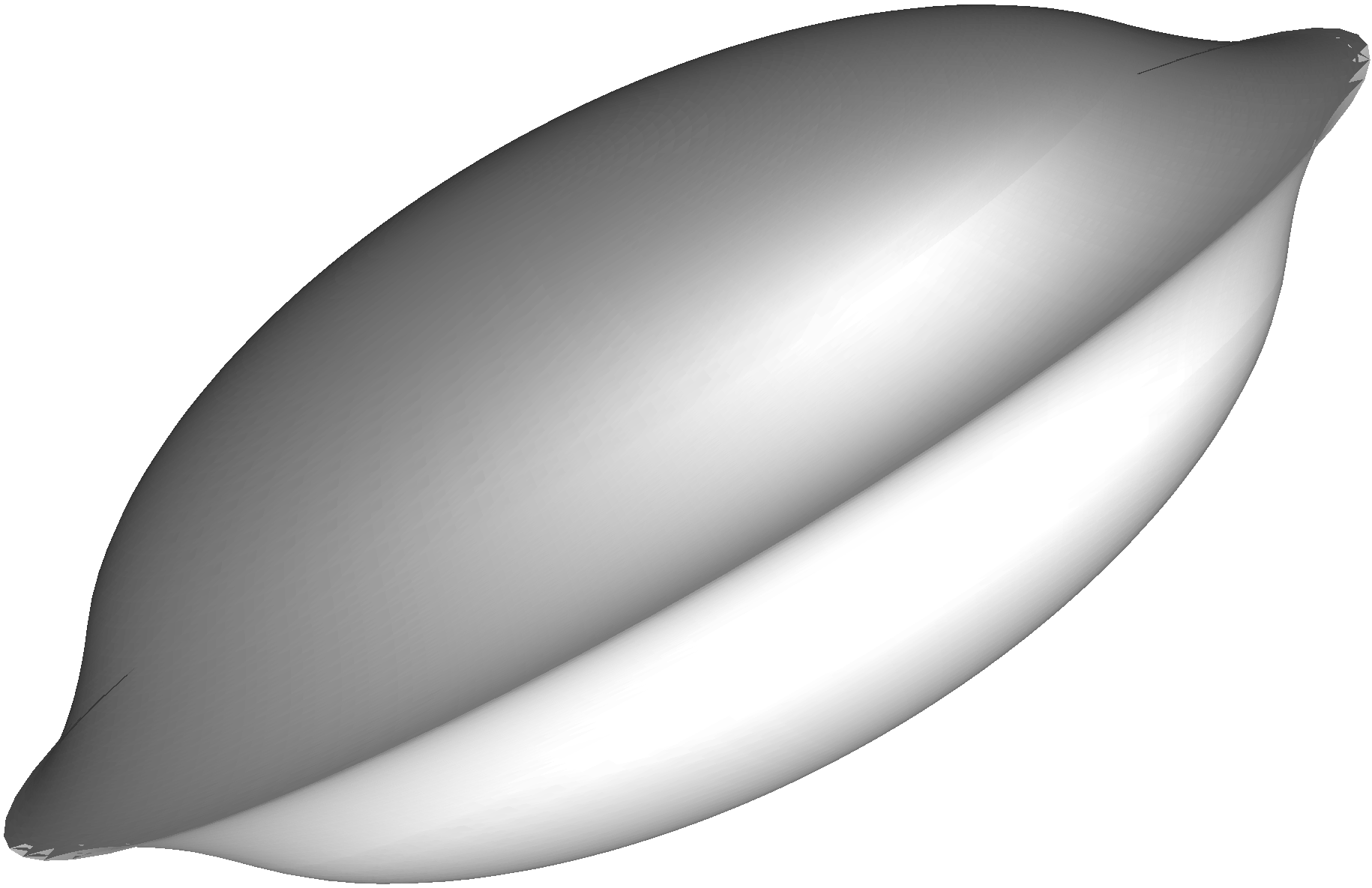}}} 
\put(105,-5){\makebox(0,0)[br]{\HAhh{\small Fig.~5}}}
\end{picture}
\] 
\end{minipage}
\hfil\\  
and 
\HAhb{\rho\sim\HAmf{1}} away from~$\HAGa$. They also show that 
$\HAMg$ can be visualized as a manifold with a 
cuspidal end of type~$(\HAal,\HAGa)$. This is illustrated by 
Fig.~5 for the case where 
\HAhb{\HAal=1}. 

\par 
The arguments used in this simple case extend to the general setting. 
This leads in Section~\ref{HAsec-M} to the proof of the following proposition 
which clarifies our choice of the name for (\ref{HAI.MM}). 
\begin{proposition}\label{HApro-I.C} 
Let $\HAMg$ have smooth cuspidal singularities of type 
$[\HAmf{V},\HAmfal,\HAmfGa,\HAmfba]$ and let 
\HAhb{\HAba=\HAba_\HAGa} be the cuspidal weight for 
\HAhb{\HAGa\in\HAmfGa}. Then there exists an open neighborhood~$\HAcU$ 
of~$\HAGa$ in~$\HAcM$ such that 
\HAhb{U:=\HAcU\HAssm\HAGa} is~a $(\HAba,\HAGa)$-cuspidal end of $\HAMg$. 
\end{proposition} 
The preceding treatment indicates that there are two possible ways of 
looking at these problems. In the first one we put forward the differential 
equation setting. Then the singular manifold has an inferior position and 
it is only the singularity function~$\rho$ which comes into play. 
In the second approach the geometric appearance of the singular manifold 
is relevant. In this case we start off with a singular manifold $\HAMg$ 
which may not be obtained from a uniformly regular ambient manifold by 
cutting out lower-dimensional submanifolds. Instead, $\HAMg$~can have more 
general singular ends~$U$; namely such that $U$~is isometric to a 
model $(\HAal,\HAGa)$-wedge over~$(B,g_B)$, where $(B,g_B)$~is as in 
(\ref{HAI.F}), and $B$~replaces~$\HABS$ in definition (\ref{HAI.Ca}). 
Theorem~\ref{HAthm-M.S} and Proposition~\ref{HApro-M.V}(i) 
guarantee then the existence of singularity 
functions~$\rho$, modeling again the geometric structure of~$\HAMg$, such 
that $\HAMg$~is singular of type~$\rho$. Consequently, we can obtain 
well-posedness theorems for degenerate parabolic equations on singular 
manifolds by applying Theorem~\ref{HAthm-I.D}. 

\par 
Up to now we have considered the case in which we introduce a conformal 
metric 
\HAhb{g/\rho^2} on~$M$ in order to render it uniformly regular. This 
means that we restrict ourselves to differential operators with isotropic 
degenerations. However, other choices are possible also. For example, 
in the setting (\ref{HAI.Ph}) we can endow 
\HAhb{(0,1)\times\HABS_+^1\times\HAGa} with the metric 
\HAhb{t^{-2\HAal}\D t^2+g_{\HABS_+^1}+g_\HAGa} instead of 
\HAhb{\D t^2+t^{2\HAal}g_{\HABS_+^1}+g_\HAGa} as above. This is a consequence 
of the next theorem which is also proved in Section~\ref{HAsec-M}. 
For simplicity, we consider the case where $M$~has only one singular end. 
The extension to the general case is straightforward. Moreover, 
\HAbeq \label{HAI.c} 
(0,1)\times\HABS\times\HAGa 
\HAeeq 
is the canonical representation of a tubular neighborhood~$U$ of~$\HAGa$ 
in~$\HAMg$ in the sense made precise later in this paper. 
\begin{theorem}\label{HAthm-I.A} 
Let (\ref{HAI.MM}) be satisfied with 
\HAhb{\HAmf{V}=\HAes} and 
\HAhb{\HAmfGa=\{\HAGa\}}, and fix 
\HAhb{\HAal>0}. Let $U$ be a tubular neighborhood of~$\HAGa$ in~$\HAMg$. 
Suppose $\HAgsg$~is a metric for~$M$ which coincides on 
\HAhb{M\HAssm U} with~$g$ and equals near~$\HAGa$ 
$$ 
t^{-2}\D t^2+t^{-2\HAal}(g_\HABS+g_\HAGa) 
\quad\text{if\/ }0<\HAal\leq1\;, 
$$ 
respectively 
\HAbeq \label{HAI.al}  
t^{-2(\HAal+1)}\D t^2+g_\HABS+g_\HAGa 
\quad\text{if\/ }\HAal>1\;, 
\HAnpb 
\HAeeq 
in the canonical representation~(\ref{HAI.c}) of~$U$. Then $\HAMgsg$~is 
uniformly regular. 
\end{theorem} 
Recall that~$g_\HABS$, resp.~$g_\HAGa$, is absent if 
\HAhb{\ell=m}, resp.\ 
\HAhb{\ell=1}. 

\par 
By applying Theorem~\ref{HAthm-I.A.P} to the setting of 
Theorem~\ref{HAthm-I.A} we obtain well-posedness results for parabolic 
problems with anisotropic degeneration. To indicate the inherent 
potential of such applications we consider the particularly 
interesting setting in which 
$\HAGa$~is a compact connected component of the boundary of~$\HAcM$. We 
also suppose, for simplicity, that $\HAcA$~is the negative Laplace-Beltrami 
operator~%
$-\HADa$ of~$\HAMg$ and assume 
\HAhb{\HAal>1}. Then it follows from (\ref{HAI.al}) that the (interior) 
flux vector field satisfies in a collar neighborhood of~$\HAGa$ 
$$ 
\HAgrad\sim(\HAda^{2\HAal}\HApl_\unu,\HAgrad_\HAGa)\;. 
$$ 
Hence it degenerates in the normal direction only and there 
is no degeneration at all in tangential directions. This is 
in contrast to the isotropic case in which Corollary~\ref{HAcor-I.C} applies 
and, in the present setting, gives 
$$ 
\HAgrad\sim\HAda^{2\HAal}(\HApl_\unu,\HAgrad_\HAGa) 
\HAnpb 
$$ 
near~$\HAGa$. 

\par 
There has been done an enormous amount of research on 
\emph{elliptic} equations on singular manifolds. All of it is related, 
in one way or another, to the seminal paper by 
V.A.~Kondrat{\cprime}ev~\cite{Kon67a}. It is virtually impossible to 
review this work here and to do justice to the many authors who 
contributed. It may suffice to mention the three most active groups and 
some of their principal exponents. First, there is the Russian school which 
builds directly on Kondrat{\cprime}ev's work and is also strongly 
application-oriented (see the numerous papers and books by  
V.G.~Maz{\cprime}ya, S.A.~Nazarov, and their coauthors, for example). 
Second, the group gathering around \hbox{B.-W.}~Schulze has constructed 
an elaborate calculus of pseudo-differential algebras on manifolds with 
singularities, mainly of conical and cuspidal type. For a lucid presentation 
of some of its aspects in the simplest setting of manifolds with cuspidal 
points and wedges we refer to the book of V.E.~Nazaikinskii, A.Yu.~Savin, 
\hbox{B.-W.}-Schulze, and B.Yu.~Sternin~\cite{NSSS06a}. Third, another 
general approach to pseudo-differential 
operators on manifolds with singularities has been developed by 
R.~Melrose and his coworkers. A~brief explanation, stressing the differences 
of the techniques used by the latter two groups, is found in the 
section~`Bibliographical Remarks' of~\cite{NSSS06a}. Henceforth, we call 
these methods `classical' for easy reference. 

\par 
To explain to which extent our point of view differs from the classical 
approach we consider the simplest case, namely, a~manifold with one conical 
singularity. By means of the stretching diffeomorphism the model 
cone~$C_1(\HABS)$ 
is represented by the `stretched manifold' 
\HAhb{(0,1)\times\HABS} whose metric is 
$$ 
g=\D t^2+t^2g_\HABS=t^2\bigl((\D t/t)^2+g_\HABS\bigr)\;. 
$$ 
Thus the corresponding Laplace-Beltrami operator is given by 
\HAhb{t^{-2}\bigl((t\HApl_t)^2+\HADa_\HABS\bigr)}. More generally, in the 
classical theories there are considered differential operators which, on 
the stretched manifold, are (in the second order case) of the 
form~$t^{-2}L$, 
where $L$~is a uniformly elliptic operator generated by the vector fields 
\HAhb{t\HApl_t,\HApl_{\HAta^1},\ldots,\HApl_{\HAta^{m-1}}} with 
\HAhb{(\HAta^1,\ldots,\HAta^{m-1})} being local coordinates for~$\HABS$. 

\par 
Instead, our approach is based on the metric 
$$ 
\hat g=g/t^2=(\D t/t)^2+g_\HABS 
$$ 
whose Laplacian is 
\HAhb{(t\HApl_t)^2+\HADa_\HABS}. Hence our theory addresses operators of 
type~$L$. As has been shown in~\cite{Ama14a}, and explained above, 
this amounts to the study of degenerate differential operators in 
the original setting. (Let us mention, in passing, that the variable 
transformation 
\HAhb{t=e^{-s}} carries 
\HAhb{\bigl((0,1)\times\HABS,\ \D t^2/t^2+g_\HABS\bigr)} onto 
\HAhb{\bigl((1,\HAiy)\times\HABS,\ \D s^2+g_\HABS\bigr)} whose Laplacian is 
\HAhb{\HApl_s^2+\HADa_\HABS}. The latter Riemannian manifold 
is easily seen to 
be uniformly regular `near infinity', that is, cofinally uniformly regular 
as defined in Section~\ref{HAsec-F}. These trivial observations form part 
of the basis of this paper.) 

\par 
The factor~$t^{-2}$ multiplying~$L$ in the classical approach does 
not play a decisive role for the proof of many results in the 
elliptic theory since it can be `moved to the right-hand side'. However, 
the situation changes drastically if a spectral parameter is included since 
\HAhb{t^{-2}L+\HAlda=t^{-2}(L+\HAlda t^2)} is no longer 
of the same type as~$L$. This is the reason why---at least up to 
now---there is no general theory of `classical' parabolic equations on 
singular manifolds. 

\par 
All singular manifolds discussed so far belong to the class of manifolds 
with `smooth singularities'. By this we mean that the bases of 
the cusps themselves do not have singularities. If they are also 
singular, we model manifolds with cuspidal corners and more complicated 
higher order singularities. For the sake of simplicity we do not 
consider such cases in this paper. However, all definitions and 
theorems presented below have been `localized' so that an extension to 
`corner manifolds' can be built directly on the present work. 

\par 
In the next section, besides fixing our basic notation, we give precise 
(localized) definitions of Riemannian manifolds which are 
uniformly regular, respectively singular of type~$\rho$. All subsequent 
considerations are given for the latter class. Corresponding assertions 
for uniformly regular manifolds are obtained by setting 
\HAhb{\rho=\HAmf{1}}. 

\par 
Section~\ref{HAsec-P} contains preliminary technical results and, 
in particular, the proof of (an extended version of) 
Example~\ref{HAexa-I.ex}(e). As a first application of these investigations 
we present, in Section~\ref{HAsec-U}, some easy examples of uniformly 
regular Riemannian manifolds. 

\par 
In Section~\ref{HAsec-C} we introduce a general class of 
`cusp characteristics' 
which provides us with ample families of singularity functions~$\rho$. 
It is a consequence of Example~\ref{HAexa-C.ex}(b) that our results 
do not only apply to manifolds with cuspidal singularities, but also to 
manifolds with `exponential' cusps and wedges, or in more general 
situations (see Example~\ref{HAexa-C.ex}(b) and Lemma~\ref{HAlem-M.R}). 

\par 
In the proximate section we introduce model wedges and explore their 
singularity behavior under various Riemannian metrics. The case of the 
`natural' metric, induced by the embedding in the ambient Euclidean space, 
is treated in Section~\ref{HAsec-E}. The last section contains the main 
results and the proofs left out in the introduction. 
\section{Notations and Definitions}\label{HAsec-N}%
By~a \emph{manifold} we always mean a smooth, that is, 
\hbox{$C^\HAiy$ manifold} with (possibly empty) boundary such that its 
underlying topological space is separable and metrizable. Thus we work in 
the smooth category. A~manifold does not need to be connected, but all 
connected components are of the same dimension. 

\par 
Let $M$ be a submanifold of some manifold~$N$. Then 
\HAhb{\HAia\HAsco M\HAhr N}, or simply 
\HAhb{M\HAhr N}, denotes the natural embedding 
\HAhb{p\HAmt p}, for which we also write~$\HAia_M$. (The meaning of~$N$ 
will always be clear from the context.) This embedding induces the 
natural (fiber-wise linear) embedding 
\HAhb{\HAia\HAsco TM\HAhr TN} of the tangent bundle of~$M$ into 
the one of~$N$. 

\par 
Let $\HANh$ be a Riemannian manifold. Then $\HAia^*h$~denotes the restriction 
of~$h$ to 
\HAhb{M\HAhr N}, that is, 
\HAhb{(\HAia^*h)(p)\HAXY=h(p)\HAXY} for 
\HAhb{p\in M} and 
\HAhb{X,Y\in T_pM\HAhr T_pN}. If $g$~is a Riemannian metric for~$M$, then 
$\HAMg$~is a \emph{Riemannian submanifold} of~$\HANh$, in symbols: 
\HAhb{\HAMg\HAhr\HANh}, if 
\HAhb{g=\HAia^*h}. If $M$~has codimension~$0$, then we write again $h$ 
for~$\HAia^*h$. 

\par 
The Euclidean metric 
$$ 
|\D x|^2=(\D x^1)^2+\cdots+(\D x^m)^2 
$$ 
of~$\HABR^m$ is also denoted 
by~$g_m$. Unless explicitly stated otherwise, we identify~$\HABR^m$ 
with~$\HARmgm$. 

\par 
Given a finite-dimensional normed vector space 
\HAhb{E=(E,\HAvsdot)} and an open subset~$V$ of $\HABR^m$ or~$\HABH^m$, 
we write 
\HAhb{\HAVsdot_{k,\HAiy}} for the usual norm of $BC^k\HAVE$ , 
the Banach space of all 
\HAhb{v\in C^k\HAVE} such that $|\HApa v|$~is uniformly bounded for 
\HAhb{\HAal\in\HABN^m} with 
\HAhb{|\HAal|\leq k}. (We use standard multi-index notation.) As usual, 
\HAhb{C^k(V)=C^k\HAVR} etc., and 
\HAhb{\HAVsdot_\HAiy=\HAVsdot_{\HAiy,0}}. 

\par 
Suppose $M$ and~$N$ are manifolds and 
\HAhb{\HAvp\HAsco M\HAra N} is a diffeomorphism. By~$\HAvp^*$ we denote the 
pull-back by~$\HAvp$ (of general tensor fields) and 
\HAhb{\HAvp_*:=(\HAvp^{-1})^*} is the corresponding push-forward. Thus 
\HAhb{\HAvp^*v=v\circ\HAvp} for a function~$v$ on~$N$. Recall that the 
pull-back~$\HAvp^*h$ of a Riemannian metric~$h$ on~$N$ is given by 
\HAbeq \label{HAN.phh} 
(\HAvp^*h)\HAXY=\HAvp^*\bigl(h(\HAvp_*X,\HAvp_*Y)\bigr) 
\HAnpb  
\HAeeq 
for all vector fields $X$ and~$Y$ on~$M$. 

\par 
As usual, 
\HAhb{(\HApl/\HApl x^1,\ldots,\HApl/\HApl x^m)} is the coordinate frame 
for~$T_{U_{\HAcoU\HAka}}M$ associated with the local coordinates 
\HAhb{\HAka=(x^1,\ldots,x^m)} on 
\HAhb{U_{\HAcoU\HAka}:=\HAdom(\HAka)}. Here $T_{U_{\HAcoU\HAka}}M$~denotes 
the restriction of~$TM$ to 
\HAhb{U_{\HAcoU\HAka}\HAhr M}. Thus 
\HAhb{\HAka_*(\HApl/\HApl x^i)=e_i}, where 
\HAhb{(e_1,\ldots,e_m)} is the standard basis for~$\HABR^m$. The basis 
for~$T_{U_{\HAcoU\HAka}}^*M$, dual to 
\HAhb{(\HApl/\HApl x^1,\ldots,\HApl/\HApl x^m)}, is 
\HAhb{(\D x^1,\ldots,\D x^m)} with $\D x^i$~being the differential of the 
coordinate function~$x^i$. 

\par 
Let $g$ be a Riemannian metric on~$M$. For a local chart 
\HAhb{\HAka=(x^1,\ldots,x^m)} the local representation for~$g$ with respect 
to these coordinates is given by 
$$ 
g=g_{ij}\D x^i\,\D x^j 
\HAqa g_{ij}:=g\left(\frac\HApl{\HApl x^i},\frac\HApl{\HApl x^j}\right)\;. 
$$ 
Here and below, we employ the standard summation convention. Then, given 
vector fields 
\HAhb{\xi=\xi^ie_i} and 
\HAhb{\eta=\eta^je_j} on~$\HAka(U_{\HAcoU\HAka})$, it follows from 
(\ref{HAN.phh}) that 
$$ 
\HAbal 
\HAka_*g(\xi,\eta)  
&=\HAka_*\bigl(g(\HAka^*\xi,\HAka^*\eta)\bigr) 
 =\HAka_*\bigl(g(\xi^i\HApl/\HApl x^i,\eta^j\HApl/\HApl x^j)\bigr)\\ 
&=\HAka_*(g_{ij}\xi^i\eta^j) 
 =\HAka_*g_{ij}\xi^i\eta^j 
 =(g_{ij}\circ\HAka^{-1})\xi^i\eta^j\;. 
 \HAeal 
$$ 
Thus $\HAka_*g(x)$~is for each 
\HAhb{x\in\HAka(U_{\HAcoU\HAka})} a~positive definite symmetric bilinear 
form. Hence there exists 
\HAhb{c(x)\geq1} such that 
\HAbeq \label{HAN.kg} 
|\xi|^2/c(x)\leq\HAka_*g(x)(\xi,\xi)\leq c(x)\,|\xi|^2 
\HAqa \xi\in\HABR^m 
\HAqb x\in\HAka(U_{\HAcoU\HAka})\;,
\HAeeq 
where 
\HAhb{|\xi|:=\sqrt{g_m(\xi,\xi)}=\sqrt{(\xi\HAsn\xi)}} is the Euclidean norm 
of 
\HAhb{\xi\in\HABR^m}. In other words, 
$$ 
g_m/c(x)\leq\HAka_*g(x)\leq c(x)g_m 
\HAqa x\in\HAka(U_{\HAcoU\HAka})\;. 
$$ 

\par 
We set 
\HAhb{Q:=(-1,1)\HAis\HABR}. If $\HAka$~is a local chart for an 
\hbox{$m$-dimensional} manifold~$M$, then it is \emph{normalized} (at~$p$) 
if 
\HAhb{\HAka(U_{\HAcoU\HAka})=Q^m} whenever 
\HAhb{U_{\HAcoU\HAka}\HAis\HAci M}, the interior of~$M$, whereas 
\HAhb{\HAka(U_{\HAcoU\HAka})=Q^m\cap\HABH^m} if $U_{\HAcoU\HAka}$~has a 
nonempty intersection with the boundary~$\HApl M$ of~$M$ (and 
\HAhb{\HAka(p)=0}). We put 
\HAhb{Q_\HAka^m:=\HAka(U_{\HAcoU\HAka})} if $\HAka$~is normalized. 
(We find it convenient 
to use normalization by cubes. Of course, we could equally well normalize 
by employing Euclidean balls.) 

\par 
Let $M$ be an \hbox{$m$-dimensional} manifold and $S$ a (nonempty) subset 
thereof. Given an atlas~$\HAgK$ for~$M$, we set 
$$ 
\HAgK_S:=\{\,\HAka\in\HAgK\ ;\ U_{\HAcoU\HAka}\cap S\neq\HAes\,\}\;. 
$$ 
Then 
$\HAgK_S$~has \emph{finite multiplicity} or: $\HAgK$~has \emph{finite 
multiplicity on}~$S$, if there exists 
\HAhb{k\in\HABN} such that any intersection of more than~$k$ coordinate 
patches~$U_{\HAcoU\HAka}$ with 
\HAhb{\HAka\in\HAgK_S} is empty. The least such~$k$ is then the multiplicity, 
$\HAmult(\HAgK_S)$, of~$\HAgK_S$. The atlas~$\HAgK$  is 
\emph{shrinkable on}~$S$, or: $\HAgK_S$ is \emph{shrinkable}, if 
$\HAgK_S$~consists of normalized charts and there exists 
\HAhb{r\in(0,1)} such that 
\HAbeq \label{HAN.S} 
\{\,\HAka^{-1}(rQ_\HAka^m)  
\ ;\ \HAka\in\HAgK_S\,\} 
\HAnpb 
\HAeeq 
is a cover of~$S$. It is \emph{shrinkable on~$S$ to} 
\HAhb{r_0\in(0,1)} if (\ref{HAN.S}) holds for each 
\HAhb{r\in(r_0,1)}. 

\par 
An atlas~$\HAgK$ for~$M$ is \emph{uniformly regular on}~$S$ if  
\HAbeq \label{HAN.KS} 
 \HAbal 
 \\ 
 \noalign{\vskip-6\jot}   
 \rm{(i)} \qquad    &\HAgK_S\text{ is shrinkable and has finite 
                     multiplicity}\;;\\ 
 \noalign{\vskip-1\jot} 
 \rm{(ii)}\qquad    &\|\tilde\HAka\circ\HAka^{-1}\|_{k,\HAiy}\leq c(k)\;, 
                     \ \ \HAka,\tilde\HAka\in\HAgK_S\;, 
                     \ \ k\in\HABN\;.
 \HAeal
\HAeeq 
In (ii) and in similar situations it is understood that only 
\HAhb{\HAka,\tilde\HAka\in\HAgK_S} with 
\HAhb{U_{\HAcoU\HAka}\cap U_{\HAcoU\tilde\HAka}\neq\HAes} are being 
considered. Two atlases $\HAgK$ and~$\tilde\HAgK$ for~$M$, which are 
uniformly regular on~$S$, are \emph{equivalent on}~$S$, in symbols: 
\HAhb{\smash{\HAgK\HAapproxS\tilde\HAgK}\HAsmtB}, if 
\HAbeq \label{HAN.Keq} 
 \HAbal 
 \\ 
 \noalign{\vskip-6\jot}  
 \rm{(i)} \qquad    &\HAcard\{\,\tilde\HAka\in\tilde\HAgK_S\;, 
                     \ \ U_{\HAcoU\tilde\HAka}
                     \cap U_{\HAcoU\HAka}\neq\HAes\,\}
                     \leq c\;, 
                     \ \ \HAka\in\HAgK_S\;;\\                      
 \noalign{\vskip-1\jot} 
 \rm{(ii)}\qquad    &\|\tilde\HAka\circ\HAka^{-1}\|_{k,\HAiy} 
                     +\|\HAka\circ\tilde\HAka^{-1}\|_{k,\HAiy}\leq c(k)\;, 
                     \ \ \HAka\in\HAgK_S\;, 
                     \ \ \tilde\HAka\in\tilde\HAgK_S\;, 
                     \ \ k\in\HABN\;.
 \HAeal
\HAeeq 
This defines an equivalence 
relation on the class of all atlases for~$M$ which are uniformly regular 
on~$S$. Each equivalence class is~a \emph{structure of uniform 
regularity on}~$S$. We write~%
$\HAeea{\HAgK}_S$ for it to indicate that it is \emph{generated} 
by~$\HAgK$, that is, contains~$\HAgK$ as a representative. If $M$~is endowed 
with a structure~%
$\HAeea{\HAgK}_S$ of uniform regularity on~$S$, then 
$\bigl(M,\HAeea{\HAgK}_S\bigr)$ is~a \emph{uniformly regular manifold 
on}~$S$. 

\par 
Let 
$\bigl(M,\HAeea{\HAgK}_S\bigr)$ be a uniformly regular manifold 
on~$S$ and let $g$ be a Riemannian metric for~$M$. Suppose 
\HAbeq \label{HAN.Rr} 
 \HAbal
 \\ 
 \noalign{\vskip-6\jot}  
 \rm{(i)} \qquad    &\HAka_*g\sim g_m\;, 
                     \ \ \HAka\in\HAgK_S\;.\\                      
 \noalign{\vskip-1\jot} 
 \rm{(ii)}\qquad    &\|\HAka_*g\|_{k,\HAiy}\leq c(k)\;, 
                     \ \ \HAka\in\HAgK_S\;, 
                     \ \ k\in\HABN\;.
 \HAeal
\HAeeq 
It follows from (\ref{HAN.Keq}) that (\ref{HAN.Rr}) prevails if $\HAgK_S$~is 
replaced by any~$\tilde\HAgK_S$ with 
\HAhb{\smash{\tilde\HAgK\HAapproxS\HAgK}\HAsmtB}. Thus it is meaningful 
to say that $g$ is~a \emph{Riemannian metric for} 
$\bigl(M,\HAeea{\HAgK}_S\bigr)$ which is \emph{uniformly regular on}~$S$ 
if (\ref{HAN.Rr}) applies to some, hence every, representative of~%
$\HAeea{\HAgK}_S$. We also say that two such metrics $g$ and~$\bar g$ 
are \emph{equivalent on}~$S$, 
\ \HAhb{\smash{g\HAsimS\bar g}\HAsmtB}, if 
\HAhb{g\HAsn S\sim\bar g\HAsn S}. This defines an equivalence relation on 
the class of all Riemannian metrics for 
$\bigl(M,\HAeea{\HAgK}_S\bigr)$ which are uniformly regular on~$S$. 
Similarly as above, 
$\HAeea{g}_S$~is the equivalence class containing the 
representative~$g$. 

\par 
By~a \HAhh{uniformly regular Riemannian manifold} \emph{on}~$S$, written as  
$\bigl(M,\HAeea{\HAgK}_S,\HAeea{g}_S\bigr)$, we mean a uniformly 
regular manifold 
$\bigl(M,\HAeea{\HAgK}_S\bigr)$ on~$S$ equipped with an equivalence 
class of uniformly regular Riemannian metrics on~$S$. It is a convenient 
abuse of language to say instead that $(M,\HAgK,g)$ is a Riemannian 
manifold which is 
uniformly regular on~$S$. Even more loosely, $\HAMg$~is 
(a~manifold which is) \HAhh{uniformly regular} \emph{on}~$S$, if there exists 
an atlas~$\HAgK$ which is uniformly regular on~$S$ such that 
$\bigl(M,\HAeea{\HAgK}_S,\HAeea{g}_S\bigr)$ is a uniformly regular 
Riemannian manifold on~$S$. 

\par 
Suppose 
\HAhb{\rho\in C^\HAiy\bigl(M,(0,\HAiy)\bigr)} and let $g$ be a Riemannian 
metric for~$M$. Then $\rho$ is~a \emph{singularity function for $\HAMg$ 
on}~$S$, if there exists an atlas~$\HAgK$ which is uniformly regular on~$S$ 
such that $(M,\HAgK,g/\rho^2)$ is a Riemannian manifold which is uniformly 
regular on~$S$. Two singularity functions are \emph{equivalent on}~$S$, 
\HAhb{\smash{\rho\HAapproxS\tilde\rho}\HAsmtB}, if 
\HAhb{\smash{\HAgK\HAapproxS\tilde\HAgK}\HAsmtB} and 
\HAhb{\smash{g/\rho^2\HAsimS g/\tilde\rho^2}\HAsmtB}. We denote by~%
$\HAeea{\rho}_S$ the equivalence class of singularity functions 
containing the representative~$\rho$, the \emph{singularity type} of~$\HAMg$ 
on~$S$. Finally, the \HAhh{Riemannian manifold}~$\HAMg$ \HAhh{is singular of 
type}~%
$\HAeea{\rho}_S$---more loosely: 
\HAhh{of type}~$\rho$ \emph{on}~$S$---if 
$(M,g/\rho^2)$ is uniformly regular on~$S$. Clearly, $\HAMg$~is 
singular of type~%
$\HAeea{\HAmf{1}}_S$ iff it is uniformly regular on~$S$. 

\par 
A~pair $\HArgK$ is~a \emph{singularity datum} for~$\HAMg$ on~$S$ if 
\HAbeq \label{HAN.sd} 
 \HAbal 
  \\ 
 \noalign{\vskip-6\jot}   
 \rm{(i)} \qquad    &\rho\in C^\HAiy\bigl((M,(0,\HAiy)\bigr)\;.\\
 \noalign{\vskip-1\jot} 
 \rm{(ii)} \qquad   &\HAgK\text{ is an atlas which is uniformly regular 
                     on }S\;.\\                      
 \noalign{\vskip-1\jot} 
 \rm{(iii)} \qquad  &\|\HAka_*\rho\|_{k,\HAiy}\leq c(k)\rho_\HAka\;, 
                     \ \ \HAka\in\HAgK_S\;, 
                     \ \ k\in\HABN\;,\\ 
 \noalign{\vskip-1\jot} 
                     &\text{where }
                     \rho_\HAka:=\HAka_*\rho(0)=\rho 
                     \bigl(\HAka^{-1}(0)\bigr)\;.\\ 
 \noalign{\vskip-1\jot} 
 \rm{(iv)} \qquad   &\rho\HAsn U_{\HAcoU\HAka}\sim\rho_\HAka\;, 
                     \ \ \HAka\in\HAgK_S\;.\\                      
 \noalign{\vskip-1\jot} 
 \rm{(v)} \qquad    &\HAka_*g\sim\rho_\HAka^2g_m\;, 
                     \ \ \HAka\in\HAgK_S\;.\\                      
 \noalign{\vskip-1\jot} 
 \rm{(vi)}\qquad    &\|\HAka_*g\|_{k,\HAiy}\leq c(k)\rho_\HAka^2\;, 
                     \ \ \HAka\in\HAgK_S\;, 
                     \ \ k\geq0\;.
 \HAeal
\HAeeq 
It is easily verified that $(M,\HAgK,g/\rho^2)$ is uniformly regular 
on~$S$ if $\HArgK$~is a singularity datum for~$\HAMg$ on~$S$. Thus 
$\rho$~is a singularity function for~$\HAMg$ if 
$\HArgK$~is a singularity datum for it. 

\par 
The `localization' of all these quantities `to~$S$' is introduced for 
technical reasons. Our principal interest concerns the choice 
\HAhb{S=M}. In this case the qualifiers `on~$S$' and the symbol~$S$ 
are omitted, of course. 
\section{Preliminaries}\label{HAsec-P}%
Let $\HAMg$ be a Riemannian manifold and 
\HAhb{X\HAis M}. For 
\HAhb{p,q\in X} we denote by 
\HAhb{d_X(p,q)=d_{g,X}(p,q)} the distance between $p$ and~$q$ \emph{in}~$X$. 
Thus $d_X(p,q)$ is the infimum of the lengths of all piece-wise smooth paths 
of~$M$ joining~$p$ to~$q$ within~$X$. If $p$ and~$q$ lie in different 
connected components, then 
\HAhb{d_X(p,q):=\HAiy}. 

\par We suppose 
\HAhb{\HABX\in\{\HABR^m,\HABH^m\}}, 
\ $X$~is open in~$\HABX$, and 
\HAhb{S\HAis X}. We denote by~$\HAda_S$ the distance in~$X$ from~$S$ to 
\HAhb{\HABX\HAssm X}, that is, 
\HAhb{\HAda_S:=\inf_{p\in S}d_X(p,\HABX\HAssm X)}, where 
\HAhb{d_X(p,\HAes):=\HAiy}. Then we assume 
\HAhb{0<\HAda\leq\HAda_S/\sqrt{m}} and set 
$$ 
Z_{\HAda,X}:=\bigl\{\,z\in\HABZ^m\cap\HABX 
\ ;\ \HAda(z+Q_z)\cap X\neq\HAes\,\bigr\}\;, 
$$ 
where 
\HAhb{Q_z:=Q^m} if 
\HAhb{z\in\HAci X} and 
\HAhb{Q_z:=Q^m\cap\HABH^m} otherwise. Given 
\HAhb{z\in Z_{\HAda,X}}, 
\HAbeq \label{HAP.l} 
\HAlda_{\HAda,z}(x):=-z+x/\HAda 
\HAqa x\in\HAda(z+Q_z)\cap X\;. 
\HAeeq 
Then 
$$ 
\HAgL=\HAgL(\HAda,X):=\{\,\HAlda_{\HAda,z}\ ;\ z\in Z_{\HAda,X}\,\} 
$$ 
is an atlas for~$X$ of multiplicity~$2^m$. Since 
\HAhb{\HAdiam\bigl(\HAda(z+Q_z)\bigr)=\sqrt{m}\HAda\leq\HAda_S} we see that 
$\HAgL_S$~is normalized and shrinkable to 
\HAhb{1/2}. Given 
\HAhb{\HAlda,\tilde\HAlda\in\HAgL_S} with 
\HAhb{\HAlda=\HAlda_{\HAda,z}} and  
\HAhb{\tilde\HAlda=\HAlda_{\HAda,\tilde z}}, 
$$ 
\tilde\HAlda\circ\HAlda^{-1}(y)=z-\tilde z+y 
\HAqa y\in\HAlda(U_\HAlda\cap U_{\tilde\HAlda})\;. 
$$ 
This shows that $\HAgL$~is uniformly regular on~$S$. Furthermore, denoting 
by~$\HApl$ the Fr\'echet derivative, 
\HAbeq \label{HAP.dg} 
\HApl\HAlda^{-1}=\HAda 1_m 
\HAqb \HAlda_*g_X=\HAda^2g_m 
\HAqa \HAlda\in\HAgL_S\;, 
\HAeeq 
where 
\HAhb{g_X=\HAia_X^*g_m} and $1_m$~is the identity 
in~$\HABR^{m\times m}$. In particular, setting 
\HAhb{X:=\HABX} it follows that  
\HAbeq \label{HAP.RH} 
\HABR^m\text{ and }\HABH^m\text{ are uniformly regular Riemannian 
manifolds}\;. 
\HAeeq 

\par 
Let $M$ be an \hbox{$m$-dimensional} manifold and 
\HAhb{S\HAis M}. Suppose $\HAgK$~is an atlas for~$M$ which is uniformly 
regular on~$S$. Then there exists 
\HAhb{r\in(0,1)} such that (\ref{HAN.S}) is a cover of~$S$. Given 
\HAhb{\HAka\in\HAgK_S}, we fix 
\HAhb{\HAda\in\bigl(0,(1-r)/\sqrt{m}\bigr)} and put 
\HAhb{\HAgL_\HAka:=\HAgL(\HAda,Q_\HAka^m)}. By the above $\HAgL_\HAka$~is 
an atlas for~$Q_\HAka^m$ of multiplicity~$2^m$ which is uniformly regular 
on~$rQ_\HAka^m$ and shrinkable to 
\HAhb{1/2} on~$rQ_\HAka^m$. Hence 
\HAbeq \label{HAP.M} 
\HAgM=\HAgM(\HAda,\HAgK) 
:=\{\,\HAlda\circ\HAka\ ;\ \HAka\in\HAgK_S,\ \HAlda\in\HAgL_\HAka\} 
\cup(\HAgK\HAssm\HAgK_S) 
\HAeeq 
is an atlas for~$M$ such that 
\HAbeq \label{HAP.Ul} 
U_{\HAlda\circ\HAka}=\HAka^{-1}(U_\HAlda)\HAis U_{\HAcoU\HAka} 
\HAqa \HAka\in\HAgK_S 
\HAqb \HAlda\in\HAgL_\HAka\;. 
\HAeeq 
It has multiplicity at most~$2^m\HAmult(\HAgK_S)$ on~$S$ and is 
shrinkable to 
\HAhb{1/2} on~$S$. For 
\HAhb{\umu,\tilde\umu\in\HAgM_S} with 
\HAhb{\umu=\HAlda\circ\HAka} and  
\HAhb{\tilde\umu=\tilde\HAlda\circ\tilde\HAka} we get from (\ref{HAP.l}) 
and~(\ref{HAP.dg}) 
\HAbeq \label{HAP.dd} 
\|\HApa(\tilde\umu\circ\umu^{-1}\|_\HAiy 
\leq\HAda^{-1}\HAda^{|\HAal|}\,\|\HApa(\tilde\HAka\circ\HAka^{-1})\|_\HAiy 
\HAqa \HAal\in\HABN^m\HAssm\{0\}\;. 
\HAeeq 
Note that 
\HAhb{\HAlda\circ\HAka\in\HAgM_S} implies 
\HAhb{\HAka\in\HAgK_S}. Thus, since $\HAgK_S$~is uniformly regular and 
\HAhb{\HAda\leq1}, 
$$ 
\|\HApa(\tilde\umu\circ\umu^{-1})\|_\HAiy\leq c(\HAal) 
$$ 
for 
\HAhb{\umu,\tilde\umu\in\HAgM_S} with 
\HAhb{\umu=\HAlda\circ\HAka} and  
\HAhb{\tilde\umu=\tilde\HAlda\circ\tilde\HAka} and 
\HAhb{\HAal\in\HABN^m\HAssm\{0\}}. Hence 
\HAbeq \label{HAP.Mu} 
\HAgM\text{ is uniformly regular on }S\;. 
\HAeeq 

\par 
Let $g$ be a Riemannian metric for~$M$. Then (\ref{HAP.dg}) implies 
\HAbeq \label{HAP.mg} 
\umu_*g=\HAlda_*\HAka_*g=\HAda^2\HAka_*g 
\HAqa \umu=\HAlda\circ\HAka\in\HAgM_S\;. 
\HAeeq 
Consequently, 
\HAbeq \label{HAP.dag} 
\|\HApa(\umu_*g)\|_\HAiy\leq c(\HAal)\HAda^2\,\|\HApa(\HAka_*g)\|_\HAiy 
\HAqa \umu\in\HAgM_S 
\HAqb \HAal\in\HABN^m\;.  
\HAeeq 

\par 
Suppose $\HArgK$ is a singularity datum for~$M$ on~$S$. Then we infer from 
(\ref{HAN.sd})(iii) and~(iv) and from~(\ref{HAP.Ul}) 
\HAbeq \label{HAP.mr} 
\umu_*\rho=(\HAka_*\rho)\circ\HAlda^{-1}\sim(\HAka_*\rho)(0) 
=\rho_\HAka\sim\rho_\umu 
\HAeeq 
and, using 
\HAhb{\HAda\leq1} once more,  
\HAbeq \label{HAP.dmr} 
\|\HApa(\umu_*\rho)\|_\HAiy\leq\HAda^{|\HAal|}\, \|\HApa(\umu_*\rho)\|_\HAiy 
\leq c(\HAal)\rho_\HAka\leq c(\HAal)\rho_\umu 
\HAnpb 
\HAeeq 
for 
\HAhb{\umu=\HAlda\circ\HAka\in\HAgM_S} and 
\HAhb{\HAal\in\HABN^m}. 

\par 
These considerations show, in particular, that a uniformly regular 
Riemannian manifold possesses a uniformly regular atlas consisting of 
arbitrarily small charts; also see Lemma~\ref{HAlem-P.V}. 

\par 
Let $\HAtMtg$ be a Riemannian manifold without boundary. Then we endow the 
product manifold 
\HAhb{M\times\tilde M} with the product metric, denoted (slightly loosely) 
by 
\HAhb{g+\tilde g}. 
\begin{theorem}\label{HAthm-P.MM} 
Suppose $\rho$~is a bounded singularity function for $\HAMg$ on 
\HAhb{S\HAis M} and $\tilde\rho$~is one for $\HAtMtg$ on 
\HAhb{\tilde S\HAis\tilde M}. Then 
\HAhb{\rho\otimes\tilde\rho} is a singularity function for 
\HAhb{(M\times\tilde M,g+\tilde g)} on 
\HAhb{S\times\tilde S}. 
\end{theorem}
\begin{proof} 
(1) 
We choose 
\HAhb{0<\bar r<r<1} and an atlas~$\HAgK$ for~$M$, resp.\ $\tilde\HAgK$ 
for~$\tilde M$, such that~$\HAgK$, resp.~$\tilde\HAgK$, 
is shrinkable to~$\bar r$ 
on~$S$, resp.~$\tilde S$, and $\HArgK$, resp.~$\HAtrtK$, is a 
singularity datum 
for $\HAMg$ on~$S$, resp.\ $\HAtMtg$ on~$\tilde S$. Denoting by~$m$, 
resp.~$\tilde m$, the dimension of~$M$, resp.~$\tilde M$, we set 
\HAhb{\HAda:=(1-r)\big/\sqrt{m+\tilde m}}. Given 
\HAhb{\HAka\in\HAgK_S} and 
\HAhb{\tilde\HAka\in\HAgK_{\tilde S}}, we put 
\HAbeq \label{HAP.dk} 
\HAda_{\tilde\HAka}:=\min\{\tilde\rho_{\tilde\HAka},\HAda\} 
\HAqb \tilde\HAda_\HAka:=\min\{\rho_\HAka,\HAda\} \;. 
\HAeeq 
We set 
$$ 
\HAgM':=\bigl\{\,(\HAlda\circ\HAka)\times(\tilde\HAlda\circ\tilde\HAka) 
\ ;\ \HAka\in\HAgK_S,\ \tilde\HAka\in\tilde\HAgK_{\tilde S}, 
\ \HAlda\in\HAgL(\HAda_{\tilde\HAka},Q_\HAka^m), 
\ \tilde\HAlda\in\HAgL(\tilde\HAda_\HAka,Q^{\tilde m})\,\bigr\} 
$$ 
and 
$$ 
\HAgM'':=\bigl\{\,\HAka\times\tilde\HAka 
\ ;\ \text{either }\HAka\in\HAgK\HAssm\HAgK_S 
\text{ or }\tilde\HAka\in\tilde\HAgK\HAssm\tilde\HAgK_{\tilde S}\,\bigr\}\;. 
$$ 
Then 
\HAhb{\HAgM:=\HAgM'\cup\HAgM''} is an atlas for 
\HAhb{M\times\tilde M} and a refinement of the product atlas 
\HAhb{\HAgK\otimes\tilde\HAgK} in the sense that for each 
\HAhb{\umu\in\HAgM} there exists 
\HAhb{\HAka\times\tilde\HAka\in\HAgK\otimes\tilde\HAgK} such that 
\HAhb{U_\umu\HAis U_{\HAcoU\HAka\times\tilde\HAka}}. Moreover, 
\HAbeq \label{HAP.MM} 
\HAgM_{S\times\tilde S}\HAis\HAgM'\;. 
\HAnpb 
\HAeeq 
Note that $\HAgM$~is normalized on 
\HAhb{S\times\tilde S} and has finite multiplicity thereon. 

\par 
Suppose 
\HAhb{\umu_i=(\HAlda_i\circ\HAka_i) 
   \times(\tilde\HAlda_i\circ\tilde\HAka_i)\in\HAgM'} for 
\HAhb{i=1,2}, and 
\HAhb{U_{\umu_1}\cap U_{\umu_2}\neq\HAes}. Then both 
\HAhb{U_{\HAka_1}\cap U_{\HAka_2}} and 
\HAhb{U_{\tilde\HAka_1}\cap U_{\tilde\HAka_2}} are nonempty. Hence 
\HAhb{\tilde\rho_{\tilde\HAka_1}\sim\tilde\rho_{\tilde\HAka_2}} and 
\HAhb{\rho_{\HAka_1}\sim\rho_{\HAka_2}}. From this, 
\HAhb{\HAda_{\tilde\HAka_i}\leq\tilde\rho_{\tilde\HAka_i}}, and the 
boundedness of~$\tilde\rho$ we infer 
\HAhb{\HAda_{\tilde\HAka_1}/\HAda_{\tilde\HAka_2}\leq c} and, 
analogously, 
\HAhb{\tilde\HAda_{\HAka_1}/\tilde\HAda_{\HAka_2}\leq c}. Thus, using 
(\ref{HAP.l}),\, (\ref{HAP.dg}), the finite multiplicity 
of~$\HAgM_{S\times\tilde S}$ and the fact that $\HAka_i$ 
and~$\tilde\HAka_i$ are normalized, we obtain (cf.~(\ref{HAP.dd})) 
$$ 
\|\umu_1\circ\umu_2^{-1}\|_{k,\HAiy} 
\leq c\bigl(\|\HAka_1\circ\HAka_2^{-1}\|_{k,\HAiy} 
+\|\tilde\HAka_1\circ\tilde\HAka_2^{-1}\|_{k,\HAiy}\bigr)\leq c(k) 
\HAnpb 
$$ 
for 
\HAhb{\umu_1,\umu_2\in\HAgM_{S\times\tilde S}} and 
\HAhb{k\in\HABN}. This proves that $\HAgM$~is uniformly regular on 
\HAhb{S\times\tilde S}. 

\par 
(2) 
By adapting (\ref{HAP.mr}) and~(\ref{HAP.dmr}) to the present setting 
we find, due to~(\ref{HAP.MM}), 
\HAbeq \label{HAP.rr0} 
\umu_*(\rho\otimes\tilde\rho) 
\sim(\rho\otimes\tilde\rho)_\umu 
\sim\rho_\HAka\tilde\rho_{\tilde\HAka} 
\HAeeq 
for 
\HAhb{\umu=(\HAlda\circ\HAka)\times(\tilde\HAlda\circ\tilde\HAka) 
   \in\HAgM_{S\times\tilde S}} and 
$$ 
\|\umu_*(\rho\otimes\tilde\rho)\|_{k,\HAiy} 
\leq c(k)(\rho\otimes\tilde\rho)_\umu 
\HAqa \umu\in\HAgM_{S\times\tilde S} 
\HAqb k\in\HABN\;. 
$$ 

\par 
(3) 
For 
\HAhb{\umu=(\HAlda\circ\HAka)\times(\tilde\HAlda\circ\tilde\HAka)\in\HAgM'} 
we find by~(\ref{HAP.mg}) 
\HAbeq \label{HAP.mgg} 
\HAbal 
\umu_*(g+\tilde g) 
&=(\HAlda\circ\HAka)_*g+(\tilde\HAlda\circ\tilde\HAka)_*\tilde g\\ 
&\sim\HAda_{\tilde\HAka}^2\HAka_*g+\tilde\HAda_\HAka^2\tilde\HAka_*\tilde g 
 \sim\HAda_{\tilde\HAka}^2\rho_\HAka^2g_m 
 +\tilde\HAda_\HAka^2\tilde\rho_{\tilde\HAka}^2g_{\tilde m}\;, 
\HAeal 
\HAeeq 
uniformly with respect to 
\HAhb{\umu\in\HAgM_{S\times\tilde S}}. Definition~(\ref{HAP.dk}) and 
the boundedness of $\rho$ and~$\tilde\rho$ imply 
\HAhb{\HAda_{\tilde\HAka}\sim\tilde\rho_{\tilde\HAka}} and 
\HAhb{\tilde\HAda_\HAka\sim\rho_\HAka}. Using this and~(\ref{HAP.rr0}) 
we get from~(\ref{HAP.mgg}) 
$$ 
\umu_*(g+\tilde g) 
\sim\rho_\HAka^2\tilde\rho_{\tilde\HAka}^2(g_m+g_{\tilde m}) 
\sim(\rho\otimes\tilde\rho)_\umu^2g_{m+\tilde m} 
\HAqa \umu\in\HAgM_{S\times\tilde S}\;. 
$$ 
Lastly, we infer from (\ref{HAP.dag}) and~(\ref{HAP.rr0}) 
$$ 
\HAbal 
\|\umu_*(g+\tilde g)\|_{k,\HAiy} 
&\leq c(k)\bigl(\HAda_{\tilde\HAka}^2\,\|\HAka_*g\|_{k,\HAiy} 
 +\tilde\HAda_\HAka^2\,\|\tilde\HAka_*\tilde g\|_{k,\HAiy}\bigr)\\ 
&\leq c(k)\rho_\HAka^2\tilde\rho_{\tilde\HAka}^2 
 \leq c(k)(\rho\otimes\tilde\rho)_\umu^2 
\HAeal 
\HAnpb 
$$ 
for 
\HAhb{\umu\in\HAgM_{S\times\tilde S}} and 
\HAhb{k\in\HABN}. This proves the assertion. 
\qed 
\end{proof} 
Our next considerations exploit the `localization to~$S$'. 
\begin{lemma}\label{HAlem-P.V} 
Let $\HAMg$ be uniformly regular on 
\HAhb{S\HAis M}. Suppose $V$~is open in~$M$ and 
\HAbeq \label{HAP.dgS} 
d_{\bar V}(S,M\HAssm V)>0\;. 
\HAeeq 
Then there exists an atlas~$\HAgM$ for~$M$ belonging to the structure 
of uniform regularity on~$S$ such that 
\HAhb{U_\umu\HAis V} for 
\HAhb{\umu\in\HAgM_S}. 
\end{lemma}
\begin{proof} 
Let $\HAgK$ be an atlas belonging to the structure of uniform regularity 
on~$S$. Choose 
\HAhb{r\in(0,1)} such that (\ref{HAN.S})~is a cover of~$S$. Fix 
\HAhb{\HAda\in\bigl(0,(1-r)/\sqrt{m}\bigr)} and set 
\HAhb{\HAgM:=\HAgM(\HAda,\HAgK)}. Then $\HAgM_S$~is uniformly regular 
by~(\ref{HAP.Mu}). 

\par 
It follows from 
\HAhb{\HAka_*g\sim g_m} for 
\HAhb{\HAka\in\HAgK_S},\, (\ref{HAP.mg}), and (\ref{HAP.dgS}) that we can 
choose~$\HAda$ so small that 
\HAhb{\HAdiam(U_\umu)<d_{\bar V}(S,M\HAssm V)} for 
\HAhb{\umu\in\HAgM_S}. 

\par 
Lastly, we infer from (\ref{HAN.KS}),\, (\ref{HAP.l}), and~(\ref{HAP.dg}) 
that 
$$ 
\|\HAka\circ\umu^{-1}\|_{k,\HAiy}+\|\umu\circ\HAka^{-1}\|_{k,\HAiy} 
\leq c(k) 
\HAqa \HAka\in\HAgK_S 
\HAqb \umu\in\HAgM_S\;. 
\HAnpb 
$$ 
Thus 
\HAhb{\smash{\HAgM\HAapproxS\HAgK}\HAsmtB}, 
which proves the claim. 
\qed 
\end{proof} 
The following lemma will be fundamental for the construction of 
singular Riemannian manifolds by `patching together simpler pieces'. 
\begin{lemma}\label{HAlem-P.P} 
\setlength{\svitemindent}{1.5\parindent}
Suppose: 
\begin{enumerate}  
\item[{\rm(i)}] 
${}$
\HAhb{\{\,V_\HAal\ ;\ \HAal\in\HAsA\,\}} is a finite family of open subsets 
of~$M$. 
\item[{\rm(ii)}] 
${}$
\HAhb{S_\HAal\HAis V_\HAal} and 
\HAhb{\{\,S_\HAal\ ;\ \HAal\in\HAsA\,\}} is a covering of~$M$. 
\item[{\rm(iii)}] 
${}$
$(\rho_\HAal,\HAgK_\HAal)$ is a singularity datum for~$(V_\HAal,g)$ 
on~$S_\HAal$. 
\item[{\rm(iv)}] 
${}$
\HAhb{\rho_\HAal\HAsn V_\HAal\cap V_{\tilde\HAal} 
   \sim\rho_{\tilde\HAal}\HAsn V_\HAal\cap V_{\tilde\HAal}}, 
\ \HAhb{\HAal,\tilde\HAal\in\HAsA}. 
\item[{\rm(v)}] 
${}$
\HAhb{\|\HAka_{\tilde\HAal}\circ\HAka_\HAal^{-1}\|_{k,\HAiy} 
   +\|\HAka_\HAal\circ\HAka_{\tilde\HAal}^{-1}\|_{k,\HAiy}\leq c(k)}
   \text{ for}\\ 
\HAhb{(\HAka_\HAal,\HAka_{\tilde\HAal}) 
     \in\HAgK_{\HAal,S_\HAal}\times\HAgK_{\tilde\HAal,S_{\tilde\HAal}}},
\ \HAhb{\HAal,\tilde\HAal\in\HAsA}, 
\ \HAhb{\HAal\neq\tilde\HAal}, 
\ \HAhb{k\in\HABN}. 
\end{enumerate} 
Then 
\HAhb{\HAgK:=\bigcup_\HAal\HAgK_{\HAal,S_\HAal}} is a uniformly regular atlas 
for~$M$ and there exists~$\rho$ belonging to  
\HAhb{C^\HAiy\bigl(M,(0,\HAiy)\bigr)} and satisfying 
\HAbeq \label{HAP.rho} 
\rho\HAsn S_\HAal\sim\rho_\HAal 
\HAqa \HAal\in\HAsA\;, 
\HAnpb 
\HAeeq 
such that $\HArgK$ is a singularity datum for~$\HAMg$.  
\end{lemma} 
\begin{proof} 
(1) 
It is a consequence of \hbox{(i)--(iii)} and~(v) that $\HAgK$~is 
a uniformly regular atlas for~$M$. 

\par 
(2) 
Since $M$~is locally compact, separable, and metrizable the same 
applies to~$V_\HAal$. Thus $V_\HAal$~is paracompact. Hence there exists a 
smooth partition of unity 
\HAhb{\{\,\chi_{\HAal,\HAba}\ ;\ \HAba\in\HAgK_\HAal\,\}} on~$V_\HAal$ 
subordinate to 
\HAhb{\{\,U_\HAba\ ;\ \HAba\in\HAgK_\HAal\,\}} (e.g.,~\cite{Con01a}). 
We extend each $\chi_{\HAal,\HAba}$ over~$M$ by setting it 
equal to~$0$ outside~$V_\HAal$ and set 
\HAhb{\psi_\HAal:=\sum_{\HAba\in\HAgK_{\HAal,S_\HAal}}\chi_{\HAal,\HAba}}. 
Then 
\HAhb{\psi_\HAal\in C^\HAiy\bigl(M,[0,1]\bigr)} with 
\HAhb{\psi_\HAal\HAsn S_\HAal=\HAmf{1}}. We put 
\HAhb{\HAvp_\HAal 
     :=\psi_\HAal\big/\sum_{\tilde\HAal\in\HAsA}\psi_{\tilde\HAal}}. 
Assumptions (i) and~(ii) guarantee that 
\HAhb{\{\,\HAvp_\HAal\ ;\ \HAal\in\HAsA\,\}} is a smooth partition of unity 
on~$M$ subordinate to the open cover 
\HAhb{\{\,V_\HAal\ ;\ \HAal\in\HAsA\,\}} of~$M$. 

\par 
We put 
\HAhb{\rho:=\sum_\HAal\HAvp_\HAal\rho_\HAal}. Then, given 
\HAhb{\HAal\in\HAsA} and 
\HAhb{x\in V_\HAal}, we infer from~(iv) 
\HAbeq \label{HAP.rph} 
\HAbal 
\rho(x)=\sum_{V_\HAba\cap V_\HAal\neq\HAes}\HAvp_\HAba(x)\rho_\HAba(x) 
&\sim\rho_\HAal(x)\sum_{V_\HAba\cap V_\HAal\neq\HAes}\HAvp_\HAba(x)\\ 
&=\rho_\HAal(x)\sum_\HAba\HAvp_\HAba(x)=\rho_\HAal(x)\;. 
\HAeal 
\HAnpb 
\HAeeq 
This proves~(\ref{HAP.rho}). 

\par 
(3) 
By (iii) 
$$ 
\HAka_*(g/\rho_\HAal^2)\sim g_m 
\HAqb \|\HAka_*(g/\rho_\HAal^2)\|_{k,\HAiy}\leq c(k) 
$$ 
for 
\HAhb{\HAka\in\HAgK_{\HAal,S_\HAal}}, 
\ \HAhb{ \HAal\in\HAsA}, and  
\HAhb{k\in\HABN}. We deduce from~(\ref{HAP.rph}) 
\HAbeq \label{HAP.kgg} 
\HAka_*(g/\rho^2) 
=\HAka_*g/\HAka_*\rho^2 
\sim\HAka_*g/\HAka_*\rho_\HAal^2 
=\HAka_*(g/\rho_\HAal^2) 
\sim g_m 
\HAeeq 
for 
\HAhb{\HAka\in\HAgK_{\HAal,S_\HAal}} and 
\HAhb{\HAal\in\HAsA}, that is, for 
\HAhb{\HAka\in\HAgK}. The definition of~$\rho$ implies 
$$ 
\HAka_*\rho 
=\sum_\HAal(\HAka_*\HAvp_\HAal)(\HAka\circ\HAka_\HAal^{-1})_* 
 \HAka_{\HAal*}\rho_\HAal 
\HAqa \HAka\in\HAgK\;. 
$$ 
From this, (iii), the chain rule, and the uniform regularity 
of~$\HAgK$ we deduce the estimate 
\HAhb{\|\HAka_*\rho\|_{k,\HAiy}\leq c(k)} for 
\HAhb{\HAka\in\HAgK} and 
\HAhb{k\in\HABN}. Consequently, we infer from the chain rule 
and~(\ref{HAP.kgg}) 
$$ 
\|\HAka_*(g/\rho^2)\|_{k,\HAiy}\leq c(k) 
\HAqa \HAka\in\HAgK 
\HAqb k\in\HABN\;. 
\HAnpb 
$$ 
This proves the last part of the assertion. 
\qed 
\end{proof} 
The following (almost trivial) lemma shows that the class of singular 
manifolds is invariant under Riemannian isometries. 
\begin{lemma}\label{HAlem-P.f} 
Let 
\HAhb{f\HAsco\tilde M\HAra M} be a diffeomorphism of manifolds. Suppose 
$g$~is a Riemannian metric for~$M$ and $\rho$~is a singularity function 
for~$\HAMg$ on 
\HAhb{S\HAis M}. Then $f^*\rho$~is a singularity function for 
$(\tilde M,f^*g)$ on~$f^{-1}(S)$. 
\end{lemma} 
\begin{proof} 
Let $\HAgK$ be an atlas which is uniformly regular on~$S$. It is easily 
verified that 
\HAhb{f^*\HAgK:=\{\,f^*\HAka\ ;\ \HAka\in\HAgK\}} is an atlas 
for~$\tilde M$ which is uniformly regular on~$f^{-1}(S)$. Note 
$$ 
(f^*\HAka)_*f^*\rho=(\rho\circ f)\circ(\HAka\circ f)^{-1} 
=\rho\circ\HAka=\HAka_*\rho 
$$ 
and 
$$ 
(f^*\HAka)_*(f^*g)=(\HAka\circ f)_*(f^{-1})_*g 
=\bigl((\HAka\circ f)\circ f^{-1}\bigr)_*g=\HAka_*g 
$$  
for 
\HAhb{\HAka\in\HAgK}. From this it is obvious that conditions~(\ref{HAN.sd}) 
carry over from $\rho$, $\HAgK$, and~$g$ to $f^*\rho$, 
$f^*\HAgK$, and~$f^*g$. 
\qed 
\end{proof} 
Suppose 
\HAhb{\HApl M\neq\HAes} and let 
\HAhb{\HAthia\HAsco\HApl M\HAhr M} be the natural embedding. Let $g$ be 
a Riemannian metric for~$M$. Then 
\HAhb{\HAthg:=\HAthia^*g} is the Riemannian metric for~$\HApl M$ 
induced by~$g$. Given a local chart~$\HAka$ for~$M$ with 
\HAhb{\HApl U_{\HAcoU\HAka}=U_{\HAcoU\HAka}\cap\HApl M\neq\HAes}, we set 
\HAhb{U_{\HAithka}:=\HApl U_{\HAcoU\HAka}} and 
\HAhb{\HAthka 
     :=\HAia_0\circ(\HAia^*\HAka)\HAsco U_{\HAithka}\HAra\HABR^{m-1}}, 
where 
\HAhb{\HAia_0\HAsco\{0\}\times\HABR^{m-1}\HAra\HABR^{m-1}}, 
\ \HAhb{(0,x')\HAmt x'}. Moreover, 
\HAhb{\HAthrho:=\HAia^*\rho=\rho\HAsn\HApl M} for 
\HAhb{\rho\HAsco M\HAra\HABR}. 
\begin{lemma}\label{HAlem-P.dM} 
Let $\HAgK$ be an atlas for~$M$  which is uniformly regular on~$S$. Then 
$$ 
\HAthgK:=\{\,\HAthka\ ;\ \HAka\in\HAgK_{\HApl M}\,\} 
$$  
is one for~$\HApl M$ and it is uniformly regular on 
\HAhb{\HApl M\cap S}. If $\HArgK$~is a singularity datum for~$\HAMg$ 
on~$S$, then 
$(\HAthrho,\HAthgK)$ is one for $(\HApl M,\HAthg)$ on 
\HAhb{\HApl M\cap S}. 
\end{lemma}
\begin{proof} 
Obvious. 
\qed 
\end{proof} 
In this lemma it is implicitly assumed that 
\HAhb{m\geq2}. However, calling---in abuse of language---every 
\hbox{$0$-dimensional} manifold uniformly regular, Lemma~\ref{HAlem-P.dM} 
holds for 
\HAhb{m=1} also, employing obvious interpretations and adaptions. 
\section{Uniformly Regular Riemannian Manifolds}\label{HAsec-U}%
On the basis of the preceding considerations we now provide proofs for 
some of the claims made in Example~\ref{HAexa-I.ex}.  

\par 
Let $\HAMg$ be a Riemannian manifold. It has \emph{bounded geometry} if it 
has an empty boundary, is complete, has a positive injectivity radius, and 
all covariant derivatives of the curvature tensor are bounded. 
\begin{theorem}\label{HAthm-U.B} 
If $\HAMg$ has bounded geometry, then it is uniformly regular. 
\end{theorem}
\begin{proof}
This follows from Th.~Aubin~\cite[Lemma~2.2.6]{Aub82a} and 
J.~Eichhorn~\cite{Eich91a} (also see M.A.~Shubin~\cite{Shu92a}). 
\qed
\end{proof} 
A~uniformly regular Riemannian manifold without boundary is complete 
(cf.\ M.~Disconzi, Y.~Shao, and G.~Simonett~\cite{DSS14a}). 
It has been shown by R.E.~Greene~\cite{Gree78a} that every manifold~$M$ 
without boundary admits a Riemannian metric~$g$ such that $\HAMg$~has 
bounded geometry. However, in view of applications to differential equations 
which we have in mind, this 
result is of restricted interest, in general. Indeed, the metric is 
then given a~priori and is closely related to the differential 
operators under consideration. 

\par 
Although Theorem~\ref{HAthm-U.B} is very general it has the disadvantage that 
it applies only to manifolds without boundary. The following results do not 
require $\HApl M$ to be empty. 
\begin{lemma}\label{HAlem-U.C} 
Let $\HAMg$ be a Riemannian manifold and suppose 
\HAhb{S\HAis M} is compact. Then there exists  a unique uniformly regular 
structure for~$M$ on~$S$, and $\HAMg$~is uniformly regular on~$S$. 
\end{lemma}
\begin{proof}
(1) 
For each 
\HAhb{p\in M} there exists a local chart~$\tilde\HAka_p$ of~$M$ with 
\HAhb{p\in U_{\HAcoU\tilde\HAka}}. We set 
\HAhb{W_{\HAcoW p}:=Q^m} if 
\HAhb{p\in\HAci M}, and 
\HAhb{W_{\HAcoW p}:=Q^m\cap\HABH^m} for 
\HAhb{p\in\HApl M}. Then we can fix 
\HAhb{\HAda_p>0} such that 
\HAhb{\HAol{\tilde\HAka_p(p)+\HAda_pW_{\HAcoW p}} 
     \HAis\HAka_p(U_{\HAcoU\HAka_p})}. From this 
it follows that, by translation and dilation, we find for each pair 
\HAhb{p,q\in M} local charts $\HAka_p$ and~$\HAka_q$, normalized at $p$ 
and~$q$, respectively, such that 
\HAhb{\|\HAka_p\circ\HAka_q^{-1}\|_{k,\HAiy}\leq c(p,q,k)} for 
\HAhb{k\in\HABN}. 

\par 
By the compactness of~$S$ we can determine a finite subset~$\HASa$ of~$S$ 
such that 
\HAhb{\bigl\{\,\HAka_p^{-1}(2^{-1}Q_{\HAka_p}^m)\ ;\ p\in\HASa\,\}} is an 
open cover of~$S$. Let $\HAgN$ be an atlas for the open submanifold 
\HAhb{M\HAssm S} of~$M$. Then 
$$ 
\HAgK:=\{\,\HAka_p\ ;\ p\in\HASa\,\}\cup\HAgN 
$$ 
is an atlas for~$M$, and 
\HAhb{\HAgK_S=\{\,\HAka_p\ ;\ p\in\HASa\,\}}. Since $\HASa$~is finite 
$\HAgK$~is uniformly regular on~$S$ and (cf.~(\ref{HAN.kg})) 
condition~(\ref{HAN.Rr}) is satisfied. 

\par 
(2) Let $\HAgL$ be an atlas for ~$M$ which is uniformly regular on~$S$. 
By the compactness of~$S$ we find a subatlas~$\HAgM$ of~$\HAgL$ such that 
$\HAgM_S$~is a finite subset of~$\HAgL_S$. It is obvious that $\HAgM$~can 
be chosen such that 
\HAhb{\smash{\HAgM\HAapproxS\HAgL}\HAsmtB}. Since $\HAgK_S$ and~$\HAgM_S$ 
are both finite, 
\HAhb{\smash{\HAgM\HAapproxS\HAgK}\HAsmtB}. Consequently, 
\HAhb{\smash{\HAgL\HAapproxS\HAgK}\HAsmtB}. This proves the uniqueness 
assertion. 
\qed
\end{proof}
\begin{corollary}\label{HAcor-U.C} 
Every compact Riemannian manifold is uniformly regular. 
\end{corollary}
The next theorem concerns submanifolds of codimension~$0$ of uniformly 
regular Riemannian manifolds. 
\begin{theorem}\label{HAthm-U.S} 
Let $\HANg$ be an \hbox{$m$-dimensional} uniformly regular Riemannian 
manifold and $\HAMg$ an \hbox{$m$-dimensional} Riemannian submanifold with 
compact boundary. Then $\HAMg$ is uniformly regular. 
\end{theorem}
\begin{proof}
By the preceding corollary we can assume 
\HAhb{\HApl M\neq\HAes}. 

\par 
Since $M$~is locally compact and $\HApl M$~is compact 
there exist relatively compact open neighborhoods $W_1$ and~$W_2$ 
of~$\HApl M$ in~$M$ with 
\HAhb{W_1\HAis\bar W_1\HAis W_2}. We set 
\HAhb{V_1:=W_2} and 
\HAhb{S_1:=\bar W_1} as well as 
\HAhb{V_2:=\HAci M} and 
\HAhb{S_2:=M\HAssm W_1}. Then $V_i$~is open in~$M$, 
\ \HAhb{S_i\HAis V_i}, and 
\HAhb{S_1\cup S_2=M}. 

\par 
The compactness of~$S_1$ in~$M$ and 
\HAhb{\D_M(S_1,M\HAssm W_2)>0} imply, due to Lemmas \ref{HAlem-P.V} 
and~\ref{HAlem-U.C}, that there exists an atlas~$\HAgK_1$ for~$M$ such that 
$(\HAmf{1},\HAgK_1)$ is a singularity datum for~$V_1$ on~$S_1$. 

\par 
Note that 
\HAhb{\D_M(S_2,\HApl M)>0}. Hence Lemma~\ref{HAlem-P.V} and the uniform 
regularity of~$\HANg$ imply the existence of an atlas~$\HAgK_2$ 
for~$\HAci M$ such that 
$(\HAmf{1},\HAgK_2)$ is a singularity datum for~$V_2$ on~$S_2$. 

\par 
Since 
\HAhb{S:=S_1\cap S_2=\bar W_1\HAssm W_1} is compact we can assume that 
$\HAgK_{1,S}$ and~$\HAgK_{2,S}$ are finite. Hence it is obvious that 
condition~(v) of Lemma~\ref{HAlem-P.P} is satisfied. Thus that lemma 
guarantees the validity of the claim. 
\qed
\end{proof}
\begin{corollary}\label{HAcor-U.S} 
Let $M$ be an \hbox{$m$-dimensional} Euclidean submanifold of~$\HABR^m$ with 
compact boundary. Then $M$~is a uniformly regular Riemannian manifold.
\end{corollary}
\begin{proof}
Set 
\HAhb{N:=\HABR^m} and recall~(\ref{HAP.RH}). 
\qed
\end{proof}
\section{Characteristics}\label{HAsec-C}%
We write 
\HAhb{J_0:=(0,1]}, 
\ \HAhb{J_\HAiy:=[1,\HAiy)}, and assume throughout that 
\HAhb{J\in\{J_0,J_\HAiy\}}. A~sub\-in\-ter\-val~$I$ of~$J$ is \emph{cofinal} 
if 
\HAhb{1\notin I}, and 
\HAhb{J\HAssm\HAci I} is a compact interval. 

\par 
We denote by~$\HAcR(J)$ the set of all 
\HAhb{R\in C^\HAiy\bigl(J,(0,\HAiy)\bigr)} satisfying 
\HAhb{R(1)=1}, such that 
\HAhb{R(\HAom):=\lim_{t\HAra\HAom}R(t)} exists in~$[0,\HAiy]$ if 
\HAhb{J=J_\HAom}. Then we write 
\HAhb{R\in\HAcC(J)} if 
\HAbeq \label{HAC.C} 
 \HAbal 
 \\ 
 \noalign{\vskip-5\jot}  
 \rm{(i)} \qquad    &R\in\HAcR(J)\text{ and }R(\HAiy)=0\text{ if } 
                     J=J_\HAiy\;;\\
 \noalign{\vskip-1\jot} 
 \rm{(ii)} \qquad   &\int_J\,\D t\big/R(t)=\HAiy\;;\\     
 \noalign{\vskip-1\jot} 
 \rm{(iii)} \qquad  &\|\HApl^kR\|_\HAiy<\HAiy\;, 
                     \ \ k\geq1\;.
 \HAeal 
\HAnpb 
\HAeeq 
The elements of~$\HAcC(J)$ are called \emph{cusp characteristics} on~$J$. 

\par 
On~$J_\HAiy$ we introduce, in addition, the set~$\HAcF(J_\HAiy)$ of 
 \emph{funnel characteristics} by: 
\HAhb{R\in\HAcF(J_\HAiy)} if 
\HAbeq \label{HAC.F} 
 \HAbal
 \rm{(i)} \qquad    &R\in\HAcR(J_\HAiy)\text{ and }R(\HAiy)>0\;;\\
 \noalign{\vskip-1\jot} 
 \rm{(ii)} \qquad  &\|\HApl^kR\|_\HAiy<\HAiy\;, 
                     \ \ k\geq1\;.
 \HAeal
\HAeeq 
\begin{HAexamples}\label{HAexa-C.ex} 
\ (a) We set 
\HAhb{\HAsR_\HAal(t):=t^\HAal} for 
\HAhb{\HAal\in\HABR}. Then\\ 
\HAhb{\HAsR_\HAal\in\HAcC(J_0)} if 
\HAhb{\HAal\geq1}, 
 \HAhb{\HAsR_\HAal\in\HAcC(J_\HAiy)} if 
\HAhb{\HAal<0}, and 
\HAhb{\HAsR_\HAal\in\HAcF(J_\HAiy)} if 
\HAhb{0\leq\HAal\leq1}. 

\par 
(b) Suppose 
\HAhb{\HAba>0} and 
\HAhb{\HAga\in\HABR}. Put 
\HAhb{R(t):=\E^{\HAba(1-t^\HAga)}}. Then 
\HAhb{R\in\HAcC(J_0)} if 
\HAhb{\HAga<0}, whereas 
\HAhb{R\in\HAcC(J_\HAiy)} for 
\HAhb{\HAga>0}. 

\par 
(c) For 
\HAhb{\HAal\geq-2/\upi}
and 
\HAhb{\HAba>0} we put 
\HAhb{R_{\arctan,\HAal,\HAba}(t):=1+\HAal\arctan\bigl(\HAba(t-1)\bigr)}. Then 
\HAhb{R_{\arctan,-2/\upi,\HAba}\in\HAcC(J_\HAiy)} and 
\HAhb{R_{\arctan,\HAal,\HAba}\in\HAcF(J_\HAiy)} if 
\HAhb{\HAal>-2/\upi}.  
\qed 
\end{HAexamples}  
Let 
\HAhb{R\in\HAcC(J)}, resp.\ 
\HAhb{R\in\HAcF(J_\HAiy)}. Then the \emph{\hbox{$R$-gauge} diffeomorphism} 
$$ 
\HAsa=\HAsa[R]\HAsco J\HAra\HABR^+ 
$$ 
is defined by 
$$ 
\HAsa(t):= 
 \left\{ 
 \HAbal 
 {}
 &\HAsign(t-1)\int_1^t\,\D\utau/R     
  &\quad \text{if } &&R\in\HAcC(J)\;,\\ 
 &\int_1^t\sqrt{1+\dot R^2}\,\D\utau     
  &\quad \text{if } &&R\in\HAcF(J_\HAiy)\;. 
 \HAeal 
 \right. 
$$  
Note that 
\HAhb{\HAsa(J)=\HABR^+} and 
\HAhb{\dot\HAsa(t)\neq0} for 
\HAhb{t\in J}. Hence $\HAsa$~is indeed a diffeomorphism whose inverse is 
written 
\HAhb{\utau=\utau[R]:=\HAsa^{-1}\HAsco\HABR^+\HAra J}. We define the 
\emph{\hbox{$R$-sequence}}~$(t_j)$ by 
\HAhb{t_j=t_j[R]:=\utau(j)} for 
\HAhb{j\in\HABN}. Then $(t_j)$~is strictly increasing to~$\HAiy$ if 
\HAhb{J=J_\HAiy}, whereas it strictly decreases to~$0$ otherwise. For 
\HAhb{k\ge1} we put 
$$ 
I_k=I_k[R]:= 
 \left\{ 
 \HAbal 
 {}
 &(0,t_k]     
  &\quad \text{if } &&J=J_0\;,\\ 
 &[k,\HAiy)         
  &\quad \text{if } &&J=J_\HAiy\;. 
 \HAeal 
 \right. 
 \HAnpb 
$$ 
Thus $I_k$~is a cofinal interval of~$J$. 
\begin{lemma}\label{HAlem-C.CF} 
Suppose 
\HAhb{R\in\HAcC(J)} or\footnote{More precisely: 
\HAhb{J=J_\HAiy} and 
\HAhb{R\in\HAcF(J)}.} 
\HAhb{R\in\HAcF(J_\HAiy)}. Set 
\HAbeq \label{HAC.r0} 
r=r[R]:= 
 \left\{ 
 \HAbal 
 {}
 &R     
  &\quad \text{if } &R\in\HAcC(J)\;,\\ 
 &\HAmf{1}        
  &\quad \text{if } &R\in\HAcF(J_\HAiy)\;.
 \HAeal 
 \right. 
 \HAnpb 
\HAeeq 
Then $r$~is a singularity function for $(\HAci J,\D t^2)$ on~$I_2$. 
\end{lemma} 
\begin{proof} 
(1) 
We set 
$$ 
J_j=J_j[R]:= 
 \left\{ 
 \HAbal 
 {}
 &(t_{j+1},t_{j-1})     
  &\quad \text{if } &&J=J_0\;,\\ 
 &(t_{j-1},t_{j+1})         
  &\quad \text{if } &&J=J_\HAiy\;. 
 \HAeal 
 \right. 
$$ 
Then $J_j$~is a nonempty open subinterval of~$\HAci J$ for 
\HAhb{j\geq1}, and 
\HAhb{\{\,J_j\ ;\ j\geq1\,\}} is a covering of~$\HAci J$ of 
multiplicity~$2$. We let 
\HAbeq \label{HAC.sj} 
\HAsa_j:=\HAsa\HAsn J_j-j 
\HAqa j\geq1\;. 
\HAeeq 
Then 
\HAhb{\HAgS=\HAgS[R]:=\{\,\HAsa_j\ ;\ j\geq1\,\}} is a normalized atlas, the 
\emph{\hbox{$R$-atlas}}, for~$\HAci J$ of multiplicity~$2$ which is 
shrinkable to 
\HAhb{1/2}. Note that 
\HAhb{\utau_j=\utau_j[R]:=\HAsa_j^{-1}} satisfies 
\HAbeq \label{HAC.tj} 
\utau_j(s)=\utau(s+j) 
\HAqa s\in Q 
\HAqb j\geq1\;. 
\HAeeq 

\par 
By (\ref{HAC.sj}) and~(\ref{HAC.tj}) we see that 
\HAhb{\HAsa_j\circ\utau_k(s)=s+k-j\in Q} if 
\HAhb{s\in Q} and 
\HAhb{\utau_k(s)\in J_j}. This proves that $\HAgS$~is uniformly regular 
on~$I$. 

\par 
(2) 
We set 
\HAhb{\rho:=R\circ\utau=\utau^*R}. Then 
\HAbeq \label{HAC.dr} 
\dot\rho=(\utau^*\dot R)\dot\utau\;.  
\HAeeq 
Furthermore, 
\HAhb{\HAsa\circ\utau=\HAid} implies 
\HAbeq \label{HAC.dt} 
\dot\utau=1/\utau^*\dot\HAsa\;. 
\HAeeq 

\par 
(3) 
Assume 
\HAhb{R\in\HAcC(J)}. If 
\HAhb{J=J_0}, then 
\HAhb{R(0)=0} by~(\ref{HAC.C})(ii). Thus, for each choice of~$J$, 
\HAbeq \label{HAC.r} 
0<\rho\leq c\;. 
\HAeeq 
Since 
\HAhb{\dot\HAsa(t)=\HAsign(t-1)/R(t)} we get from~(\ref{HAC.dt}) 
\HAbeq \label{HAC.t} 
\dot\utau=\HAsign(\utau-1)\rho\;. 
\HAeeq 
Hence, by~(\ref{HAC.dr}) and setting 
\HAhb{\HAve:=\HAsign(\utau-1)}, 
\HAbeq \label{HAC.drb} 
\dot\rho=b_1\rho 
\HAqa b_1:=\HAve\utau^*\dot R\in BC\HARp\;.  
\HAeeq 
Furthermore, 
\HAbeq \label{HAC.db} 
\dot b_1=\HAve(\utau^*\ddot R)\dot\utau=(\utau^*\ddot R)\rho\in BC\HARp\;,  
\HAeeq 
due to (\ref{HAC.t}) and~(\ref{HAC.C})(iii). Consequently, we obtain 
from~(\ref{HAC.drb}) 
$$ 
\ddot\rho=b_2\rho 
\HAqa b_2:=\dot b_1+b_1^2\in BR(\HABR)\;. 
$$ 
By induction 
\HAbeq \label{HAC.dkr} 
\HApl^k\rho=b_k\rho 
\HAqa b_k:=\dot b_{k-1}+b_{k-1}b_1 
\HAqb k\geq2\;. 
\HAeeq 
Thus $b_k$~is a polynomial function in the variables 
\HAhb{b_1,\dot b_1,\ldots,\HApl^{k-1}b_1} with coefficients in~$\HABZ$. 

\par 
From \hbox{(\ref{HAC.t})--(\ref{HAC.db})} we get 
$$ 
\ddot b_1=\HAve(\utau^*\HApl^3R)\rho^2+\utau^*(\ddot R)(\utau^*\dot R)\rho\;. 
$$ 
Hence we find, once more inductively, that $\HApl^\ell b_1$~is a 
polynomial function in the variables 
\HAhb{\rho,\utau^*\HApl R,\ldots,\utau^*\HApl^{\ell+1}R} with coefficients 
in~$\HABZ$. Consequently, $b_k$~is a polynomial function in the variables 
\HAhb{\rho,\utau^*\HApl R,\ldots,\utau^*\HApl^{k+1}R}. Hence 
\HAhb{b_k\in BC(\HABR)} by (\ref{HAC.r}) and (\ref{HAC.C})(iii). 
Thus we obtain from~(\ref{HAC.dkr}) 
\HAbeq \label{HAC.dkrr} 
|\HApl^k\rho|\leq c(k)\rho 
\HAqa k\geq1\;. 
\HAeeq 

\par 
It follows from 
\HAhb{\HApl\log\rho=\dot\rho/\rho} and the last estimate that 
\HAhb{\HAba:=\|\HApl\log\rho\|_\HAiy<\HAiy}. Hence, by the 
mean-value theorem, 
$$ 
\bigl|\log\bigl(\rho(s)\big/\rho(t)\bigr)\bigr| 
=|\log \rho(s)-\log\rho(t)| 
\leq\HAba\,|s-t| 
\HAqa s,t\geq0\;.  
$$ 
This implies 
\HAhb{\E^{-\HAba}\leq\rho(s)/\rho(t)\leq\E^\HAba} for 
\HAhb{|s-t|\leq1}, that is, 
\HAbeq \label{HAC.rr} 
\rho(s)\sim\rho(t) 
\HAqa s,t\in\HABR^+ 
\HAqb |s-t|\leq1\;. 
\HAeeq 

\par 
Since 
\HAhb{\rho_j:=\utau_j^*R=\rho(\cdot+j)} we deduce from~(\ref{HAC.rr}) 
\HAbeq \label{HAC.rrj} 
\rho_j\sim\rho_j(0) 
\HAqa j\geq1\;. 
\HAeeq 
Furthermore, since 
\HAhb{\HApl\rho_j=(\HApl\rho)(\cdot+j)}, we obtain from (\ref{HAC.dkrr}) 
and~(\ref{HAC.rrj}) 
\HAbeq \label{HAC.dkrj} 
\|\HApl^k\rho_j\|_\HAiy\leq c(k)\rho_j(0) 
\HAqa j\geq1 
\HAqb k\geq0\;. 
\HAeeq 
Due to 
\HAhb{R=r} and 
\HAhb{\utau_j^*r=\HAka_{j*}r} we see from (\ref{HAC.rrj}) 
and~(\ref{HAC.dkrj}) that 
\HAhb{r\in C\bigl(\HAci J,(0,\HAiy)\bigr)} satisfies 
(\ref{HAN.sd})(iii),~(iv) with 
\HAhb{\HAgK=\HAgS} and 
\HAhb{S=I_2}. 

\par 
(4) 
Suppose 
\HAhb{R\in\HAcF(J_\HAiy)}. Since 
\HAhb{\dot\HAsa=(1+\dot R^2)^{1/2}} we get from~(\ref{HAC.dt}) 
\HAbeq \label{HAC.tF} 
\dot\utau=\bigl(1+(\utau^*\dot R)^2\bigr)^{-1/2}\;. 
\HAeeq 
Using this and 
\HAhb{\|\dot R\|_\HAiy<\HAiy} we obtain 
\HAbeq \label{HAC.dtC} 
1/c\leq\dot\utau\leq1\;. 
\HAeeq 
From (\ref{HAC.tF}),\; (\ref{HAC.dtC}), and (\ref{HAC.F})(ii) we deduce 
inductively 
\HAbeq \label{HAC.dkt} 
\|\HApl^k\utau\|_\HAiy<\HAiy 
\HAqa k\geq1\;. 
\HAeeq 
Hence (\ref{HAC.tj}) implies 
\HAbeq \label{HAC.tjk} 
\dot\utau_j\sim\HAmf{1} 
\HAqa \|\HApl^k\utau_j\|_\HAiy\leq c(k) 
\HAqa j\geq1 
\HAqb k\geq1\;. 
\HAnpb 
\HAeeq 
Thus 
\HAhb{r=\HAmf{1}} satisfies (\ref{HAN.sd})(iv) with 
\HAhb{\HAgK=\HAgS} and 
\HAhb{S=I_2}, and 
(\ref{HAN.sd})(iii) is trivially true. 

\par 
(5) 
Again we assume 
\HAhb{R\in\HAcC(J)} or 
\HAhb{R\in\HAcF(J_\HAiy)}. Then 
\HAbeq \label{HAC.sg} 
\HAsa_{j*}\,\D t^2=\utau_j^*\,\D t^2=\D\utau_j^2=\dot\utau_j^2\,\D s^2\;. 
\HAeeq 
If 
\HAhb{R\in\HAcC(J)}, then we get 
\HAhb{\dot\utau_j^2=\rho_j^2} from~(\ref{HAC.t}). Hence 
$$ 
\HAsa_{j*}\,\D t^2=\rho_j^2\,\D s^2=(r_{\HAsa_j})^2\,\D s^2 
\HAqa j\geq1\;. 
$$ 
If 
\HAhb{R\in\HAcF(J_\HAiy)}, then we obtain from (\ref{HAC.tjk}) 
and~(\ref{HAC.sg}) 
$$ 
\HAsa_{j*}\,\D t^2\sim\D s^2=(r_{\HAsa_j})^2\,\D s^2 
\HAqa j\geq1\;. 
\HAnpb 
$$ 
Hence (\ref{HAN.sd})(v) applies to~$r$ and 
\HAhb{g=\D t^2} with 
\HAhb{\HAgK=\HAgS} and 
\HAhb{S=I_2} as well. 

\par 
(6)  
Using~(\ref{HAC.sg}), we infer from (\ref{HAC.t}) and (\ref{HAC.dkrj}) if 
\HAhb{R\in\HAcC(J)}, respectively from (\ref{HAC.t}) and (\ref{HAC.dkt}) if 
\HAhb{R\in\HAcF(J_\HAiy)}, that 
$$ 
\|\HApl^k(\HAsa_*\,\D t^2)\|_\HAiy\leq c(k)r_\HAsa^2 
\HAqa \HAsa\in\HAgS_I 
\HAqb k\geq0\;. 
\HAnpb 
$$ 
Thus (\ref{HAN.sd})(vi) is also satisfied. This proves the assertion. 
\qed 
\end{proof} 
\section{Model Cusps and Funnels}\label{HAsec-F}%
We suppose 
\HAhb{R\in\HAcR(J)}, 
\ \HAhb{0\leq d\leq\bar d}, and $B$~is a \hbox{$d$-dimensional} 
submanifold of~$\HABR^{\bar d}$. Let $I$ be an open cofinal subinterval 
of~$\HAci J$. We set 
$$ 
P(R,B;I):= 
\bigl\{\,\bigl(t,R(t)y\bigr)\ ;\ t\in\HAci J, 
\ y\in B\,\bigr\} 
\HAis\HABR\times\HABR^{\bar d}=\HABR^{1+\bar d}\;. 
$$ 
Then 
\HAhb{P=P(R,B)=P(R,B;\HAci J)} is~a 
\HAhb{(1+d)}-dimensional submanifold of~$\HABR^{1+\bar d}$, the 
(\emph{model}) $(R,B)$\emph{-pipe on}~$J$, also called 
(\emph{model}) \hbox{$R$\emph{-pipe}} \emph{over} (\emph{the basis})~$B$ 
\emph{on}~$J$. Note 
$$ 
\HApl P(R,B)=P(R,\HApl B)\;, 
$$ 
where 
\HAhb{P(R,\HAes):=\HAes}. An \hbox{$R$-pipe} is an \hbox{$R$-\emph{cusp}} 
if 
\HAhb{R(\HAom)=0}, where 
\HAhb{\HAom\in\{0,\HAiy\}} and 
\HAhb{J=J_\HAom}, and an \hbox{$R$\emph{-funnel}} otherwise. The map 
\HAbeq \label{HAF.ph} 
\HAvp=\HAvp[R]\HAsco P\HAra\HAci J\times B 
\HAqb \bigl(t,R(t)y\bigr)\HAmt(t,y) 
\HAnpb 
\HAeeq 
is a diffeomorphism, the \emph{canonical stretching diffeomorphism} 
of~$P$. 

\par 
If 
\HAhb{d=0}, then $B$~is a countable discrete subset of~$\HABR^{\bar d}$. In 
abuse of language and for a unified presentation we call it uniformly 
regular Riemannian manifold as well and write formally $(B,g_B)$ for~$B$, 
although $g_B$~has no proper meaning. In this case $g_B$~has to be replaced 
by~$0$ in the formulas below.  

\par 
Suppose 
\HAhb{p\in C^\HAiy\bigl(P,(0,\HAiy)\bigr)} and $g_P$~is a Riemannian metric 
for~$P$. Then $p$~is a \emph{cofinal singularity function for 
$(P,g_P)$ on} 
\HAhb{S\HAis B} if there exists a cofinal subinterval~$I$ of~$J$ such 
that $p$~is a singularity function for $(P,g_P)$ on 
\HAhb{\HAvp^{-1}(I\times S)}. It follows from 
Lemma~\ref{HAlem-U.C} that this is then true for every cofinal 
subinterval of~$J$. In related situations the qualifier `cofinal' has 
similar (obvious) meanings. 

\par 
We consider the following assumption: 
\HAbeq \label{HAF.B} 
\HAbal 
\\ 
\noalign{\vskip-7\jot} 
&g_B\text{ is a Riemannian metric for $B$,  
 \ $S\HAis B$, and}\\ 
\noalign{\vskip-1\jot} 
&b\text{ is a bounded singularity function for $(B,g_B)$ on }S.  
\HAeal 
\HAeeq 
\begin{lemma}\label{HAlem-F.g} 
Let condition~(\ref{HAF.B}) apply. Suppose 
\HAhb{a\in C^\HAiy\bigl(J,(0,\HAiy)\bigr)} and $r$~is a bounded 
singularity function for $(\HAci J,a\,\D t^2)$ on some cofinal 
subinterval of~$J$. Let 
\HAhb{R\in\HAcR(J)} and set 
\HAbeq \label{HAF.gr} 
g:=\HAvp^*(a\,\D t^2+g_B) 
\HAqb p:=\HAvp^*(r\otimes b)\;. 
\HAnpb 
\HAeeq 
Then $p$~is a cofinal singularity function for $(P,g)$ on~$S$. 
\end{lemma} 
\begin{proof} 
By Theorem~\ref{HAthm-P.MM} 
\ \HAhb{r\otimes b} is a cofinal singularity function for 
\HAhb{(\HAci J\times B,\ a\,\D t^2+g_B)} on~$S$. Hence the assertion follows 
from (\ref{HAF.ph}) and Lemma~\ref{HAlem-P.f}. 
\qed 
\end{proof} 
\begin{corollary}\label{HAcor-F.g} 
Put 
\HAbeq \label{HApro-F.gph} 
\hat g:=\HAvp^*\bigl((a\,\D t^2+g_B)\big/(r\otimes b)^2\bigr)\;. 
\HAnpb 
\HAeeq 
Then $(P,\hat g)$ is cofinally uniformly regular on~$S$. 
\end{corollary} 
The following two propositions are cornerstones for the construction of 
wide classes of singular manifolds. 
\begin{proposition}\label{HApro-F.P} 
Let (\ref{HAF.B}) be satisfied and suppose 
\HAhb{R\in\HAcC(J)} or 
\HAhb{R\in\HAcF(J_\HAiy)}. Set 
\HAhb{a:=\HAmf{1}}. Define~$r$ by~(\ref{HAC.r0}) and $g$~by (\ref{HAF.gr}). 
Then $p$~is a cofinal singularity function for $(P,g)$ on~$S$. 
\end{proposition} 
\begin{proof} 
Lemmas \ref{HAlem-C.CF} and~\ref{HAlem-F.g}. 
\qed 
\end{proof} 
We write 
\HAhb{\psi\in\HAcR_0\HAJhJ} if 
\HAhb{\psi\in\HAcR(J)} with 
\HAhb{\psi(J)=\hat J\in\{J_0,J_\HAiy\}} and 
\HAhb{\HAthpsi(t)\neq0} for 
\HAhb{t\in J}. Thus $\psi$~is a diffeomorphism from~$J$ onto~$\hat J$. 
\begin{proposition}\label{HApro-F.I} 
Suppose (\ref{HAF.B}) applies, 
\HAhb{\psi\in\HAcR_0\HAJhJ}, and 
\HAhb{\hat R\in\HAcC(\hat J\,)}, or 
\HAhb{\hat J=J_\HAiy} and 
\HAhb{\hat R\in\HAcF(J_\HAiy)}. Set 
$$ 
R:=\psi^*\hat R 
\HAqb \HAvp:=\HAvp[R] 
\HAqb g:=\HAvp^*(\HAthpsi^2\,\D t^2+g_B)\;, 
$$ 
and 
$$ 
r=r[R]:= 
 \left\{ 
 \HAbal 
 {}
 &R     
  &\quad \text{if } &\hat R\in\HAcC(\hat J\,)\;,\\ 
 &\HAmf{1}   
  &\quad \text{if } &\hat R\in\HAcF(J_\HAiy)\;. 
 \HAeal 
 \right. 
 \HAnpb 
$$  
Then 
\HAhb{p:=\HAvp^*(r\otimes b)} is a cofinal singularity function 
for $(P,g)$ on~$S$. 
\end{proposition} 
\begin{proof} 
We write 
\HAhb{\hat P:=P(\hat R,B)}, 
\ \HAhb{\hat\HAvp:=\HAvp[\hat R]}, and 
\HAhb{\hat g:=\hat\HAvp^*(\D s^2+g_B)}. Then 
\HAbeq \label{HAF.Ph} 
\Phi:=\hat\HAvp^{-1}\circ(\psi\times\HAid_B)\circ\HAvp 
\HAsco P\HAra\hat P 
\HAeeq 
is a diffeomorphism and 
\HAbeq \label{HAF.Phg} 
\Phi^*\hat g=\HAvp^*(\psi^*\times\HAid_B)\hat\HAvp_*\hat g 
=\HAvp^*(\psi^*\times\HAid_B)(\D s^2+g_B) 
=\HAvp^*(\HAthpsi^2\,\D t^2+g_B)=g\;. 
\HAeeq 
Furthermore, setting 
\HAhb{\hat r:=r[\hat R]} and 
\HAhb{\hat p:=\hat\HAvp^*(\hat r\otimes b)}, 
\HAbeq \label{HAF.Php} 
\Phi^*\hat p=\HAvp^*(\psi^*\times\HAid_B)(\hat r\otimes b) 
=\HAvp^*(r\otimes b)=p\;. 
\HAeeq 
Proposition~\ref{HApro-F.P} guarantees that $\hat p$~is a 
cofinal singularity function for $(\hat P,\hat g)$ on~$S$. Hence the 
assertion follows from \hbox{(\ref{HAF.Ph})--(\ref{HAF.Php})}
and Lemma~\ref{HAlem-P.f}. 
\qed 
\end{proof} 
Now we provide some examples. The most important ones concern 
\hbox{$\HAal$\emph{-pipes}}, that is, \hbox{$\HAsR_\HAal$-pipes} over~$B$ 
on~$J$. We write 
\HAhb{P_\HAal=P_\HAal(B):=P(\HAsR_\HAal,B)} and 
\HAhb{\HAvp_\HAal:=\HAvp[\HAsR_\HAal]} for 
\HAhb{\HAal\in\HABR}. 
\begin{HAexamples}\label{HAexa-F.ex} 
Let (\ref{HAF.B}) be satisfied. 

\par 
(a) 
Set 
\HAhb{g_\HAal:=\HAvp_\HAal^*(\D t^2+g_B)}, 
$$ 
p_\HAal:=\HAvp_\HAal^*(\HAsR_\HAal\otimes b) 
\text{ if either $J=J_0$ and }\HAal\geq1\;, 
\text{ or $J=J_\HAiy$ and }\HAal<0\;, 
$$ 
and 
$$ 
p_\HAal:=\HAvp_\HAal^*(\HAmf{1}\otimes b) 
\text{ if $J=J_\HAiy$ and }0\leq\HAal\leq1\;. 
\HAnpb 
$$ 
Then $p_\HAal$~is a cofinal singularity function for 
$(P_\HAal,g_\HAal)$ on~$S$. 
\begin{proof} 
Example~\ref{HAexa-C.ex}(a) and Proposition~\ref{HApro-F.P}. 
\qed 
\end{proof} 

\par 
(b) 
We put 
$$ 
p_\HAal:=\HAvp_\HAal^*(\HAsR_\HAal\otimes b) 
\text{ if $J=J_0$ and }0<\HAal\leq1\;, 
$$ 
and 
$$ 
p_\HAal:=\HAvp_\HAal^*(\HAmf{1}\otimes b) 
\text{ if either $J=J_0$ and }\HAal\leq0\;, 
\text{ or $J=J_\HAiy$ and }\HAal\geq1\;. 
$$ 
We also fix 
\HAhb{\HAba\neq0} such that 
\HAhb{0<\HAba\leq\HAal} if 
\HAhb{J=J_0} and 
\HAhb{0<\HAal\leq1}, 
\ \HAhb{\HAba\geq\HAal} if 
\HAhb{J=J_\HAiy} and 
\HAhb{\HAal>1}, and 
\HAhb{\HAba\leq\HAal} if 
\HAhb{J=J_0} and 
\HAhb{\HAal\leq0}. Then $p_\HAal$~is a cofinal singularity function for 
$(P_\HAal,g_{\HAal,\HAba})$ on~$S$, where 
\HAhb{g_{\HAal,\HAba}:=\HAvp_\HAal^*(t^{2(\HAba-1)}\,\D t^2+g_B)}. 
\begin{proof} 
Note that 
\HAhb{\HAsR_\HAba\in\HAcR_0\HAJhJ} with 
\HAhb{\hat J=J} if 
\HAhb{\HAba>0}, and 
\HAhb{\hat J=J_\HAiy} if 
\HAhb{J=J_0} and 
\HAhb{\HAba<0}. Moreover, 
\HAhb{\HAsR_\HAba^*\HAsR_\HAga=\HAsR_{\HAba\HAga}} for 
\HAhb{\HAga\in\HABR}. 

\par 
We put 
\HAhb{\psi:=\HAsR_\HAba} and 
\HAhb{\hat R:=\HAsR_{\HAal/\HAba}} so that 
\HAhb{\psi^*\hat R=\HAsR_\HAal}. It follows 
from Example~\ref{HAexa-C.ex}(a) that 
\HAhb{\hat R\in\HAcC(J_0)} if 
\HAhb{J=J_0} and 
\HAhb{0<\HAal\leq1}, and 
\HAhb{\hat R\in\HAcF(J_\HAiy)} otherwise. Moreover, 
\HAhb{\HAthpsi=\HAthsR_\HAba\sim\HAsR_{\HAba-1}}. Now the claim follows from 
Proposition~\ref{HApro-F.I}. 
\qed 
\end{proof} 

\par 
(c) 
Suppose 
\HAhb{J=J_0} and 
\HAhb{R(t):=1-\HAal\arctan(1-1/t)} with 
\HAhb{\HAal\geq-2/\upi}. Set 
$$ 
p:=\HAvp^*(\HAmf{1}\otimes b)\text{ if }\HAal>-2/\upi 
\HAqb p:=\HAvp^*(R\otimes b)\text{ if }\HAal=-2/\upi\;, 
\HAnpb 
$$ 
and 
\HAhb{g:=\HAvp^*(t^{-4}\,\D t^2+g_B)}. Then $p$~is a cofinal 
singularity function for $(P,g)$ on~$S$. 
\end{HAexamples} 
\begin{proof} 
Put 
\HAhb{\hat R:=R_{\arctan,\HAal,1}} (see Example~\ref{HAexa-C.ex}(c)) and 
\HAhb{\psi:=\HAsR_{-1}}. Then 
\HAhb{R=\psi^*\hat R} and 
\HAhb{\HAthpsi\sim\HAsR_{-2}}. Hence Example~\ref{HAexa-C.ex}(c) and 
Proposition~\ref{HApro-F.I} imply the assertion. 
\qed 
\end{proof} 
\section{Submanifolds of Euclidean Spaces}\label{HAsec-E}%
Now we consider the case where $\HAMg$ is a Riemannian submanifold of 
$\HARngn$ for some 
\HAhb{n\in\HABN^\times}. In other words, we assume 
$$ 
\HAMg\HAhr\HARngn\;. 
$$ 
By Nash's theorem this is no restriction of generality. It is now natural 
and convenient to describe~$M$ by local parametrizations. Hereby, given a 
local chart~$\HAka$ for~$M$, the map 
$$ 
i_\HAka:=\HAia_M\circ\HAka^{-1} 
\in C^\HAiy\bigl(\HAka(U_{\HAcoU\HAka}),\HABR^n\bigr) 
$$ 
is the \emph{local parametrization associated with}~$\HAka$. The 
following lemma provides a useful tool for establishing that a given 
function~$\rho$ on~$M$ is a singularity function for~$\HAMg$. 

\par 
\emph{By a parametrization-regular 
\hbox{(p-r)} singularity datum for~$\HAMg$ on} 
\HAhb{S\HAis M} we mean a pair $\HArgK$ with the following properties: 
\HAbeq \label{HAE.pr} 
 \HAbal
  \\ 
 \noalign{\vskip-6\jot}   
 \rm{(i)} \qquad    &\HAgK\text{ is an atlas for $M$ such that $\HAgK_S$ is 
                     shrinkable}\\
 \noalign{\vskip-1\jot} 
                    &\text{and has finite multiplicity}\;.\\ 
 \noalign{\vskip-1\jot} 
 \rm{(ii)} \qquad   &\rho\in C^\HAiy\bigl((M,(0,\HAiy)\bigr)
                     \text{ satisfies (\ref{HAN.sd})(iii) and (iv)}\;.\\  
 \noalign{\vskip-1\jot} 
 \rm{(iii)} \qquad  &\HAka_*g\geq\rho_\HAka^2g_m/c\;, 
                     \ \ \HAka\in\HAgK_S\;.\\ 
 \noalign{\vskip-1\jot} 
 \rm{(iv)} \qquad   &\|\HApl^ki_\HAka\|_\HAiy\leq c(k)\rho_\HAka\;, 
                     \ \ \HAka\in\HAgK_S\;,
                     \ \ k\geq1\;,
 \HAeal
\HAeeq 
where $\HApl$~denotes the Fr\'echet derivative. Clearly, $\rho$~is~a 
\emph{\hbox{p-r} singularity function for~$\HAMg$ on}~$S$ if there exists 
an atlas~$\HAgK$ such that $\HArgK$~is a 
\hbox{p-r} singularity datum for~$\HAMg$ on~$S$. 
\begin{lemma}\label{HAlem-E.i} 
Suppose $\HArgK$ is~a 
\hbox{p-r} singularity datum for~$\HAMg$ on~$S$. Then it is a singularity 
datum for~$\HAMg$ on~$S$. 
\end{lemma}
\begin{proof} 
(1) 
In the following, we identify a linear map 
\HAhb{a\HAsco\HABR^m\HAra\HABR^n} with its representation matrix 
\HAhb{[a]\in\HABR^{n\times m}} with respect to the standard bases. Then 
\HAbeq \label{HAE.ig} 
\HAka_*g=\HAka_*(\HAia_M^*g_n)=i_\HAka^*g_n 
=(\HApl i_\HAka)^\top\HApl i_\HAka 
\HAqa \HAka\in\HAgK\;. 
\HAeeq 
From this and (\ref{HAE.pr})(iii) and~(iv) it follows 
\HAbeq \label{HAE.grg} 
\HAka_*g\sim\rho_\HAka^2g_m 
\HAqb \|\HAka_*g\|_{k,\HAiy}\leq c(\HAka)\rho_\HAka^2  
\HAqa \HAka\in\HAgK_S 
\HAqb k\in\HABN\;. 
\HAeeq 
Hence $[\HAka_*g]$~has its spectrum in 
\HAhb{[\rho_\HAka^2/c,\ c\rho_\HAka^2]\HAis\HABR} for 
\HAhb{\HAka\in\HAgK_S}. Consequently, the spectrum of~$[\HAka_*g]^{-1}$ 
is contained in 
\HAhb{[\rho_\HAka^{-2}/c,\ c\rho_\HAka^{-2}]} for 
\HAhb{\HAka\in\HAgK_S}. This implies 
$$ 
\|[\HAka_*g]^{-1}\|_\HAiy\leq c/\rho_\HAka^2 
\HAqa \HAka\in\HAgK_S\;. 
$$ 
Thus, by the chain rule and (\ref{HAE.grg}), it follows 
\HAbeq \label{HAE.gink} 
\|[\HAka_*g]^{-1}\|_{k,\HAiy}\leq c(k)\rho_\HAka^{-2} 
\HAqa \HAka\in\HAgK_S 
\HAqb k\in\HABN\;. 
\HAeeq 

\par 
(2) 
We set 
$$ 
\HALda_\HAka(x):=[\HAka_*g]^{-1}(\HApl i_\HAka)^\top\in\HABR^{m\times n} 
\HAqa x\in Q_\HAka^m 
\HAqb \HAka\in\HAgK_S\;. 
$$ 
Then (\ref{HAE.pr})(iv) and (\ref{HAE.gink}) imply 
\HAbeq \label{HAE.L} 
\|\HALda_\HAka\|_{k,\HAiy}\leq c(k)\rho_\HAka^{-1} 
\HAqa \HAka\in\HAgK_S 
\HAqb k\in\HABN\;. 
\HAeeq 
Given 
\HAhb{\HAka\in\HAgK_S} and 
\HAhb{p\in U_{\HAcoU\HAka}}, 
\HAbeq \label{HAE.T} 
T_pM=\{p\}\times\HApl i_\HAka\bigl((\HAka(p)\bigl)\HARm 
\HAhr\{p\}\times\HABR^n=T_p\HABR^n\;. 
\HAeeq 
We read off (\ref{HAE.ig}) that $\HALda_\HAka$~is a left inverse 
for~$\HApl i_\HAka$. Furthermore, 
$$ 
\ker(\HALda_\HAka)=\ker\bigl((\HApl i_\HAka)^\top\bigr) 
=\bigl(\HAim(\HApl i_\HAka)\bigr)^\bot\;. 
$$ 
It follows from this and (\ref{HAE.T}) that~$T_p\HAka$, the tangential 
of~$\HAka$ at~$p$, is given by 
$$ 
T_p\HAka\HAsco T_pM\HAra T_{\HAka(p)}\HABR^m 
\HAqb (p,\xi)\HAmt\bigl(\HAka(p),\HALda_\HAka(\HAka(p))\xi\bigr) 
$$ 
(cf.\ \cite[Remark~10.3(d)]{AmE06a}). Thus we find 
\HAhb{\HApl(\tilde\HAka\circ\HAka^{-1})=\HALda_{\tilde\HAka}\HApl i_\HAka} 
for 
\HAhb{\HAka,\tilde\HAka\in\HAgK_S} with 
\HAhb{U_{\HAcoU\HAka}\cap U_{\HAcoU\tilde\HAka}\neq\HAes}. Hence 
(\ref{HAE.pr})(iv) and (\ref{HAE.L}) imply 
$$ 
\|\tilde\HAka\circ\HAka^{-1}\|_{k,\HAiy}\leq c(k) 
\HAqa \HAka,\tilde\HAka\in\HAgK_S 
\HAqb k\in\HABN\;, 
$$ 
due to 
\HAhb{\HAim(\tilde\HAka\circ\HAka^{-1})\HAis Q^m}. Thus, recalling 
(\ref{HAE.pr})(i), we see that $\HAgK$~is uniformly regular on~$S$. 
This proves the claim. 
\qed 
\end{proof} 
In the next lemma we consider a particularly simple, but important, 
\hbox{p-r} regular singularity datum. In this special situation it is 
the converse of the preceding lemma. 
\begin{lemma}\label{HAlem-E.c} 
Suppose 
\HAhb{S\HAis M} and $S$~is compact in~$\HABR^n$. If $(M,\HAgK,g)$ is 
uniformly regular on~$S$, then $(\HAmf{1},\HAgK)$ is~a 
\hbox{p-r} singularity datum for $\HAMg$ on~$S$. 
\end{lemma}
\begin{proof} 
Due to the hypotheses, conditions (\ref{HAE.pr})\hbox{(i)--(iii)} are 
trivially satisfied with 
\HAhb{\rho=\HAmf{1}}. 

\par 
For each 
\HAhb{p\in S} there is a normalized local chart~$\HAvp_p$ 
for~$\HABR^n$ such that 
\HAhb{\HAvp_p(p)=0}, 
\HAbeq \label{HAE.kp} 
\|\HAvp_p^{-1}\|_{k,\HAiy}\leq c(k,p) 
\HAqa k\in\HABN\;, 
\HAeeq 
and 
\HAhb{\HAka_p:=\HAvp_p\HAsn(M\cap U_{\HAvp_p})} is a normalized local chart 
for~$M$ with 
\HAhb{\HAka_p(p)=0\in\HABR^m}. By the compactness of~$S$ in~$\HABR^n$ 
there exists a finite subset~$P$ of~$S$ such that 
\HAhb{\bigl\{\,U_p:=\HAdom(\HAka_p)\ ;\ p\in P\,\bigr\}} is an open covering 
of~$S$ in~$M$. We set 
\HAhb{\hat\HAgK:=\{\,\HAka_p\ ;\ p\in P\,\}} and 
\HAhb{\tilde\HAgK:=\hat\HAgK\cup(\HAgK\HAssm\HAgK_S)}. Then 
$\tilde\HAgK$~is an atlas for~$M$ and 
\HAhb{\tilde\HAgK_S=\hat\HAgK}. For 
\HAhb{p\in P} we define 
\HAhb{f_p:=\HAvp_p^{-1}\HAsco Q_{\HAvp_p}^n\HAra\HABR^n}. Then 
\HAhb{i_{\HAka_p}=f_p\HAsn Q_{\HAka_p}}, where $\HABR^m$~is identified with 
the subspace 
\HAhb{\HABR^m\times\{0\}} of~$\HABR^n$, of course. Since $\tilde\HAgK_S$~is 
finite, it is obvious from (\ref{HAE.kp}) that 
\HAbeq \label{HAE.KSk} 
\|\HApl^ki_\HAka\|_\HAiy\leq c(k) 
\HAqa \HAka\in\tilde\HAgK_S 
\HAqb k\geq1\;. 
\HAeeq 
By the same reason, and since $\HAgK$~has finite multiplicity on~$S$, 
we see that 
\HAhb{\smash{\tilde\HAgK\HAapproxS\HAgK}\HAsmtB}. Hence (\ref{HAE.KSk}) 
holds for~$\HAgK_S$ as well, that is, condition (\ref{HAE.pr})(iv) is valid 
also.  
\qed 
\end{proof} 
Now we return to the setting of the preceding section. It follows from 
Corollary~\ref{HAcor-F.g} and Proposition~\ref{HApro-F.P} that, given 
\HAhb{R\in\HAcC(J)} or 
\HAhb{R\in\HAcF(J_\HAiy)}, the \hbox{$R$-pipe} 
\HAhb{P=P(R,B)} can be equipped with countably many nonequivalent metrics 
which make it into a cofinally uniformly regular Riemannian manifold. 
However, since 
\HAhb{\HAia_P\HAsco P\HAhr\HABR^{1+\bar d}}, it is most natural to 
endow~$P$ with the metric 
\HAhb{g_P:=\HAia_P^*g_{1+\bar d}}. The following proposition gives 
sufficient conditions guaranteeing that $g$ in Proposition~\ref{HApro-F.P} 
can be replaced by~$g_P$. 
\begin{proposition}\label{HApro-E.P} 
Suppose $(B,g_B)$ is~a \hbox{$d$-dimensional} bounded Riemannian 
submanifold of $(\HABR^{\bar d},g_{\bar d})$ and $b$~is a 
\hbox{p-r} singularity function for $(B,g_B)$ on 
\HAhb{S\HAis B}. Also suppose 
\HAhb{R\in\HAcC(J)} or 
\HAhb{R\in\HAcF(J_\HAiy)} and define~$r$ by~(\ref{HAC.r0}). Then 
\HAhb{p=\HAvp^*(R\otimes b)} is a cofinal 
\hbox{p-r} singularity function for $(P,g_P)$ on~$S$. 
\end{proposition}
\begin{proof} 
(1) 
Let $\HAgB$ be an atlas for~$B$ such that $(b,\HAgB)$ is~a 
\hbox{p-r} singularity datum for $(B,g)$ on~$S$, and let $\HAgS$ be the 
\hbox{$R$-atlas} for~$\HAci J$. We write 
$$ 
f:=\HAia_P\circ\HAvp^{-1}\HAsco\HAci J\times B\HAra\HABR^{1+\bar d} 
\HAqb Y_\HAba:=i_\HAba\;, 
$$ 
and use the notations of Section~\ref{HAsec-C}. Then 
$$ 
f_{j\HAba}:=f\circ(\utau_j\times\HAba^{-1})=(\utau_j,\rho_jY_\HAba) 
\HAsco Q\times\HAba(U_\HAba)\HAra\HABR^{1+\bar d} 
$$ 
is a diffeomorphism onto an open subset~$U_{j\HAba}$ of~$P$. 
We denote by~$\varpi$ the permutation 
\HAhb{\HABR^{1+d}\HAra\HABR^{d+1}}, 
\ \HAhb{(t,y)\HAmt(y,t)} (which is only needed if 
\HAhb{\HApl B=\HAes}). Then, see~(\ref{HAF.ph}), 
$$ 
\HAka_{j\HAba}:=\varpi\circ f_{j\HAba}^{-1}(\HAba,\utau_j) 
\HAsco U_{j\HAba}\HAra\HAba(U_\HAba)\times Q 
$$ 
is a local chart for~$P$ and 
\HAhb{f_{j\HAba}=i_{\HAka_{j\HAba}}}. We set 
$$ 
\HAgK:=\{\,\HAka_{j\HAba}\ ;\ j\geq1,\ \HAba\in\HAgB\,\} 
=\HAvp^*\bigl(\varpi\circ(\HAgB\otimes\HAgS)\bigr) 
$$ 
where 
\HAhb{\HAgB\otimes\HAgS} is the product atlas on 
\HAhb{B\times\HAci J} and 
$$ 
\varpi\circ(\HAgB\otimes\HAgS) 
:=\{\,\HAsa\times\HAba\ ;\ \HAba\in\HAgB,\ \HAsa\in\HAgS\,\}\;. 
$$ 
By Lemma~\ref{HAlem-C.CF} we know that $\HAgS$~is uniformly regular on 
\HAhb{I:=I_2[R]}. Hence 
\HAhb{\HAgB\otimes\HAgS} is uniformly regular on 
\HAhb{S\times I} by Theorem~\ref{HAthm-P.MM}. From this and 
Lemma~\ref{HAlem-P.f} it follows that $\HAgK$~is a uniformly regular atlas 
for~$P$ on 
\HAhb{V:=\HAvp^{-1}(I\times S)}. 

\par 
(2) 
Given 
\HAhb{\HAka=\HAka_{j\HAba}\in\HAgK}, 
$$ 
\HAbal 
\HAka_*g_P 
&=\HAka_*\HAia_P^*g_{1+\bar d}=(\HAia_P\circ\HAka^{-1})^*g_{1+\bar d}\\
&=f_{j\HAba}^*(\D t^2+|\D y|^2)=\D\utau_j^2+|\D (\rho_jY_\HAba)|^2\;. 
\HAeal 
$$ 
Hence 
\HAhb{\D(\rho_jY_\HAba)=\HAthrho_j\,\D s\,Y_\HAba+\rho_j\,\D Y_\HAba} implies 
$$ 
\HAka_*g_P=(\HAthtau_j^2+\HAthrho_j^2\,|Y_\HAba|^2)\,\D s^2 
+2\rho_j\HAthrho_j\,\D s\,(Y_\HAba\HAsn\D Y_\HAba) 
+\rho_j^2\,|\D Y_\HAba|^2\;. 
$$ 
Using 
\HAhb{|\D Y_\HAba|^2=\HAba_*g_B} and estimating the next to the last 
term by the Cauchy-Schwarz inequality gives 
$$ 
\HAka_*g_P\geq(\HAthtau_j^2+(1-1/\HAve)\HAthrho_j^2\,|Y_\HAba|^2)\,\D s^2 
+(1-\HAve)\rho_j^2\HAba_*g_B 
\HAnpb 
$$ 
for 
\HAhb{0<\HAve<1}, 
\ \HAhb{j\geq1}, and 
\HAhb{\HAba\in\HAgB}. 

\par 
(3) 
Suppose 
\HAhb{R\in\HAcC(J)}. Then 
\HAbeq \label{HAE.tr} 
\HAthtau_j^2=\rho_j^2 
\HAeeq 
by (\ref{HAC.t}), and 
\HAhb{\HAthrho_j^2\leq c\rho_j^2} by (\ref{HAC.rrj}) and~(\ref{HAC.dkrj}). 
Thus the boundedness of~$B$ in $(\HABR^{\bar d},g_{\bar d})$ implies 
that we can choose~$\HAve$ sufficiently close to~$1$ such that 
\HAbeq \label{HAE.gC} 
\HAka_*g_B\geq\rho_j^2(\D s^2+\HAba_*g_B)/c 
\geq\rho_j^2(\D s^2+b_\HAba^2g_d)/c 
\HAeeq 
for 
\HAhb{j\geq1}, 
\ \HAhb{\HAba\in\HAgB_S}, and  
\HAhb{\HAka=\HAka_{j\HAba}}, where the last inequality holds since 
$(b,\HAgK)$ is~a 
\hbox{p-r} singularity datum for $(B,g_B)$ on~$S$. 

\par 
(4)   
Assume 
\HAhb{R\in\HAcF(J_\HAiy)}. Then (\ref{HAC.dtC}) implies 
$$ 
\HAthtau_j\sim\HAmf{1} 
\HAqa j\geq1\;. 
$$ 
From this and 
$$ 
\HAthrho_j=(\HAthR\circ\utau_j)\HAthtau_j 
\HAqa \|\HAthR\|_{k,\HAiy}<\HAiy  
\HAqb k\in\HABN\;, 
$$ 
we get 
\HAbeq \label{HAE.rdj} 
\|\HAthrho_j\|_{k,\HAiy}\leq c(k) 
\HAqa j\geq1 
\HAqb k\in\HABN\;. 
\HAeeq 
Thus, similarly as above, 
\HAbeq \label{HAE.gF} 
\HAka_*g_B\geq\bigl(\D s^2+\rho_j^2(0)b_\HAba^2g_d\bigr)\big/c 
\HAqa j\geq1  
\HAqb \HAba\in\HAgB_S 
\HAqb \HAka=\HAka_{j\HAba}\;. 
\HAeeq 

\par 
(5) 
Now we proceed analogously to the proof of Theorem~\ref{HAthm-P.MM}. 
Recalling that $\HAgS$~is shrinkable to 
\HAhb{1/2} on~$I$ we fix
\HAhb{r\in(1/2,1)} such that 
\HAhb{\bigl\{\,\HAka^{-1}(rQ_\HAka^{d+1})\ ;\ \HAka\in\HAgK_V\,\bigr\}} 
is a covering of~$V$. Then we set 
\HAhb{\HAda:=(1-r)\big/\sqrt{d+1}}, 
$$ 
\HAda_\HAba:=\min\{b_\HAba,\HAda\} 
\HAqb \HAda_j:=\min\bigl\{1/R(t_j),\ \HAda\bigr\}\,, 
$$ 
and 
$$ 
\HAgL_\HAba:=\HAgL(\HAda_\HAba,Q) 
\HAqb \HAgL_j:=\HAgL(\HAda_j,Q_\HAba^d) 
$$ 
for 
\HAhb{\HAba\in\HAgB_S} and 
\HAhb{j\geq1}. Note that the boundedness of~$b$ implies 
\HAbeq \label{HAE.db} 
\HAda_\HAba\sim b_\HAba 
\HAqa \HAba\in\HAgK_S\;. 
\HAeeq 
Furthermore, 
\HAbeq \label{HAE.drF} 
\HAda_j\sim1/\rho_j(0) 
\HAqb j\geq1\;,\quad\text{if }r\in\HAcF(J_\HAiy)\;, 
\HAnpb 
\HAeeq 
since 
\HAhb{R(t_j)=\rho_j(0)} and 
\HAhb{1/R\leq c} in this case. 

\par 
Given 
\HAhb{\HAka=\HAka_{j\HAba}\in\HAgK_V}, we define 
$$ 
\HAgN_\HAka:= 
 \left\{ 
 \HAbal 
 {}
 &\{\,\umu\times\HAlda 
  \ ;\ \HAlda\in\HAgL_\HAba\;,\ \umu:=\HAid_{Q_\HAba^d}\,\}      
  &\quad \text{if }&R\in\HAcC(J)\;,\\ 
 &\{\,\umu\times\HAlda\ ;\ \HAlda\in\HAgL_\HAba\;,\ \umu\in\HAgL_j\,\}       
  &\quad \text{if }&R\in\HAcF(J_\HAiy)\;. 
 \HAeal 
 \right. 
$$ 
Then $\HAgN_\HAka$~is an atlas for~$Q_\HAka^{d+1}$ which is 
uniformly regular on~$rQ_\HAka^{d+1}$. Consequently, cf.~(\ref{HAP.M}),
$$ 
\HAgP:=\{\,\unu\circ\HAka\ ;\ \HAka\in\HAgK_V\;,\ \unu\in\HAgN_\HAka\,\} 
\cup(\HAgK\HAssm\HAgK_V) 
$$ 
is an atlas for~$P$ which is uniformly regular on~$V$. Observe 
$$ 
\HAgP_V\HAis\{\,\unu\circ\HAka 
\ ;\ \HAka\in\HAgK_V\;,\ \unu\in\HAgN_\HAka\,\}\;. 
\HAnpb 
$$ 
Hence condition~(\ref{HAE.pr})(i) is satisfied. 

\par 
(6) 
By the assumption on $(b,\HAgB)$ 
\HAbeq \label{HAE.bb} 
\|\HAba_*b\|_{k,\HAiy}\leq c(k)b_\HAba 
\HAqb b\HAsn U_\HAba\sim b_\HAba 
\HAqa \HAba\in\HAgB_S 
\HAqb k\geq0\;. 
\HAeeq 
Furthermore, 
\HAbeq \label{HAE.rj} 
\|\rho_j\|_{k,\HAiy}\leq c(k)\rho_j(0) 
\HAqb \rho_j\HAsn J_j\sim\rho_j(0) 
\HAqa j\geq1 
\HAqb k\geq0\;.  
\HAeeq 
Indeed, if 
\HAhb{R\in\HAcC(J)}, then this is a consequence of (\ref{HAC.dkrj}) and 
(\ref{HAC.rrj}), respectively. If 
\HAhb{R\in\HAcF(J_\HAiy)}, then 
\HAhb{\rho_j(0)=R(t_j)\geq1/c} for 
\HAhb{j\geq1}. Hence (\ref{HAE.rj}) follows from~(\ref{HAE.rdj}). 

\par 
We deduce from (\ref{HAP.mr}) that 
\HAbeq \label{HAE.bp} 
b_\HAba=\HAba_*b(0)=\HAka_*b(0) 
\sim(\unu\circ\HAka)_*b(0)=\upi_*b(0) 
\HAeeq 
and 
\HAbeq \label{HAE.rp} 
\rho_j(0)=(\HAsa_j)_*R(0)=\HAka_*R(0) 
\sim(\unu\circ\HAka)_*R(0)=\upi_*R(0) 
\HAeeq 
for 
\HAhb{\upi=\unu\circ\HAka\in\HAgP_V} with 
\HAhb{\HAka=\HAka_{j\HAba}\in\HAgK_V} and 
\HAhb{\unu\in\HAgL_\HAka}. From \hbox{(\ref{HAE.bb})--(\ref{HAE.rp})} 
we derive 
$$ 
\|\upi_*p\|_{k,\HAiy}\leq c(k)p_\upi 
\HAqb p\HAsn U_\upi\sim p_\upi 
\HAqa \upi\in\HAgP_V 
\HAqb k\geq0\;. 
\HAnpb 
$$ 
Thus condition~(\ref{HAE.pr})(ii) applies. 

\par 
(7) 
From (\ref{HAE.gC}),\, (\ref{HAE.gF}),\, (\ref{HAE.bp}),\, (\ref{HAE.rp}), 
and~(\ref{HAP.mg}) we get 
$$ 
\upi_*g_P\geq\rho_j^2(0)(\HAda_\HAba^2\,\D s^2+b_\HAba^2g_d)/c 
\quad\text{if }R\in\HAcC(J)\;, 
$$ 
respectively 
$$ 
\upi_*g_P\geq\bigl(\HAda_\HAba^2\,\D s^2 
+\rho_j^2(0)\HAda_j^2b_\HAba^2g_d\bigr)\big/c 
\quad\text{if }R\in\HAcF(J_\HAiy)\;, 
$$ 
for 
\HAhb{\upi=\unu\circ\HAka\in\HAgP_V} with 
\HAhb{\HAka=\HAka_{j\HAba}} and 
\HAhb{\unu\in\HAgN_\HAka}. From this, (\ref{HAE.db}), and (\ref{HAE.drF}) we 
obtain in either case 
\HAhb{\upi_*g_P\geq p_\upi^2g_{1+d}/c} for 
\HAhb{\upi\in\HAgP_V}. Thus condition~(\ref{HAE.pr})(iii) is fulfilled. 

\par 
(8) 
By the assumption on $(p,\HAgB)$ 
\HAbeq \label{HAE.dY} 
\|\HApa Y_\HAba\|_\HAiy\leq c(\HAal)b_\HAba 
\HAqa \HAba\in\HAgB_S 
\HAqb \HAal\in\HABN^d\HAssm\{0\}\;. 
\HAeeq 
Given 
\HAhb{\upi=\unu\circ\HAka\in\HAgP_V} with 
\HAhb{\HAka=\HAka_{j\HAba}\in\HAgK_V} and 
\HAhb{\unu=\umu\times\HAlda\in\HAgN_{\HAka_{j\HAba}}}, 
\HAbeq \label{HAE.ip} 
i_\upi=i_\HAka\circ\unu^{-1} 
=\bigl(\HAlda_*\utau_j,(\HAlda_*\rho_j)\umu_*Y_\HAba\bigr)\;. 
\HAeeq 
Suppose 
\HAhb{R\in\HAcC(J)}. Then we get from (\ref{HAE.tr}) and (\ref{HAP.dg}) 
\HAbeq \label{HAE.dklt} 
\|\HApl^k(\HAlda_*\utau_j)\|_\HAiy 
=\HAda_\HAba\,\|\HApl^{k-1}(\HAlda_*\rho_j)\|_\HAiy 
=\HAda_\HAba^k\,\|\HAlda_*(\HApl^{k-1}\rho_j)\|_\HAiy 
\leq c(k)b_\HAba\rho_j(0)
\HAeeq 
for 
\HAhb{j,k\geq1} and 
\HAhb{\HAba\in\HAgB_S}, due to 
\HAhb{0<\HAda_\HAba\leq1} and (\ref{HAE.db}). By means of 
\hbox{(\ref{HAE.bp})--(\ref{HAE.dklt})} 
and  
\HAhb{\umu=\HAid} we deduce 
\HAbeq \label{HAE.daip} 
\|\HApa i_\upi\|_\HAiy\leq c(\HAal) p_\upi 
\HAqa \upi\in\HAgP_V 
\HAqb \HAal\in\HABN^{1+d}\HAssm\{0\}\;, 
\HAnpb 
\HAeeq 
if 
\HAhb{R\in\HAcC(J)}. 

\par 
Assume 
\HAhb{R\in\HAcF(J_\HAiy)}. Then (\ref{HAC.tjk}) and the definition of~$r$ 
imply, similarly as above, 
$$ 
\|\HApl^k(\HAlda_*\utau_j)\|_\HAiy 
\leq c(k) b_\HAba\leq c(k) p_\upi 
$$ 
for 
\HAhb{\upi=\unu\circ\HAka\in\HAgP_V}, 
\ \HAhb{\HAka=\HAka_{jp}}, and 
\HAhb{\unu\in\HAgN_\HAka}. Analogously, we get from~(\ref{HAE.rdj}) 
\HAbeq \label{HAE.dklj} 
\|\HApl^k(\HAlda_*\rho_j)\|_\HAiy\leq c(k) p_\upi 
\HAqa \upi=\unu\circ\HAka\in\HAgP_V 
\HAqb \HAka=\HAka_{j\HAba} 
\HAqb \unu\in\HAgN_\HAka\;, 
\HAeeq 
for 
\HAhb{k\geq1}. Finally, similar arguments invoking (\ref{HAE.dY}) lead to 
\HAbeq \label{HAE.damY} 
\|\HApl^\HAal(\umu_*Y_\HAba)\|_\HAiy\leq c(\HAal)\HAda_jp_\upi 
\HAqa \HAal\in\HABN^{1+d}\HAssm\{0\}\;, 
\HAeeq 
for 
\HAhb{\upi=\unu\circ\HAka\in\HAgP_V } with 
\HAhb{\HAka=\HAka_{j\HAba}} and 
\HAhb{\unu\in\HAgN_\HAka}. By (\ref{HAE.rdj}) 
\ \HAhb{|\rho_j(s)-\rho_j(0)|\leq c} for 
\HAhb{s\in Q} and 
\HAhb{j\geq1}. Hence 
\HAbeq \label{HAE.rrF} 
1-c/\rho_j(0)\leq\rho_j(s)/\rho_j(0)\leq 1+c/\rho_j(0) 
\HAqa s\in Q 
\HAqb j\geq1\;. 
\HAeeq 
Assume 
\HAhb{R(\HAiy)<\HAiy}. Then 
\HAhb{1/c\leq\rho_j(s)\leq c} for 
\HAhb{s\in Q} and 
\HAhb{j\geq1}. In this case it is obvious that 
\HAbeq \label{HAE.rj0} 
\rho_j\sim\rho_j(0) 
\HAqa j\geq1\;. 
\HAeeq 
If, however, 
\HAhb{R(\HAiy)=\HAiy}, then we see from (\ref{HAE.rrF}) that there 
exists~$j_0$ such that (\ref{HAE.rj0}) holds for 
\HAhb{j\geq j_0}. As above, we observe that (\ref{HAE.rj0}) applies for 
\HAhb{1\leq j\leq j_0} also. Thus (\ref{HAE.rj0}) is true in general. Using 
this we infer from (\ref{HAE.drF}) and (\ref{HAE.damY}) that 
$$ 
\|(\HAlda_*\rho_j)\HApa(\umu_*Y_\HAba)\|_\HAiy\leq c(\HAal)p_\upi 
\HAqa \HAal\in\HABN^d\HAssm\{0\}\;, 
$$ 
for 
\HAhb{\upi=\unu\circ\HAka\in\HAgP_V} with 
\HAhb{\HAka=\HAka_{j\HAba}} and 
\HAhb{\unu\in\HAgN_\HAka}. Moreover, (\ref{HAE.dklj}),\, (\ref{HAE.damY}), 
\ \HAhb{0<\HAda_j\leq1}, and the boundedness of~$b$ guarantee 
$$ 
\|\HApl^k(\HAlda_*\rho_j)\HApa(\umu_*Y_\HAba)\|_\HAiy\leq c(k,\HAal)p_\upi 
\HAqa k\geq1 
\HAqb \HAal\in\HABN^d\;, 
$$ 
for 
\HAhb{\upi=\unu\circ\HAka\in\HAgP_V} with 
\HAhb{\HAka=\HAka_{j\HAba}} and 
\HAhb{\unu\in\HAgN_\HAka}. Here we also use the boundedness of~$B$ 
in~$\HABR^{\bar d}$ if 
\HAhb{\HAal=0}. This implies that estimate~(\ref{HAE.daip}) holds in this 
case as well. Hence condition~(\ref{HAE.pr})(iv) is also satisfied. This 
proves the assertion. 
\qed 
\end{proof} 
\begin{HAremark}\label{HArem-E.P} 
Let the hypotheses of Proposition~\ref{HApro-E.P} be satisfied with 
\HAhb{R\in\HAcC(J_0)}. Set 
\HAhb{(B_1,g_{B_1}):=(P,g_P)}, 
\ \HAhb{\bar d_1:=1+\bar d}, 
\ \HAhb{b_1:=p}, and 
\HAhb{S_1:=V=\HAvp^{-1}(I\times S)}. Then $(B_1,g_{B_1})$ 
is a bounded Riemannian submanifold of $(\HABR^{\bar d_1},g_{\bar d_1})$ 
and $b_1$~is a bounded 
\hbox{p-r} singularity function for $(B_1,g_{B_1})$ on~$S_1$. 

\par 
Fix 
\HAhb{J_1\in\{J_0,J_\HAiy\}} and 
\HAhb{R_1\in\HAcC(J_1)}, resp.\ 
\HAhb{R_1\in\HAcF(J_\HAiy)}. Set 
\HAhb{r_1:=R_1} if 
\HAhb{R_1\in\HAcC(J_1)}, resp.\ 
\HAhb{r_1:=\HAmf{1}} if 
\HAhb{R_1\in\HAcF(J_1)}. Denote by 
\HAhb{\HAvp_1\HAsco P_1=P(R_1,B_1)\HAra\HAci J_1\times B_1} the canonical 
stretching diffeomorphism of~$P_1$ and set 
\HAhb{g_{P_1}:=\HAia_{P_1}^*g_{1+\bar d_1}}. Then 
Proposition~\ref{HApro-E.P} applies to guarantee that 
\HAhb{p_1:=\HAvp_1^*(r_1\otimes b_1)} is a cofinal 
singularity function for $(P_1,g_{P_1})$ on~$S_1$. In particular, 
$(P_1,g_{P_1}/p_1^2)$ is cofinally uniformly regular and, given 
cofinal subintervals $I_1$ of~$J_1$ and $I$ of~$J$, resp.,  
$$ 
\HAvp_{1*}(g_{P_1}/p_1^2)\underset{I_1\times I\times S}{\sim} 
(r_1\otimes R\otimes b_1)^{-2}(\D s_1^2+\D s^2+g_B)  
\HAnpb 
$$ 
and 
\HAhb{\HAvp_1(P_1)=\HAci J_1\times\HAci J\times B}. 
\qed 
\end{HAremark} 
This remark shows that we can iterate Proposition~\ref{HApro-E.P} to handle 
`higher order' singularities, e.g.\ cuspidal corners or funnels 
with edges. 
\section{Singular Ends}\label{HAsec-M}%
Throughout this section, $\HAMg$ is an \hbox{$m$-dimensional} Riemannian 
manifold and 
\HAhb{J\in\{J_0,J_\HAiy\}}. 

\par 
Suppose:
$$ 
\HAbal 
{}
\\ 
 \noalign{\vskip-7\jot}   
\rm{(i)} \qquad &R\in\HAcC(J),\ \ell\in\{1,\ldots,m\}, 
                 \ \bar\ell\geq\ell\;.\\ 
\noalign{\vskip-1\jot} 
\rm{(ii)}\qquad &(B,g_B)\text{ is a compact  
                 $(\ell-1)$-dimensional}\\ 
\noalign{\vskip-1\jot} 
                 &\text{Riemannian submanifold of }\HABR^{\bar\ell}\;.\\
\noalign{\vskip-1\jot} 
\rm{(iii)}\qquad&(\HAGa,g_\HAGa)\text{ is a compact connected 
                 $(m-\ell)$-dimensional}\\
\noalign{\vskip-1\jot} 
                &\text{Riemannian manifold without boundary}\;.\\ 
\HAeal 
$$ 
Then 
$$ 
W=W(R,B,\HAGa):=P(R,B)\times\HAGa 
$$ 
is the \emph{smooth model \hbox{$\HAGa$-wedge} over the 
\hbox{$(R,B)$-pipe}} 
\HAhb{P=P(R,B)}. It is a submanifold of 
\HAhb{\HABR^{1+\bar\ell}\times\HAGa} of dimension~$m$, and 
\HAhb{\HApl W=\HApl P\times\HAGa}. If 
\HAhb{\ell=m}, then $\HAGa$~is a one-point space and $W$~is naturally 
identified with~$P$ (equivalently: there is no $(\HAGa,g_\HAGa)$). 
Thus every pipe is also a wedge. This convention 
allows for a uniform language by speaking, in what follows, of wedges only. 
Given a cofinal subinterval~$I$ of~$J$, we set 
$$ 
W[I]:=P(R,B;I)\times\HAGa\;. 
\HAnpb 
$$ 
We fix a Riemannian metric~$h_P$ for~$P$ and set 
\HAhb{g_W:=h_P+g_\HAGa}. 

\par 
Let $V$ be open in~$M$. Then~$\HAVg$, more loosely:~$V$, is~a 
\emph{smooth wedge of type}~$(W,g_W)$ in~$\HAMg$ if it is isometric 
to~$(W,g_W)$. More precisely, $(Vg)$~is said to be \emph{modeled} by 
$[\Phi,W,g_W]$ if $\Phi$~is an isometry from~$\HAVg$ onto~$(W,g_W)$, 
a~\emph{modeling isometry} for~$\HAVg$. 

\par 
Assume 
\HAbeq \label{HAM.S} 
\HAbal 
{}
\\ 
 \noalign{\vskip-7\jot}   
&\{V_0,V_1,\ldots,V_k\}\text{ is a finite open covering of $M$ such that}\\
\noalign{\vskip-1\jot} 
&\qquad 
\HAbal 
{}
\rm{(i)} \qquad &V_i\cap V_j=\HAes\;,\ 1\leq i<j\leq k\;;\\
\noalign{\vskip-2\jot} 
\rm{(ii)}\qquad &V_0\cap V_i\text{ is a relatively compact for } 
                1\leq i\leq k\;;\\
\noalign{\vskip-2\jot} 
\rm{(iii)}\qquad&(V_i,g)\text{ is a smooth wedge in $\HAMg$ for } 
                1\leq i\leq k\;.
\HAeal 
\HAeal 
\HAnpb  
\HAeeq 
Then $\HAMg$ is~a Riemannian \emph{manifold with} (finitely many) 
\emph{smooth singularities}. 

\par 
The following theorem is the main result of this paper. It is shown 
thereafter that we can derive from it all results stated in the 
introduction---and many more---by appropriate choices of the 
modeling data. 
\begin{theorem}\label{HAthm-M.S} 
Suppose $\HAMg$ is a Riemannian manifold with smooth singularities. 
Let $\rho_0$ be a singularity function for~$\HAMg$ on~$V_0$ and assume that 
$\rho_i$~is a cofinal singularity function for~$(V_i,g)$, 
\ \HAhb{1\leq i\leq k}. Then there exists a singularity function~$\rho$ 
for~$\HAMg$ such that
\HAhb{\rho\sim\rho_j} on~$V_j$ for 
\HAhb{0\leq j\leq k}. Thus 
\HAhb{(M,\ g/\rho^2)} is uniformly regular. 
\end{theorem}
\begin{proof} 
Suppose $(V_i,g)$ is modeled by 
$[\Phi_i,W_i,g_i]$ for 
\HAhb{1\leq i\leq k}, where we write~$W_i$ for 
$W(R_i,B_i,\HAGa_i)$ with 
\HAhb{R_i\in\HAcR(J_i)} and
\HAhb{g_i:=g_{W_i}}. Given a cofinal subinterval~$I_i$ of~$J_i$, we set 
\HAhb{S_i:=\Phi_i^{-1}\bigl(W_i[I_i]\bigr)}. By the relative compactness 
of 
\HAhb{V_0\cap V_i} we can find a closed subset~$S_0$ of~$V_0$ such that 
\HAhb{S_0\HAjs V_0\HAbssm\bigcup_{i=1}^kV_i} and 
\HAhb{\HAdist(S_0\cap V_i,\ V_i\HAssm V_0)>0} as well as closed cofinal 
subintervals~$I_i$ of~$J_i$, 
\ \HAhb{1\leq i\leq k}, such that 
\HAhb{\{S_0,S_1,\ldots,S_k\}} is a covering of~$M$. By the assumptions 
on~$\rho_j$, 
\ \HAhb{0\leq j\leq k}, we can find atlases~$\HAgK_j$, 
\ \HAhb{0\leq j\leq k}, such that $(\rho_j,\HAgK_j)$ is a singularity datum 
for $(V_j,g_j)$ on~$S_j$. Since 
\HAhb{V_0\cap V_i} is relatively compact it follows that 
\HAhb{\rho_0\sim\HAmf{1}} and 
\HAhb{\rho_i\sim\HAmf{1}} on 
\HAhb{V_0\cap V_i} for 
\HAhb{1\leq i\leq k}. Thus 
\HAhb{\rho_i\HAsn(V_i\cap V_j)\sim\rho_j\HAsn(V_i\cap V_j)} for 
\HAhb{0\leq i<j\leq k}, due to (\ref{HAM.S})(i). Note that 
\HAhb{S_0\cap S_i} is relatively compact in 
\HAhb{V_0\cap V_i}. Hence we can assume that 
$\HAgK_{i,\ S_0\cap S_i}$ is finite for 
\HAhb{1\leq i\leq k}. From this and (\ref{HAM.S})(i) it is clear that 
condition~(v) of Lemma~\ref{HAlem-P.P} is satisfied. Hence that lemma 
guarantees the validity of the assertion. 
\qed 
\end{proof} 
Let $\HAVg$ be a smooth wedge in~$\HAMg$ modeled by $[\Phi,W,g_W]$. 
Then 
\HAhb{W=P\times\HAGa} with 
\HAhb{P=P[R,B]}, and 
\HAhb{\HAvp=\HAvp[R]} is the canonical stretching isometry 
from~$(P,h_P)$ onto 
\HAhb{(J\times B,\ \HAvp_*h_P)}.  Hence 
\HAbeq \label{HAM.ps} 
\Psi:=(\HAvp\times\HAid_\HAGa)\circ\Phi 
\HAsco\HAVg\HAmt(J\times B\times\HAGa,\ \HAvp_*h_P+g_\HAGa) 
\HAeeq  
is a modeling isometry for~$\HAVg$. Since $B$ and~$\HAGa$ are compact, 
$\HAmf{1}_B$~and~$\HAmf{1}_\HAGa$ are singularity functions for $B$ 
and~$\HAGa$, respectively. Suppose 
\HAhb{r\in C^\HAiy\bigl(J,(0,\HAiy)\bigr)}. Then 
\HAhb{r\otimes\HAmf{1}_B\otimes\HAmf{1}_\HAGa} is the 
`constant extension' of~$r$ over 
\HAhb{J\times B\times\HAGa}. It satisfies 
$$ 
(\HAvp\times\HAid_\HAGa)^* 
(r\otimes\HAmf{1}_B\otimes\HAmf{1}_\HAGa) 
(t,y,z)=r(t) 
\HAqa (t,y,z)\in J\times B\times\HAGa\;. 
$$ 
Thus, in abuse of notation, we set 
\HAbeq \label{HAM.ph} 
\Phi^*r:=\Psi^*(r\otimes\HAmf{1}_B\otimes\HAmf{1}_\HAGa) 
\HAeeq 
without fearing confusion. In other words: we identify~$r$ with its 
point-wise extension over 
\HAhb{P\times\HAGa}. 
\begin{proposition}\label{HApro-M.V} 
Let $\HAVg$ be a smooth wedge in~$\HAMg$ modeled by 
$[\Phi,W,g_W]$. Assume that one of the following conditions is satisfied: 
\setlength{\svitemindent}{1.3\parindent}
\begin{enumerate}  
\item[{\rm(i)}] 
${}$
\HAhb{R\in\HAcC(J)} or 
\HAhb{R\in\HAcF(J_\HAiy)}, 
\ \HAhb{h_P=g_P}, and 
\HAhb{r:=R} if 
\HAhb{R\in\HAcC(J)}, whereas 
\HAhb{r=\HAmf{1}} otherwise. 
\item[{\rm(ii)}] 
\begin{enumerate}  
\item[$(\HAal)$] 
${}$
\HAhb{J=J_0}, 
\ \HAhb{\HAal\in(-\HAiy,1]}, and 
\HAhb{R=\HAsR_\HAal}. 
\item[$(\HAba)$] 
${}$
\HAhb{\HAba\neq0} and satisfies 
\HAhb{\HAba\leq\HAal} with 
\HAhb{\HAba>0} if\/
\HAhb{\HAal>0}. 
\item[$(\HAga)$] 
${}$
\HAhb{h_P=\HAvp_\HAal^*(t^{2(\HAba-1)}\D t^2+g_B)}.  
\item[$(\HAda)$] 
${}$
\HAhb{r:=\HAsR_\HAal} if\/ 
\HAhb{0<\HAal\leq1} and 
\HAhb{r:=\HAmf{1}} otherwise. 
\end{enumerate}  
\end{enumerate} 
Then 
\HAhb{\rho:=\Phi^*r} is a cofinal singularity function for~$\HAVg$. 
\end{proposition} 
\begin{proof} 
Suppose $p$~is a cofinal singularity function for~$(P,h_P)$. Then 
\HAhb{p\otimes\HAmf{1}_\HAGa} is one for 
\HAhb{W=P\times\HAGa}, due to Theorem~\ref{HAthm-P.MM}. 

\par 
If (i)~is satisfied, then Lemma~\ref{HAlem-E.i} and 
Proposition~\ref{HApro-E.P} guarantee that 
\HAhb{\HAvp^*(R\otimes\HAmf{1}_B)} is a cofinal 
singularity function for~$(P,g_P)$. 

\par 
Let (ii)~apply. Then it follows from Example~\ref{HAexa-F.ex}(b) that 
\HAhb{\HAvp_\HAal^*(r\otimes\HAmf{1}_B)} is a cofinal singularity 
function for~$(P_\HAal,h_P)$. Now the considerations 
preceding the proposition imply the claims. 
\qed 
\end{proof} 
For the next lemma we recall definition~(\ref{HAI.dV}) 
where now $\HAcM$~is replaced by~$M$. 
\begin{lemma}\label{HAlem-M.V} 
Suppose 
\HAhb{R\in\HAcC(J_\HAiy)\cup\HAcF(J_\HAiy)}. Let $\HAVg$ be a 
smooth wedge in~$\HAMg$ modeled by 
$\bigl[\Phi,P(R,B),g_{P(R,B)}\bigr]$. If 
\HAhb{R\in\HAcC(J_\HAiy)}, then there exists 
a cofinal singularity function~$\rho$ for~$\HAVg$ satisfying 
\HAhb{\rho\sim R\circ\HAda_V}. If 
\HAhb{R\in\HAcF(J_\HAiy)}, then $\HAVg$~is 
cofinally uniformly regular. 
\end{lemma} 
\begin{proof} 
Suppose 
\HAhb{R\in\HAcC(J_\HAiy)}. Then $\Phi^*R$~is a cofinal singularity 
function for~$\HAVg$ by Proposition~\ref{HApro-M.V}(i). Since 
$\Phi$~is an isometry it follows 
\HAhb{\Phi^*R\sim R\circ\HAda_V}. This implies the assertion in the 
present case. If 
\HAhb{R\in\HAcF(J_\HAiy)}, then the claim follows also from the cited 
proposition.  
\qed 
\end{proof}
\begin{HAproofof1.2} 
The foregoing lemma shows that a tame end is cofinally uniformly regular. 
Let 
\HAhb{\{V_0,V_1,\ldots,V_k\}} be an open covering of~$\HAMg$ as 
in the definition preceding Theorem~\ref{HAthm-I.T}. Then we can 
shrink~$V_0$ slightly to~$\tilde V_0$ such that 
\HAhb{\{\tilde V_0,V_1,\ldots,V_k\}} is still an open covering 
and Lemma~\ref{HAlem-U.C} applies to guarantee that $\HAMg$~is 
uniformly regular on~$\tilde V_0$. Now the assertion follows 
from Theorem~\ref{HAthm-M.S}. 
\qed 
\end{HAproofof1.2} 
With the help of Theorem~\ref{HAthm-M.S} and Proposition~\ref{HApro-M.V} 
it is easy to construct uniformly regular Riemannian metrics in a 
great variety of geometric constellations. We leave this to the reader 
and proceed to study manifolds with smooth cuspidal singularities. 
For this we suppose: 
\HAbeq \label{HAM.Ga} 
\HAbal 
\\ 
\noalign{\vskip-7\jot}  
 \rm{(i)}  \qquad 
 &\HAcMg\text{ is an \hbox{$m$-dimensional} Riemannian manifold}\;.\\
 \noalign{\vskip-1\jot} 
 \rm{(ii)} \qquad  
 &(\HAGa,g_\HAGa)\text{ is a compact connected Riemannian submanifold}\\ 
 \noalign{\vskip-1\jot} 
 &\text{of $\HAcMg$ without boundary and codimension }\ell\geq1\,.\\ 
 \noalign{\vskip-1\jot} 
 \rm{(iii)}\qquad 
 &\HAGa\HAis\HApl\HAcM\text{ if }\HAGa\cap\HApl\HAcM\neq\HAes\;. 
\HAeal 
\HAeeq 
In the following, we use the notation preceding 
definition~(\ref{HAI.Ca}). First we assume 
\HAhb{\HAGa\HAis\HAci\HAcM}. Then there exists a uniform open tubular 
neighborhood~$\HAcU$ of~$\HAGa$ in~$\HAci\HAcM$ 
(e.g.\ M.W.~Hirsch~\cite{Hir94a} or A.A.~Kosinski~\cite{Kos93a}). More 
precisely, there exist 
\HAhb{\HAve\in(0,1)}, an open subset 
\HAhb{\HAcU=\HAcU_\HAve} of~$\HAcM$ with 
\HAhb{\HAcU\cap\HAGa=\HAGa}, and~a `tubular' diffeomorphism 
\HAhb{\utau\HAsco\HAcU\HAra\HABB^\ell\times\HAGa} such that 
\HAhb{\utau(\HAGa)=\{0\}\times\HAGa}, the 
tangential~$T\utau$ of~$\utau$ equals on~$T\HAGa$ the identity 
multiplied with the factor~$\HAve$, and 
\HAbeq \label{HAM.tg0} 
\utau_*g\sim g_{\HABB^\ell}+g_\HAGa\;. 
\HAeeq 

\par 
Let $T^\bot\HAGa$ be the normal bundle of~$\HAGa$. For 
\HAhb{\xi\in\HABS^{\ell-1}} and 
\HAhb{q\in\HAGa} there exists a unique 
\HAhb{\unu_\xi(q)\in T_q^\bot\HAGa} satisfying 
$$ 
(T_q\utau)\unu_\xi(q)=\bigl((0,\xi),q\bigr) 
   \in T_0\HABR^\ell\times\HAGa\;. 
$$ 
Let 
\HAhb{\HAga_{\unu,q}\HAsco[0,\HAve]\HAra\HAcM} be the geodesic emanating 
from~$q$ in direction 
\HAhb{\unu\in T_q^\bot\HAGa}. Then 
$$ 
p=p(t,\xi,q):=\utau^{-1}(t,\xi,q) 
=\HAga_{\HAve\unu_\xi(q),q}(t) 
\HAqa (t,\xi,q)\in[0,1)\times\HABS^{\ell-1}\times\HAGa\;. 
$$ 
From this we infer 
\HAbeq \label{HAM.td} 
t\sim\HAda_\HAcU\bigl(p(t,\xi,q),\HAGa\bigr) 
\HAqa (t,\xi,q)\in[0,1)\times\HABS^{\ell-1}\times\HAGa\;. 
\HAeeq 

\par 
Next we suppose 
\HAhb{\HAGa\HAis\HApl\HAcM}. Let 
\HAhb{\HAcU\HAhthdot=\HAcU\HAhthdot_{\kern-1.2ex\HAve}}  
be an open tubular neighborhood of~$\HAGa$ in~$\HApl\HAcM$ with 
associated tubular diffeomorphism 
\HAbeq \label{HAM.dt} 
\utau\kern.3ex\HAhthdot\HAsco\HAcU\HAhthdot\HAra\HABB^{\ell-1}\times\HAGa\;. 
\HAeeq 
Furthermore, there exists a uniform collar 
\HAhb{\HAcV=\HAcV_\HAve} for~$\HApl\HAcM$ over~$\HAcU\HAhthdot$. 
That is to say: 
by making~$\HAve$ smaller, if necessary, we can assume that 
$\HAcV$~is an open subset of~$\HAcM$ such that 
\HAhb{\HAcV\cap\HApl\HAcM=\HAcU\HAhthdot} and there exists a diffeomorphism 
\HAhb{\utau^+\HAsco\HAcV\HAra[0,1)\times\HAcU\HAhthdot} with 
\HAhb{\utau^+(\HAcU\HAhthdot)=\{0\}\times\HAcU\HAhthdot}, 
\ $T\utau^+$~equals the identity in $T_\HAGa\HApl\HAcM$ multiplied 
by~$\HAve$, and 
\HAbeq \label{HAM.tg+} 
\utau_*^+g\sim\D t^2+g_{\HApl\HAcM}\;. 
\HAeeq 
Note that 
\HAhb{\HABB_+^\ell\HAis[0,1)\times\HABB^{\ell-1}}. 
Hence it follows from (\ref{HAM.dt}) that there exists an open subset 
\HAhb{\HAcU=\HAcU_\HAve} of~$\HAcW$ such that 
\HAhb{\HAcU\cap\HApl\HAcM=\HAcU\HAhthdot} and 
\HAbeq \label{HAM.tau} 
\utau:=(\HAid_{[0,1)}\times\utau\kern.3ex\HAhthdot)\circ\utau^+ 
\HAsco\HAcU\HAra\HABB_+^\ell\times\HAGa 
\HAeeq  
is a diffeomorphism satisfying 
$$ 
\utau(\HAcU\HAhthdot) 
=\{0\}\times\HABB^{\ell-1}\times\HAGa 
\HAqb \utau(\HAGa)=\{0\}\times\HAGa\;. 
$$ 
By (\ref{HAM.tg0}) and (\ref{HAM.tg+}) we find 
$$ 
\utau_*g\sim\D t^2+g_{\HABB^{\ell-1}}\times g_\HAGa 
\sim g_{\HABB_+^\ell}\times g_\HAGa\;. 
$$ 
We let 
$\HAga\kern.3ex\HAhthdot_{\kern-.5ex\unu\kern.3ex\HAhthdot\kern-1pt,q}$ 
be the geodesic in~$\HApl\HAcM$ emanating from 
\HAhb{q\in\HAGa} in direction 
\HAhb{\unu\kern.3ex\HAhthdot\in T_{\HApl\HAcM}^\bot\HAGa}, 
where $T_{\HApl\HAcM}^\bot\HAGa$~is the orthogonal complement 
of~$T_q\HAGa$ in~$T_q\HApl\HAcM$. Suppose 
\HAhb{\xi=(s,\eta)} belongs to~$\HABS_+^{\ell-1}$ with 
\HAhb{s\in[0,1)} and 
\HAhb{\eta\in\HABR^{\ell-2}}, 
\ \HAhb{0\leq t\leq1}, and 
\HAhb{q\in\HAGa}. Define 
\HAhb{\unu\kern.3ex\HAhthdot_{\kern-.5ex\eta}(q)} in 
$T_{\HApl\HAcM,q}^\bot\HAGa$ by 
\HAhb{(T\HAhthdot_q\utau\kern.3ex\HAhthdot)
   \unu\kern.3ex\HAhthdot_{\kern-.5ex\eta}(q) 
   =\bigl((0,\eta),q\bigr)\in T_0\HABR^{\ell-2}\times\HAGa}, 
where 
$T\HAhthdot\utau\kern.3ex\HAhthdot$~is the tangential 
of~$\utau\kern.3ex\HAhthdot$ in~$\HApl\HAcM$. Set 
\HAhb{r=r(t,\eta,q) 
   :=\HAga\kern.3ex\HAhthdot
   _{\kern-.5ex\HAve\unu\kern.3ex\HAhthdot_{\kern-.5ex\eta}(q),q} 
   (t)\in\HAcU\HAhthdot}. 
Analogously, let $\umu_s(r)$ in 
$T_r^\bot\HApl\HAcM$ be given by 
\HAhb{(T_r\utau^+)\umu_s(r)=\bigl((0,s),r\bigr) 
   \in T_0\HABR\times\HAcU\HAhthdot}. Then 
$$ 
p=p(t,\xi,q):=\utau^{-1}(t\xi,q) 
=\HAga_{\HAve\umu_s(r(t,\eta,q)),r(t,\eta,q)}(t)\in\HAcU\;. 
$$ 
This means that we reach~$p$ from 
\HAhb{q\in\HAGa} in two steps. First we go from~$q$ to  
\HAhb{r\in\HAcU\HAhthdot} by following during the time interval~$[0,t]$ the 
geodesic in~$\HAcU\HAhthdot$ which emanates from~$q$ 
in direction~$\HAve\unu\kern.3ex\HAhthdot_{\kern-.7ex\eta}(q)$. 
Second, we follow during the time interval~$[0,t]$ the geodesic 
in~$\HAcU$ emanating from~$r$ in direction~$\HAve\umu_s(r)$ to 
arrive at~$p$. Observe 
$$ 
\HAdist_\HAcU(p,r)=\HAdist_\HAcU(p,\HAcU\HAhthdot) 
=\HAda_\HAcU(p,\HAcU\HAhthdot)\;. 
$$ 
Hence 
$$ 
t\sim\HAdist_\HAcU(p,r)\leq\HAdist_\HAcU(p,q) 
\leq\HAdist_\HAcU(p,r)+\HAdist_{\HAcU\HAhthdot}(r,q) 
\leq2t\;. 
$$ 
From this we infer 
\HAbeq \label{HAM.td2} 
t\sim\HAda_\HAcU\bigl(p(t,\xi,q),\HAGa\bigr) 
\HAqa (t,\xi,q)\in[0,1)\times\HABS^{\ell-1}\times\HAGa\;. 
\HAeeq 
Henceforth, 
\HAhb{\HABB:=\HABB^\ell} and 
\HAhb{\HABS:=\HABS^{\ell-1}} if 
\HAhb{\HAGa\in\HAci\HAcM}, whereas 
\HAhb{\HABB:=\HABB_+^\ell} and 
\HAhb{\HABS:=\HABS_+^{\ell-1}} otherwise. Then 
\HAhb{U=U_\HAGa:=\HAcU\HAssm\HAGa} is, in either case, 
a~\emph{tubular neighborhood of $\HAGa$ in}~$\HAMg$ and 
\HAhb{\utau=\utau\HAsn U\HAsco U\HAra\HAthBB\times\HAGa} is the 
(\emph{associated}) \emph{tubular diffeomorphism}, defined by 
(\ref{HAM.tau}) if 
\HAhb{\HAGa\in\HApl\HAcM}. By~$\HAda_\HAGa$ we denote the restriction of 
$\HAdist_\HAcU(\cdot,\HAGa)$ to~$U$. 

\par 
Let 
\HAhb{R\in\HAcC(J_0)} and 
\HAhb{\HAvp=\HAvp[R]}. With the (\hbox{$\ell$-dimensional}) polar 
coordinate diffeomorphism~$\upi$  the composition 
\HAbeq \label{HAM.UW} 
U\stackrel\utau\HAlora 
\HAthBB\times\HAGa 
\stackrel{\upi\times\HAid_\HAGa}\HAllllora 
(0,1)\times\HABS\times\HAGa 
\stackrel{\HAvp^{-1}\times\HAid_\HAGa}\HAlllllora 
W(R,\HABS,\HAGa) 
\HAeeq 
defines a diffeomorphism~$\Phi$ from~$U$ onto the model 
\hbox{$\HAGa$-wedge} 
\HAhb{W=W(R,\HABS,\HAGa)} over the spherical, 
resp.\ semi-spherical, \hbox{$R$-cusp} 
\HAhb{P=P(R,\HABS)}. We call~$U$ \emph{smooth singular end of~$\HAMg$ 
of type}~$(R,\HAGa)$ if $\Phi$~is an isometry from~$(U,g)$ 
onto~$(W,g_W)$, where 
\HAhb{h_P:=g_P}. 
\begin{lemma}\label{HAlem-M.R} 
Let $U$ be a smooth singular end of~$\HAMg$ of type~$(R,\HAGa)$. 
Then there exists a cofinal singularity function~$\rho$ for 
$(U,g)$ satisfying 
\HAhb{\rho\sim R\circ\HAda_\HAGa}. 
\end{lemma}
\begin{proof} 
It is a consequence of Proposition~\ref{HApro-E.P}, 
Lemma~\ref{HAlem-E.i}, and Lemma~\ref{HAlem-P.f} that 
\HAhb{\rho:=\Phi^*R} is a cofinal singularity function for $(U,g)$. From 
(\ref{HAM.UW}) and (\ref{HAM.ps}) we deduce 
\HAhb{\Psi=(\upi\otimes\HAid_\HAGa)\circ\utau}. Moreover, 
\HAhb{\Psi\bigl(p(t,\xi,q)\bigr)=(t,\xi,q)} for 
$(t,\xi,q)$ belonging to 
\HAhb{(0,1)\times\HABS\times\HAGa}. Hence 
\HAhb{(\Phi^*R)\bigl(p(t,\xi,q)\bigr)=R(t)} by~(\ref{HAM.ph}). 
Now the claim is implied by (\ref{HAM.td}), respectively (\ref{HAM.td2}). 
\qed 
\end{proof} 
It is clear that the assertion of this lemma is independent of the 
particular choice of~$U$, that is, of~$\HAve$. 
\begin{HAproofof1.6,1.8} 
The statements follow directly from Lemma~\ref{HAlem-M.R} with 
\HAhb{R=\HAsR_\HAal}, Lemma~\ref{HAlem-M.V}, and 
Theorem~\ref{HAthm-M.S}. 
\qed 
\end{HAproofof1.6,1.8} 
\begin{HAproofof1.9} 
We set 
\HAhb{R:=\HAsR_\HAal} if 
\HAhb{0<\HAal\leq1}, and 
\HAhb{R:=\HAsR_{-\HAal}} for 
\HAhb{\HAal>1}. It follows from Example~\ref{HAexa-F.ex}(b) (setting 
\HAhb{\HAba:=\HAal} if 
\HAhb{\HAal\leq1} and 
\HAhb{\HAba:=-\HAal} otherwise) and Theorem~\ref{HAthm-P.MM} 
that 
$$ 
g_W:=\HAvp^*\bigl(t^{-2\HAal}(t^{2(\HAal-1)}\D t^2+g_\HABS)\bigr) 
+t^{-2\HAal}g_\HAGa 
=\HAvp^*(t^{-2}\D t^2+t^{-2\HAal}g_\HABS)+t^{-2\HAal}g_\HAGa 
$$ 
is a cofinally uniformly regular metric for 
\HAhb{W=W(R,\HABS,\HAGa)} if 
\HAhb{0<\HAal\leq1}, whereas 
$$ 
g_W:=\HAvp^*(t^{-2(\HAal+1)}\D t^2+g_\HABS)+g_\HAGa 
$$ 
is one if 
\HAhb{\HAal>1}. Thus~$\Phi$, defined by (\ref{HAM.UW}), is an 
isometry from $(U,\HAgsg)$ onto $(W,g_W)$. Hence the claim follows 
once more from Lemma~\ref{HAlem-P.f}. 
\qed 
\end{HAproofof1.9} 
Lastly, we mention that there occur interesting and important singular 
manifolds if assumption (\ref{HAM.Ga})(iii) is dropped, that is, if 
$\HAGa$~intersects~$\HAci\HAcM$ as well as~$\HApl\HAcM$. 
Then $\HAGa$~is no longer a smooth singular end but has cuspidal 
corners, for example. Such cases are not considered here although the 
technical means for their study have been provided in the preceding 
sections. 

\def\cprime{$'$} \def\polhk#1{\setbox0=\hbox{#1}{\ooalign{\hidewidth
  \lower1.5ex\hbox{`}\hidewidth\crcr\unhbox0}}}

\end{document}